\documentclass[11pt]{article}
\usepackage{fullpage}
\usepackage{amsmath,amssymb,amsthm,amsfonts}
\usepackage{mathtools,booktabs,microtype}
\usepackage{hyperref}
\hypersetup{hidelinks}
\allowdisplaybreaks
\numberwithin{equation}{section}
\theoremstyle{plain}
\newtheorem{theorem}{Theorem}[section]
\newtheorem{proposition}[theorem]{Proposition}
\newtheorem{lemma}[theorem]{Lemma}

\newtheorem{problem}[theorem]{Open Problem}
\theoremstyle{definition}

\theoremstyle{remark}
\newtheorem{remark}[theorem]{Remark}
\newcommand{\C}{\mathbb{C}}
\newcommand{\Q}{\mathbb{Q}}
\newcommand{\A}{\mathbb{A}}
\newcommand{\id}{\mathrm{id}}
\newcommand{\Miss}{\mathcal M}

\title{Counterexamples to the Jacobian conjecture\\ in dimensions greater than two}
\author{Shuhong Gao\thanks{AI disclosure: The main idea and framework are
due to the author, and Claude Fable 5 assisted in the proofs and in the writing
of the paper. Codex assisted with the subsequent mathematical review,
exact verification, and revision, including the equivalence and obstruction results.}\\[4pt]
\normalsize School of Mathematical and Statistical Sciences\\
Clemson University, Clemson, SC 29634-0975\\
Email: \texttt{sgao@clemson.edu}}
\date{July 30, 2026; revised November 4, 2026}

\begin{document}
\maketitle

\begin{abstract}
The Jacobian conjecture, open since 1939, asks whether every polynomial map of
$\C^n$ whose Jacobian determinant is a nonzero constant must have a polynomial
inverse. It was refuted in dimension three by Alp\"oge on July 19, 2026, with
an infinite family by Gallagher (July 20) and a geometric explanation by
Speyer (July 23): the counterexample sweeps the tangent lines of a plane
curve. We give a self-contained account of this tangent-sweep mechanism and
generalize it from plane curves to direction fields on hypersurfaces. The
resulting construction, including stabilization of the three-dimensional
family, produces counterexamples in every dimension greater
than two and, in each dimension, of arbitrarily large geometric degree (the
number of preimages of a typical point). We work it out in five explicit
maps: one three-dimensional of degree four, two four-dimensional of degrees
five and ten, and two five-dimensional of degrees six and twelve. We show
that our explicit examples in dimensions $n\ge4$ are not polynomially
left-right equivalent to the counterexamples of Alp\"oge and Gallagher
when these are embedded in $\C^n$ by adjoining identity coordinates. The
counterexamples provide explicit examples of \'etale morphisms
$\C^n\to\C^n$ that are not proper: they are everywhere unramified, and fail
to be injective only through points escaping to infinity. All identities were
verified in exact rational arithmetic. Part~\ref{part:equivalence} proves the
higher-dimensional inequivalence results, and Appendix~\ref{app:fibers}
determines exact fiber structures for the three-dimensional examples. A complete classification of the fibers of
$F_4,F_5,F_6,F_7$ remains open.
\end{abstract}

\clearpage
\tableofcontents
\clearpage

\section{Introduction}

Let $F\colon\C^n\to\C^n$ be a polynomial map. Throughout this paper we use the
following conventions. The value $F(x)$ is written as a \emph{column} vector of
the $n$ components of $F$; consequently each partial derivative
$\partial_iF:=\partial F/\partial x_i$ is again a column vector, and the Jacobian
matrix of $F$ is
\[
J(F)\;=\;\bigl(\partial_1F,\;\partial_2F,\;\dots,\;\partial_nF\bigr),
\]
the square matrix whose columns are the partial derivatives, with the variables
taken in the order in which they are displayed. More generally, for column
vectors $V_1,\dots,V_k$ of equal length, $(V_1,\dots,V_k)$ denotes the matrix
with these columns; adjoining a column $\Delta$ to a matrix $J$ is written
$(\Delta,J)$. If $F$ has a polynomial inverse, the chain rule forces $\det J(F)$
to be a nonzero constant. The \emph{Jacobian conjecture}, raised by Keller
\cite{Keller1939}, asserts the converse: a polynomial map with
$\det J(F)\in\C^\times$ (a \emph{Keller map}) is a polynomial automorphism. The
problem became widely known through the survey of Bass, Connell and Wright
\cite{BCW1982} and the book of van den Essen \cite{Essen2000}.

On July 19, 2026, the conjecture was refuted in dimension three by an explicit
counterexample announced by L.~Alp\"oge \cite{Alpoge2026}. The next day,
Gallagher \cite{Gallagher2026} gave a uniform one-variable construction
producing, for every integer $d\ge3$, an explicit counterexample of geometric
degree $d$; expositions by Tao (July 21) \cite{Tao2026} and Speyer
(July 23) \cite{Speyer2026} followed. The first
counterexample is a polynomial map $F\colon\C^3\to\C^3$ with
$\det J(F)\equiv-2$ that identifies three distinct points. Its mechanism, isolated
and named the \emph{tangent sweep} by Speyer \cite{Speyer2026}, is strikingly
classical: one sweeps the tangent lines of a plane curve --- a map that
is unavoidably many-to-one, by projective duality --- and then conjugates by
monomial maps so that the Jacobian factor picked up by the sweep is cancelled
exactly. The price of the cancellation is that the ramification of the sweep is
pushed to infinity: the resulting Keller map is everywhere unramified but not
proper, and injectivity fails through escape to infinity, a loophole the Jacobian
hypothesis never excluded.

The main contributions of this paper are a geometric construction framework,
explicit higher-dimensional examples with proved inequivalence to the earlier
three-dimensional constructions, and an analysis of the resulting \'etale
morphisms that are not proper.

First, we formulate tangent sweeps and their monomial twists in arbitrary
dimension. A determinant pencil organizes the normalization equations for
prescribed direction fields on parametrized hypersurfaces, and a stage method
organizes the divisibility conditions needed to obtain polynomial maps with
nonzero constant Jacobian. This framework incorporates the one-variable
construction of Gallagher \cite{Gallagher2026}, which supplies a
three-dimensional map $H_d:\C^3\to\C^3$ of every geometric degree $d\ge3$.
Each $H_d$ extends to every dimension $n\ge4$ by adjoining identity
coordinates. Explicitly, writing $H_d=(H_{d,1},H_{d,2},H_{d,3})$, define
\[
\begin{split}
 H_{n,d}:\C^n&\longrightarrow\C^n,\\
 (x_1,\ldots,x_n)&\longmapsto
 \bigl(H_{d,1}(x_1,x_2,x_3),H_{d,2}(x_1,x_2,x_3),
 H_{d,3}(x_1,x_2,x_3),x_4,\ldots,x_n\bigr).
\end{split}
\]
We call this extension a \emph{stabilization} of $H_d$; in product notation,
$H_{n,d}=H_d\times\id_{\A^{n-3}}$. Set $H_{3,d}=H_d$.
These maps give counterexamples in every dimension $n\ge3$ and of every
geometric degree $d\ge3$: adjoining identity coordinates preserves both the
Jacobian determinant and the geometric degree. Here \emph{geometric degree} means the degree of the
induced function-field extension, equivalently the number of points in a
general fiber over $\C$; it is distinct from the total degrees of the
coordinate polynomials. The all-dimension, all-degree existence statement
thus incorporates the earlier family. Our general hypersurface construction
also provides a method for seeking examples beyond these products, subject
to solving the normalization and polynomiality conditions. We solve these
conditions for the explicit examples below; their solvability for arbitrary
free data is not assumed.

Second, we prove that several of the constructed examples cannot be obtained
from any lower-dimensional map by adjoining identity coordinates and making
independent polynomial changes of source and target coordinates. The maps
$F_4,F_5,F_7$ satisfy this exclusion even after further stabilization. In
particular, they are not equivalent to any stabilization of the earlier
three-dimensional counterexamples, including examples of the same geometric
degree. The map $F_6$ admits an explicit splitting into a four-dimensional
factor $\Psi$ and one identity coordinate, and $\Psi$ satisfies the same
stable exclusion. Omitted targets and intrinsic strata of a canonical finite
cover provide the obstructions. Theorem~\ref{thm:product-summary} states the
precise conclusions, proved in Part~\ref{part:equivalence}.

Third, the construction gives explicit \'etale morphisms
$\A^n_{\C}\to\A^n_{\C}$ that are not proper. The nonzero constant Jacobian
makes each morphism \'etale. Its geometric degree is greater than one, whereas
a proper \'etale self-map of complex affine space would be a connected finite
covering of the simply connected space $\C^n$, hence have degree one.
The sweep and twist describe the failure of properness concretely: the
ramification of the finite sweep is removed from the affine source, and
inverse branches instead escape to infinity. We determine every fiber size
for two three-dimensional examples and identify the loci needed for the
higher-dimensional equivalence obstructions.

Part~\ref{part:construction} develops the construction framework:
Section~\ref{sec:background} gives the background,
Section~\ref{sec:sweep} treats plane tangent sweeps and the
three-dimensional examples $F,G$, and Section~\ref{sec:dim4} develops
the hypersurface and stage methods that produce $F_4,F_5,F_6,F_7$,
of geometric degrees $5,10,6,12$.

Part~\ref{part:equivalence} develops the general equivalence obstructions
in Section~\ref{prod:sec:theory}, identifies both $F$ and $G$ with members
of Gallagher's general seed family in Section~\ref{prod:sec:G}, and treats $F_4,F_5,F_6,F_7$
in Sections~\ref{prod:sec:F4-omitted}--\ref{prod:sec:F7}.
Section~\ref{sec:conclusion} collects the conclusions and further questions.
Appendices~\ref{app:fibers} and~\ref{app:verification} contain the detailed
fiber calculations and verification-file descriptions, respectively.
The main product-structure results are the following.

Write $H\sim K$ if $H=\alpha\circ K\circ\beta$ for polynomial
automorphisms $\alpha,\beta$ of the target and source. This is polynomial
\emph{left-right equivalence}. Adjoining $r$ identity coordinates means
replacing $H$ by $H\times\id_{\A^r}$.
\begin{theorem}[Product structure]\label{thm:product-summary}
For each $H\in\{F_4,F_5,F_7\}$, with source dimension $n$, one has
\[
 H\times\id_{\A^r}\nsim\Phi\times\id_{\A^{r+1}}
 \quad\text{for every }r\ge0
 \text{ and every polynomial }\Phi:\A^{n-1}\to\A^{n-1}.
\]
There are explicit polynomial automorphisms $\alpha,\beta$ such that
\[
 \alpha\circ F_6\circ\beta^{-1}=\Psi\times\id_{\A^1},
 \qquad \det J(\Psi)=-54,\qquad \deg_{\rm geom}\Psi=6.
\]
The same stable exclusion holds for $H=\Psi$ and $n=4$.
Consequently $F_4,F_5,\Psi$ are pairwise inequivalent, including after
common stabilization, and none is a stabilization of a three-dimensional map.
\end{theorem}

The assertions for $F_4,F_5,F_7$ follow respectively from
Theorems~\ref{prod:thm:main}, \ref{prod:thm:F5-complete}, and
\ref{prod:thm:F7-complete}. The splitting and the assertion for $\Psi$
are Theorem~\ref{prod:thm:F6-split}. The three four-dimensional maps
have distinct geometric degrees $5,10,6$, which are preserved by
left-right equivalence and by adjoining identity coordinates.

These exclusions include same-degree members of earlier three-dimensional
families. Both $F$ and $G$ belong to Gallagher's general seed family
(Section~\ref{prod:sec:G}), but their geometric degrees $3$ and $4$
distinguish them from each other. Their complete fiber descriptions
provide the three-dimensional geometric models.

\clearpage
\part{Construction of counterexamples}\label{part:construction}

\section{Background and previous work}\label{sec:background}

We recall what was known before 2026; see \cite{Essen2000} for a comprehensive
treatment. Wang \cite{Wang1980} proved the conjecture for maps of degree at most
$2$ in any dimension. Bass--Connell--Wright \cite{BCW1982} and Yagzhev
\cite{Yagzhev1980} showed that it suffices to prove the conjecture for maps of the
form $x+H(x)$ with $H$ cubic homogeneous (in all dimensions simultaneously), and
Dru\.zkowski \cite{Druzkowski1983} reduced further to the case where each
component of $H$ is the cube of a linear form. In dimension two, Moh
\cite{Moh1983} verified the conjecture for maps of degree at most $100$; the
two-dimensional case remains open and is untouched by the counterexamples
discussed here, which exist only in dimension $\ge3$ (by Wang's theorem, degree
$2$ examples are impossible, and the known constructions produce degree $\ge3$).

Two structural facts frame the counterexample. A theorem of
Bia{\l}ynicki-Birula and Rosenlicht \cite{BBR1962} states that an injective
polynomial map $\C^n\to\C^n$ is bijective with polynomial inverse; hence a Keller
counterexample must fail injectivity. Since a Keller map is \'etale,
it is a local biholomorphism everywhere, so non-injectivity can only occur through
failure of \emph{properness}: a proper Keller map would be a finite
unramified analytic covering of the simply connected space $\C^n$, hence
would have one sheet since its source is connected. Locally, distinct
preimages cannot merge in the source; a drop in fiber size occurs when
inverse branches escape to infinity. The locus of points over which a
polynomial map is not proper was studied by Jelonek \cite{Jelonek1993}, who showed
it is either empty or a hypersurface. In the real category, Pinchuk
\cite{Pinchuk1994} had already constructed a non-injective polynomial map of
$\mathbb{R}^2$ with everywhere nonvanishing (nonconstant) Jacobian, an early
warning that unramifiedness does not force injectivity. The counterexamples below
realize this scenario over $\C$ with \emph{constant} Jacobian: they are \'etale of
geometric degree $d>1$ onto their (open, dense) image, with the fiber count
dropping only across the Jelonek hypersurface.

\section{The tangent sweep of a plane curve}\label{sec:sweep}

In this section we describe the main idea behind Alp\"oge's counterexample and
Gallagher's family \cite{Gallagher2026}, following the \emph{tangent sweep}
picture due to Speyer \cite{Speyer2026}.

\subsection{The sweep and its Jacobian}
Let $p,q\in\C[w]$ with $p$ nonconstant, and let
$\mathcal{K}\colon w\mapsto(p(w),q(w))^{\mathsf T}$ be a parametrized plane curve.
We choose the special class of parametrizations satisfying
$q'(w)=\tfrac{w}{2}p'(w)$, with the integration constant fixed below.
For $d=\deg p\ge1$ this gives $\deg q=d+1$. This is a choice of curve
and direction field, not a normalization available for every polynomial
plane curve. In particular the slope $q'/p'=w/2$ grows linearly.

The \emph{tangent sweep} of $\mathcal{K}$ is the map
sending (point of $\mathcal K$, position along its tangent line) to the plane:
\[
T_0(\gamma,w)=\bigl(p(w)+\gamma p'(w),\;q(w)+\gamma q'(w)\bigr)^{\mathsf T},
\qquad
\det J(T_0)=\gamma\,(p'q''-p''q') ,
\]
the variables being ordered $(\gamma,w)$.
The factor $p'q''-p''q'$ vanishes at inflection points. Normalizing the direction
field to $\delta=(2,w)^{\mathsf T}$, which is tangent exactly when $q'=\tfrac{w}{2}p'$, i.e.
\begin{equation}\label{eq:normalization}
q(w)=\int_0^w \tfrac{s}{2}\,p'(s)\,ds ,
\end{equation}
gives the inflection-free normal form
\[
S(\gamma,w)=\bigl(p(w)+2\gamma,\;q(w)+\gamma w\bigr)^{\mathsf T},\qquad \det J(S)=2\gamma .
\]
Thus the sweep is unramified except over $\gamma=0$, which maps onto the curve
itself, and its Jacobian is a \emph{coordinate}.

\subsection{Fibers and duality}\label{subsec:fibdual}

Fix a target point $(X,Y)\in\C^2$ and solve $S(\gamma,w)=(X,Y)$. The first
coordinate determines the position along the tangent line,
$\gamma=\tfrac12(X-p(w))$; substituting into the second coordinate leaves a
single equation in $w$, the \emph{tangency polynomial}
\begin{equation}\label{eq:tangency}
W_{X,Y}(w)\;=\;q(w)+\tfrac{w}{2}\bigl(X-p(w)\bigr)-Y .
\end{equation}
Its meaning is geometric: the tangent line of $\mathcal K$ at the parameter
value $w$ is the line $\{\mathcal K(w)+\gamma\,(2,w)^{\mathsf T}:\gamma\in\C\}$, and
$W_{X,Y}(w)=0$ says exactly that this line passes through $(X,Y)$. A root
$w_0$ of $W_{X,Y}$ is thus a \emph{parameter of tangency}: it names the point
$\mathcal K(w_0)$ of the curve at which some tangent line through $(X,Y)$
touches $\mathcal K$. (Picture a point outside a parabola: two tangent lines
pass through it, touching the parabola at two points; the two parameter values
naming those points of contact are the two roots of $W_{X,Y}$.)

\begin{lemma}\label{lem:fibroots}
For every $(X,Y)\in\C^2$ the assignment
$w_0\mapsto\bigl(\tfrac12(X-p(w_0)),\,w_0\bigr)$ is a bijection from the set
of distinct roots of $W_{X,Y}$ onto the fiber $S^{-1}(X,Y)$. If $\deg p=d$
then $\deg_wW_{X,Y}=d+1$ with leading coefficient
$-\mathrm{lc}(p)/(2(d+1))$, independent of $(X,Y)$: every fiber of $S$ is
finite of cardinality at most $d+1$, and $S$ is generically $(d+1)$-to-one.
Moreover the normalization \eqref{eq:normalization} gives the exact derivative
identity
\begin{equation}\label{eq:Wprime}
W_{X,Y}'(w)\;=\;\tfrac12\bigl(X-p(w)\bigr)\;=\;\gamma ,
\qquad
W_{X,Y}''(w)\;=\;-\tfrac12\,p'(w) .
\end{equation}
\end{lemma}
\begin{proof}
The first sweep equation gives the displayed value of $\gamma$, and the
second becomes $W_{X,Y}=0$, proving the bijection. If the leading term
of $p$ is $aw^d$, then that of $q$ is $adw^{d+1}/(2(d+1))$;
subtraction gives the asserted nonzero leading coefficient of $W$.
Differentiating and using $q'=wp'/2$ proves both identities.
For a fixed $X$, only the finitely many roots of $p(w)=X$ can be
multiple roots of $W$; choosing $Y$ outside their corresponding finite
set of values makes $W$ squarefree. Thus a general target has $d+1$
distinct preimages.
\end{proof}

The identity \eqref{eq:Wprime} --- the derivative of the tangency polynomial
\emph{is} the sweep multiplier --- organizes all fiber degenerations, and is
used repeatedly later:
\begin{itemize}
\item \emph{Multiple roots lie on the curve.} $w_0$ is a multiple root iff
$W=W'=0$ there, i.e.\ iff $\gamma=0$ and $(X,Y)=\mathcal K(w_0)$: the
discriminant locus of $W_{X,Y}$ in the $(X,Y)$-plane is the swept curve
itself (the envelope of its tangent lines). Concretely, for the cuspidal cubic
of \S\ref{sec:first}, where $W_{X,Y}$ is proportional to $w^3-2w^2+Xw-2Y$,
one has the exact identity
$\operatorname{disc}_w\bigl(w^3-2w^2+Xw-2Y\bigr)=-4\,E(X,Y)$ with $E$ the
implicit equation of the cubic --- the sweep-level source of the discriminant
factorizations in Theorem~\ref{thm:strat3} and Appendix~\ref{app:fibers}.
\item \emph{Contact bookkeeping.} A point of $\mathcal K$ sees its own
tangent doubly: at a smooth point $\mathcal K(w_0)$ the root $w_0$ of
$W_{\mathcal K(w_0)}$ has multiplicity exactly $2$ (by \eqref{eq:Wprime},
$W''(w_0)=-\tfrac12p'(w_0)\neq0$). At a cusp ($p'(w_0)=0$, hence also
$q'(w_0)=0$) the multiplicity is $\ge3$: the cusp sees its cuspidal tangent
triply. At a node, \emph{two} parameters $w_1\neq w_2$ are each double roots.
\item \emph{The multiplier detects the point of contact.} A root has
$\gamma=0$ iff the target $(X,Y)$ is the point of tangency itself; these are
exactly the solutions lying on the ramification locus of $S$.
\end{itemize}

Counting distinct roots gives the fiber sizes of the sweep, and --- after the
twist of \S\ref{subsec:twist}, which deletes precisely the solutions with
$\gamma=0$ --- of the Keller maps built from it:

\begin{proposition}[fiber sizes]\label{prop:fibsizes}
Let $\deg p=d$, and let $m_1,\ldots,m_r$ be the multiplicities of the
distinct roots of $W_{X,Y}$. Then $\sum_i m_i=d+1$ and
\[
 \#S^{-1}(X,Y)=r=d+1-\sum_i(m_i-1).
\]
On the target chart where the monomial twist is invertible, the associated
Keller map has fiber size
\[
 \#\{i:m_i=1\}=d+1-\sum_{m_i\ge2}m_i.
\]
In particular, when the only multiple roots are respectively one double
root, one triple root, or two double roots, the sweep fiber sizes are
$d,d-1,d-1$, and the twisted fiber sizes are $d-1,d-2,d-3$.
These are the ordinary smooth-contact, cusp, and node cases in the two
plane examples below. Higher contact multiplicities must be treated by
the general formulas.
\end{proposition}
\begin{proof}
Lemma~\ref{lem:fibroots} identifies sweep points with distinct roots.
Equation~\eqref{eq:Wprime} identifies $\gamma=0$ precisely with the
multiple roots. On the invertible chart of the twist, exactly the roots
with $\gamma\ne0$ survive, giving the second formula.
\end{proof}

The degree $d+1$ also counts the tangent lines through a general target:
their slopes $w/2$ distinguish their parameters, so this is the class
(the degree of the dual curve) of the swept curve. For a projective plane
curve of degree $D$ whose only singularities are ordinary nodes and cusps,
the Pl\"ucker formula is $D(D-1)-2\delta-3\kappa$. We use it below only
for the cubic and quartic with these specified singularities. The constant
leading coefficient of $W$ proves directly that its degree is $d+1$ at
every affine target. The boundary of the twist requires a separate analysis.

\subsection{The monomial twist}\label{subsec:twist}
The sweep is not a Keller map ($\det J(S)=2\gamma$), but the factor $\gamma$ can be
cancelled by conjugation with monomial maps. Pad the sweep to three variables and
compose:
\begin{multline*}
\begin{pmatrix}x\\ y\\ z\end{pmatrix}
\;\xrightarrow{\ \text{monomial}\ }\;
\begin{pmatrix}x\\ xy\\ x^2z\end{pmatrix}
\;\xrightarrow{\ \text{affine}\ }\;
\begin{pmatrix}x\\ \gamma\\ u\end{pmatrix}
\;\xrightarrow{\ w=\gamma u\ }\;
\begin{pmatrix}x\\ \gamma\\ w\end{pmatrix}\\
\xrightarrow{\ \text{padded sweep}\ }\;
\begin{pmatrix}\gamma x\\ P\\ Q\end{pmatrix}
\;\xrightarrow{\ \text{twist}\ }\;
\begin{pmatrix}C\\ P/C\\ Q/C^2\end{pmatrix},
\end{multline*}
where $\gamma=\gamma_0+a\,xy+b\,x^2z$, $u=1+xy$, $P=p(w)+2\gamma$,
$Q=q(w)+\gamma w$ and $C=\gamma x$. The Jacobians multiply to
\[
C^{-3}\cdot 2\gamma^2\cdot\gamma\cdot x^3\cdot(-b)\;=\;-2b,
\]
a nonzero constant. (Permuting the components of the final map changes the sign
of the determinant by the sign of the permutation; this accounts for
$\det J(F)=-2$ below, where the traditional component order reverses the one
displayed here.) The twist requires $C\mid P$ and $C^2\mid Q$; these
divisibilities give finitely many side conditions. With the affine stage
fixed, they are linear in the coefficients of $p$; jointly in $p$ and
the stage parameters they need not be linear. These are the side conditions of
\cite{Speyer2026}; \S\ref{subsec:stage} systematizes these conditions in every
dimension. Because the
ramification locus $\gamma=0$ of the sweep is exactly where the twisted coordinate
$C=\gamma x$ degenerates, the composite is everywhere unramified; the
$d+1$ sheets of the sweep survive, but sheets that would have merged over the
curve now escape to infinity. \emph{The construction converts ramification into
non-properness.} Varying $d=\deg p$ produces Keller counterexamples of geometric
degree $d+1$ for every $d\ge2$ --- this is Gallagher's family
\cite{Gallagher2026} --- as established in that construction.

\subsection{The first counterexample and its fiber structure}\label{sec:first}

We first run the scheme of \S\ref{subsec:twist} with $\deg p=2$, which recovers
the counterexample of Alp\"oge \cite{Alpoge2026}, which is also the
geometric-degree-three case of Gallagher's general seed construction
(Section~\ref{prod:sec:G}). The solution of the side
conditions in closed form is
\[
p(w)=-3w^2+4w,\qquad
q(w)=-w^3+w^2,\qquad
\gamma=2-3xy-x^2z,\quad u=1+xy,\quad w=\gamma u
\]
($q$ is \eqref{eq:normalization} for this $p$), and the map, in Alp\"oge's
original component order,
\[
F\;=\;\Bigl(\frac{q(w)+\gamma w}{(\gamma x)^2},\;\;
\frac{p(w)+2\gamma}{\gamma x},\;\;\gamma x\Bigr)^{\mathsf T}\colon\ \C^3\longrightarrow\C^3
\]
(the sweep order of \S\ref{subsec:twist} lists the components in reverse; the
reversal is a transposition of the outer components and multiplies the Jacobian
determinant by $-1$).

\begin{theorem}\label{thm:alpoge}
$F$ is a polynomial map with components of degrees $7,6,4$; its Jacobian
determinant is identically $-2$; and its generic fiber consists of exactly
$3$ points.
\end{theorem}

A few comments on the ingredients. In the sweep order the determinant is
$-2b=2$ with $b=-1$, and the component reversal contributes the sign $-1$.
The divisibilities required by the twist are explicit:
$p(w)+2\gamma=\gamma\bigl(-3\gamma u^2+4u+2\bigr)$ and
$q(w)+\gamma w=\gamma^2\bigl(-\gamma u^3+u^2+u\bigr)$, and both bracketed
factors vanish at $x=0$. The tangency
polynomial \eqref{eq:tangency} is, up to a constant factor, the cubic
\[
w^3-2w^2+Xw-2Y ,
\]
whose leading coefficient is constant, so every target point admits three
tangency parameters counted with multiplicity.

Written out in full --- matching the form in \cite{Speyer2026,Tao2026} --- the
map is $F=(f_1,f_2,f_3)^{\mathsf T}$ with $f_3=\gamma x$,
$f_2f_3=p(w)+2\gamma$, $f_1f_3^{\,2}=q(w)+\gamma w$:
\begin{align*}
f_1 &= (1+xy)^3 z + y^2(1+xy)(4+3xy),\\
f_2 &= y + 3x(1+xy)^2 z + 3xy^2(4+3xy),\\
f_3 &= 2x - 3x^2y - x^3 z ,
\end{align*}
and
$F(0,0,-\tfrac14)=F(1,-\tfrac32,\tfrac{13}{2})=F(-1,\tfrac32,\tfrac{13}{2})
=(-\tfrac14,0,0)$: three distinct points with a common image. The underlying
curve $\mathcal{K}=(p,q)^{\mathsf T}$ is a rational \emph{cuspidal cubic} with one cusp, at
$(p,q)(\tfrac23)=(\tfrac43,\tfrac4{27})$, and no node; the Pl\"ucker formula
gives class $m=3\cdot2-3\cdot1=3$, matching the fiber count.

The fibers of $F$ have sizes $3$, $1$, or $0$. Their complete
stratification, including the omitted curve and the non-properness set,
is given in Theorem~\ref{thm:strat3} in Appendix~\ref{app:F}.

\subsection{A three-dimensional example of geometric degree four}\label{sec:deg4}

We now run the scheme of \S\ref{subsec:twist} with $\deg p=3$. Gallagher's
family \cite{Gallagher2026} contains examples of every geometric degree
$n\ge3$, including a degree-four instance; the example below was obtained
by solving the side conditions directly, and is recorded here with
its verification data and complete fiber description. It belongs to
Gallagher's general seed family: equation~\eqref{prod:eq:G-equivalence}
gives the explicit linear equivalence. The solution in closed form is:
\[
p(w)=w^3-6w^2+6w,\qquad
q(w)=\tfrac38 w^4-2w^3+\tfrac32 w^2,\qquad
\gamma=2-4xy-x^2z,\quad u=1+xy,\quad w=\gamma u ,
\]
($q$ is \eqref{eq:normalization} for this $p$), and the map
\[
G\;=\;\Bigl(\gamma x,\;\;\frac{p(w)+2\gamma}{\gamma x},\;\;
\frac{q(w)+\gamma w}{(\gamma x)^2}\Bigr)^{\mathsf T}\colon\ \C^3\longrightarrow\C^3 .
\]

\begin{theorem}\label{thm:deg4}
$G$ is a polynomial map with components of degrees $4,11,12$; its Jacobian
determinant is identically $2$; and its generic fiber consists of exactly
$4$ points.
\end{theorem}

Here $b=-1$ and the components are in the displayed order, so
$\det J(G)=-2b=2$. The divisibilities are again explicit:
$p(w)+2\gamma=\gamma\bigl(6u-6\gamma u^2+\gamma^2u^3+2\bigr)$
and $q(w)+\gamma w=\gamma^2\bigl(\tfrac38\gamma^2u^4-2\gamma u^3+\tfrac32
u^2+u\bigr)$, and both bracketed factors vanish at $x=0$. The tangency
polynomial \eqref{eq:tangency} is the quartic
\[
\tfrac14 w^4-2w^3+3w^2-Xw+2Y ,
\]
whose leading coefficient is constant, so every target point admits four tangency
parameters counted with multiplicity.

Written out in full, $G=(g_1,g_2,g_3)^{\mathsf T}$ with
{\small
\begin{align*}
g_1 ={}& -x^3z-4x^2y+2x\\[2pt]
g_2 ={}& x^6y^3z^2+8x^5y^4z+3x^5y^2z^2+16x^4y^5+20x^4y^3z+3x^4yz^2+32x^3y^4+18x^3y^2z+x^3z^2\\
&+28x^2y^3+8x^2yz+16xy^2+2xz+2y\\[2pt]
g_3 ={}& \tfrac{3}{8}x^6y^4z^2+3x^5y^5z+\tfrac{3}{2}x^5y^3z^2+6x^4y^6+\tfrac{21}{2}x^4y^4z\\
&+\tfrac{9}{4}x^4y^2z^2+18x^3y^5+14x^3y^3z+\tfrac{3}{2}x^3yz^2+\tfrac{43}{2}x^2y^4\\
&+9x^2y^2z+\tfrac{3}{8}x^2z^2+14xy^3+3xyz+\tfrac{9}{2}y^2+\tfrac{1}{2}z
\end{align*}
}
The Jacobian and polynomiality follow by exact expansion; the generic
fiber count follows from Lemma~\ref{lem:fibroots} and the invertible twist chart. The underlying curve $(p,q)^{\mathsf T}$ is a rational quartic with
\emph{two} cusps, at $w=2\pm\sqrt2$ (the common roots of $p'$ and $q'$), and, by
the genus count $\delta+\kappa=3$ for a rational quartic, one node; the Pl\"ucker
formula gives class $m=4\cdot3-2\cdot1-3\cdot2=4$, matching the fiber count. The
richer singularity structure of the curve produces a richer stratification than in
Theorem~\ref{thm:strat3}, which is worked out completely in
Appendix~\ref{app:G} (Theorem~\ref{thm:fibG}), and follows the pattern of
Proposition~\ref{prop:fibsizes} with $d=3$: double contact over smooth curve
points gives fibers of size $2$, the two cusps give size $1$, the node --- where
two branches each impose a double contact --- gives \emph{empty} fibers, and over
the hyperplane $v_1=0$, where the monomial twist degenerates, the fiber has
$3$ points (or $1$ on a distinguished conic) --- the twist stratum invisible at
sweep level. Thus $G$ attains every fiber size in $\{4,3,2,1,0\}$.

\section{A general construction from parametrized hypersurfaces}\label{sec:dim4}

In this section we propose a general framework, valid in any dimension $n>2$:
one sweeps a tangent direction field on a parametrized hypersurface of dimension
$n-2$ in $\C^{n-1}$, pads by one multiplier variable, and twists by monomial
maps. The framework is set up in \S\ref{subsec:sweephyp} (where the branch
coefficients $L_k$ are listed explicitly for $n=4$ and $n=5$) and
\S\ref{subsec:forms} (which treats general $n$ only). We then specialize:
\S\ref{subsec:n4} is devoted to dimension four, with two counterexamples
$F_4$ and $F_5$, and \S\ref{subsec:n5} to dimension five, where the direction
field $\Delta=(1,w_1,w_1^2,w_2)^{\mathsf T}$ leads to a rigidity theorem for the middle
escape branch and, on the bottom branch, to a five-dimensional counterexample
$F_6$. Finally, \S\ref{subsec:curvefam} solves the curve-type direction field
$\Delta=(1,w_1,\dots,w_1^{n-2})^{\mathsf T}$ uniformly in every dimension --- the
normalization reduces to a $w_1$-family of constant-Jacobian maps of
$\C^{n-3}$ together with a free potential --- and produces a second
five-dimensional counterexample $F_7$, of geometric degree twelve.

\subsection{Sweeping a direction field on a hypersurface}\label{subsec:sweephyp}
The architecture sweeps a \emph{single} tangent direction field on a
parametrized hypersurface, and we set it up at once for arbitrary dimension
$n\ge3$. Let $w=(w_1,\dots,w_{n-2})$, let $X\in\C[w]^{\,n-1}$ be a column vector
of $n-1$ polynomials, viewed as a parametrized hypersurface
$X\colon\C^{n-2}\to\C^{n-1}$, and write $\partial_i=\partial/\partial w_i$, so
that $J(X)=(\partial_1X,\dots,\partial_{n-2}X)$ is an
$(n-1)\times(n-2)$ matrix whose columns span the parametrized tangent space on the
immersion locus. We call a polynomial field $\Delta\in\C[w]^{n-1}$
\emph{algebraically tangent} to $X$ if it satisfies the determinant identity
\begin{equation}\label{eq:tangcrit}
\det\bigl(\Delta,\,J(X)\bigr)\;\equiv\;0 ,
\end{equation}
where $(\Delta,J(X))$ is the square matrix obtained by adjoining the column
$\Delta$. When $J(X)$ has rank $n-2$, this is equivalent to
$\Delta(w)\in\operatorname{im}J(X)(w)$. At a rank drop it need not be;
our constructions use the determinant identity, including at those parameters. Under the standard identification
$\Lambda^{n-1}\C^{\,n-1}\cong\C$, the left-hand side is the exterior product
$\Delta\wedge\partial_1X\wedge\dots\wedge\partial_{n-2}X$ --- the
exterior-product form of the criterion referred to in the introduction. The padded sweep is
\[
F_0(x,\gamma,w)\;=\;\begin{pmatrix}\gamma x\\ X(w)+\gamma\Delta(w)\end{pmatrix}\colon
\ \C^n\to\C^n .
\]

\begin{lemma}\label{lem:JL}
Let $\Delta$ be a polynomial tangent field of $X$, i.e.\
$\det\bigl(\Delta,J(X)\bigr)\equiv0$. Then, with the variables ordered
$(x,\gamma,w_1,\dots,w_{n-2})$,
\[
\det J(F_0)\;=\;\gamma\,\det\bigl(\Delta,\;J(X)+\gamma J(\Delta)\bigr)
\;=\;\sum_{k=1}^{n-2}L_k\,\gamma^{\,k+1},
\]
where $L_k$ is the sum of the determinants $\det(\Delta,C_1,\dots,C_{n-2})$ in
which exactly $k$ of the columns $C_i$ are taken from $J(\Delta)$ and the rest
from $J(X)$.
\end{lemma}

We record the coefficients explicitly in the two dimensions used below. For
$n=4$ (parameters $w=(w_1,w_2)$, columns in $\C[w]^3$) we write $L:=L_1$ and
$M:=L_2$:
\[
\det J(F_0)=\gamma^2L+\gamma^3M,\qquad
L=\det(\Delta,\partial_1X,\partial_2\Delta)+\det(\Delta,\partial_1\Delta,\partial_2X),
\qquad
M=\det(\Delta,\partial_1\Delta,\partial_2\Delta).
\]
For $n=5$ (parameters $w=(w_1,w_2,w_3)$, columns in $\C[w]^4$),
$\det J(F_0)=\gamma^2L_1+\gamma^3L_2+\gamma^4L_3$ with
\begin{align*}
L_1&=\det(\Delta,\partial_1\Delta,\partial_2X,\partial_3X)
   +\det(\Delta,\partial_1X,\partial_2\Delta,\partial_3X)
   +\det(\Delta,\partial_1X,\partial_2X,\partial_3\Delta),\\
L_2&=\det(\Delta,\partial_1\Delta,\partial_2\Delta,\partial_3X)
   +\det(\Delta,\partial_1\Delta,\partial_2X,\partial_3\Delta)
   +\det(\Delta,\partial_1X,\partial_2\Delta,\partial_3\Delta),\\
L_3&=\det(\Delta,\partial_1\Delta,\partial_2\Delta,\partial_3\Delta).
\end{align*}
Note that $L_{n-2}$ involves $\Delta$ alone, and more generally the higher
coefficients are dominated by derivatives of $\Delta$: which of them can be made
a nonzero constant is decided by the direction field before the hypersurface is
chosen. All of these identities were verified by exact expansion.

\subsection{Prescribed direction fields and their tangency criterion}\label{subsec:forms}

In Lemma~\ref{lem:JL} the pair $(X,\Delta)$ is given together, with $\Delta$
tangent to $X$. In the constructions that follow we reverse the roles: we
\emph{prescribe} the direction field --- an explicit polynomial map
$\Delta\colon\C^{n-2}\to\C^{n-1}$, assigning to each parameter value $w$ the
direction that will be swept at the (yet unknown) point $X(w)$ --- and regard
the hypersurface $X$ as the unknown. For a fixed $\Delta$, the tangency
criterion \eqref{eq:tangcrit} then becomes a first-order partial differential
equation on the components of $X$, generally nonlinear in the full first
derivative array but multilinear in its columns; the normalization
adds the conditions on the coefficients $L_k$, in which $X$ appears with
progressively smaller weight as $k$ grows: $L_k$ is of degree $n-2-k$ in the
derivatives of $X$, and in particular the top coefficient $L_{n-2}$ involves
$\Delta$ alone. Which normalization branch is available is therefore decided by
the choice of $\Delta$ before any equation on $X$ is solved. This section works
with a general $n>2$ throughout; everything specific to $n=4$ and $n=5$ is
deferred to \S\ref{subsec:n4} and \S\ref{subsec:n5}.

Both the tangency criterion and Lemma~\ref{lem:JL} reduce to plain matrix
algebra, which we use to organize this analysis. Assume the column
$\Delta\in\C[w]^{\,n-1}$ is \emph{unimodular}, i.e.\ its entries generate the
unit ideal of $\C[w]$, and fix a polynomial completion
$U\in\mathrm{SL}_{n-1}(\C[w])$ with first column $\Delta$.
All fields used here have first entry $1$, for which
$U=(\Delta,e_2,\ldots,e_{n-1})$ is an explicit such completion. Let
\[
N\;\in\;\C[w]^{\,(n-2)\times(n-1)}
\]
be the matrix formed by the last $n-2$ rows of $U^{-1}$. Then $N\Delta=0$ (the
rows of $N$ are a free basis of the syzygies of $\Delta$), and for arbitrary
columns $V_1,\dots,V_{n-2}$,
\begin{equation}\label{eq:pairing}
\det\bigl(\Delta,V_1,\dots,V_{n-2}\bigr)
\;=\;\det\bigl(N\,(V_1,\dots,V_{n-2})\bigr),
\end{equation}
since $U^{-1}(\Delta,V_1,\dots,V_{n-2})$ has first column $e_1$ and
$\det U^{-1}=1$. Applying \eqref{eq:pairing} to the columns of
$J(X)+\gamma J(\Delta)$ and setting
\[
P\;:=\;N\,J(X),\qquad
\widetilde P\;:=\;N\,J(\Delta)
\qquad\bigl((n-2)\times(n-2)\ \text{matrices}\bigr),
\]
Lemma~\ref{lem:JL} becomes
\begin{equation}\label{eq:wedge}
\det J(F_0)\;=\;\gamma\,\det\bigl(P+\gamma\widetilde P\bigr),
\qquad\text{with}\qquad
\det P=0\ \ (\text{tangency}),
\end{equation}
and the branch coefficients $L_k$ of Lemma~\ref{lem:JL} are the coefficients of
$\det(P+\gamma\widetilde P)$ as a polynomial in $\gamma$: the tangency
$L_0=\det P\equiv0$, and, for $1\le k\le n-2$, $L_k$ is the sum of the
determinants obtained from $P$ by replacing exactly $k$ of its columns by the
corresponding columns of $\widetilde P$.

The normalization conditions now split into $n-2$ \emph{branches}, one for each
index $1\le k\le n-2$:
\[
\text{the \emph{$L_k$-branch}:}\qquad
L_j\equiv0\ (j\neq k),\qquad L_k\equiv\text{const}\neq0,
\qquad\text{giving}\qquad \det J(F_0)=L_k\,\gamma^{\,k+1} .
\]
(For $n=4$, where $L=L_1$ and $M=L_2$ as in \S\ref{subsec:sweephyp}, the two
branches are called the \emph{$L$-branch} and the \emph{$M$-branch}.)
Which branches can be realized is constrained by $\widetilde P=N\,J(\Delta)$,
which depends on $\Delta$ alone: since every term of $L_k$ contains $k$ columns
of $\widetilde P$, one has $L_j\equiv0$ automatically for all
$j>\operatorname{rank}\widetilde P$, so the rank of $\widetilde P$ bounds the
highest available branch. At the two extremes, a direction field depending on a
single parameter forces $\operatorname{rank}\widetilde P\le1$ (only the bottom
branch survives), while the tautological field with $\widetilde P$ the identity
opens the top branch $k=n-2$. Mixed direction fields, whose $\widetilde P$ has
intermediate rank, first occur for $n=5$ and are the subject of
\S\ref{subsec:n5}.

A singular matrix $P$ need not have rank one, so a single row multiplier
does not describe tangency in general dimension. When $\Delta$ has a
constant entry equal to $1$, there is instead a global elementary
potential construction. Write
$\Delta=(\delta_1,\dots,\delta_{n-1})^{\mathsf T}$ with $\delta_1=1$, take the
syzygy matrix $N$ with rows $-\,\delta_{i+1}e_1^{\mathsf T}+e_{i+1}^{\mathsf T}$
($1\le i\le n-2$), and define the \emph{potentials}
\[
G_i\;:=\;X_i-\delta_i\,X_1\qquad(i=2,\dots,n-1),
\qquad
\Gamma\;:=\;(G_2,\dots,G_{n-1})^{\mathsf T},
\qquad
\widehat\Delta\;:=\;(\delta_2,\dots,\delta_{n-1})^{\mathsf T}.
\]
Differentiating by the product rule,
$\partial_jG_i=\partial_jX_i-\delta_i\,\partial_jX_1-X_1\,\partial_j\delta_i$,
which in matrix form is
\[
P\;=\;N\,J(X)\;=\;J(\Gamma)+X_1\,J(\widehat\Delta),
\qquad\qquad
\widetilde P\;=\;N\,J(\Delta)\;=\;J(\widehat\Delta)
\]
(for the second identity, the row $-\delta_{i+1}e_1^{\mathsf T}+e_{i+1}^{\mathsf T}$
applied to $\partial_j\Delta$ gives
$\partial_j\delta_{i+1}-\delta_{i+1}\partial_j\delta_1=\partial_j\delta_{i+1}$,
since $\delta_1=1$ is constant). The whole branch analysis therefore takes
place in the $(n-2)\times(n-2)$ pencil
\[
P+\gamma\widetilde P\;=\;J(\Gamma)+\mu\,J(\widehat\Delta),
\qquad
\mu\;=\;X_1+\gamma ,
\]
in which the hypersurface enters only through the potentials $\Gamma$ and the
scalar $X_1$; this device is used systematically in \S\ref{subsec:n5}
and \S\ref{subsec:curvefam}.

\subsection{The twist in dimension \texorpdfstring{$n$}{n}: the stage equations}\label{subsec:stage}

All the higher-dimensional maps below are produced by one and the same
procedure, which converts a normalized sweep into a Keller map when the
polynomiality conditions below are satisfied. Since we use
it repeatedly, we record it once, together with the equations that have to be
solved. The input is a padded sweep $F_0=(\gamma x,\,S)^{\mathsf T}$,
$S=X+\gamma\Delta$, on a branch of \S\ref{subsec:forms}:
$\det J(F_0)=c\,\gamma^{\,k+1}$. The output has the shape
\[
F\;=\;\Bigl(C,\;\frac{S_1}{C^{\,e_1}},\;\dots,\;
\frac{S_{n-1}}{C^{\,e_{n-1}}}\Bigr)^{\mathsf T},
\qquad C=\gamma x ,
\]
after $\gamma$ and $w$ are expressed through the source variables by a
monomial map and a stage. Three layers of choices are involved.

\emph{(i) Discrete data.} Positive integers: weights $d_j$ for the parameters
($\operatorname{wt}w_j=d_j$), twist exponents $e_i$ for the target components,
and $x$-degrees $m_j$ for the source monomials. They must satisfy
\begin{equation}\label{eq:discrete}
\begin{aligned}
 \operatorname{wt}(X_i)&\ge e_i\quad(\text{every monomial}),&
 1+\operatorname{wt}(\Delta_i)&\ge e_i,\\
 \sum_j m_j&=\sum_i e_i,&
 \sum_j d_j+(k+1)&=\sum_i e_i.
\end{aligned}
\end{equation}
The first two conditions make
$E_i:=S_i\bigl(\gamma,\;\gamma^{d_1}u_1,\dots,\gamma^{d_{n-2}}u_{n-2}\bigr)
/\gamma^{\,e_i}$ a \emph{polynomial} in $(\gamma,u)$; the last two make the
$x$- and $\gamma$-powers cancel in the determinant chain below.

\emph{(ii) The stage.} On the source $\C^n\ni(x,y_1,\dots,y_{n-1})$ put
$v_j=x^{m_j}y_j$ and let the stage be
\[
(\gamma,u_1,\dots,u_{n-2})^{\mathsf T}
\;=\;\theta^0+L\,v+\textstyle\sum_\mu q_\mu\,\mu(v):
\]
a base point $\theta^0$, an invertible linear part $L$ with columns
$\mathrm{col}_1,\dots,\mathrm{col}_{n-1}$, and finitely many correction
vectors $q_\mu\in\C^{n-1}$ attached to monomials $\mu(v)$ of degree $\ge2$.
Let $J\subseteq\{1,\dots,n-1\}$ be the set of indices of the variables $v_j$
occurring in some correction monomial.

\begin{lemma}\label{lem:stage}
If every correction vector lies in
$V:=\mathrm{span}\{\mathrm{col}_j:j\notin J\}$, then the stage is a polynomial
automorphism of $\C^{n-1}$ with constant Jacobian determinant $\det L$.
\end{lemma}

\begin{proof}
The Jacobian is $L+\sum_\mu q_\mu\,\nabla\mu(v)$; in the multilinear expansion
of its determinant, any term replaces a nonempty set of columns
$\mathrm{col}_j$ with $j\in J$ (since $\partial_j\mu\neq0$ forces $j\in J$) by
vectors in $V$; together with the untouched columns
$\{\mathrm{col}_j:j\notin J\}$ this places more than $\dim V$ columns inside
$V$, so every such term vanishes. For the inverse: $L^{-1}q_\mu\in
\mathrm{span}\{e_j:j\notin J\}$, so the $J$-coordinates of $v$ are affine in
the stage values; substituting them into the corrections (which involve only
those coordinates) recovers the remaining coordinates polynomially.
\end{proof}

\emph{(iii) The side conditions.} $F$ is a polynomial map iff
$x^{\,e_i}\mid E_i\circ\text{stage}$ for every $i$; since $v_j$ has
$x$-degree $m_j$, this is the finite system
\begin{equation}\label{eq:side}
\bigl[\mu\bigr]\bigl(E_i\circ\text{stage}\bigr)\;=\;0
\qquad\text{for every monomial $\mu(v)$ of $x$-degree}<e_i,\ i=1,\dots,n-1 .
\end{equation}
After choosing a base point $\theta^0$ with $\gamma_0\ne0$ and
$E_i(\theta^0)=0$, one can solve these equations in increasing weighted
degree, provided the resulting constrained linear systems are solvable.
The base-point equations are nonlinear in $\theta^0$; equivalently they
say $S(\gamma_0,w^0)=0$, with $w_j^0=\gamma_0^{d_j}u_j^0$.
For a column of weight $m_j$, the linear monomial condition is
$\nabla E_i(\theta^0)\cdot\mathrm{col}_j=0$ when $e_i>m_j$.
For a correction monomial $\mu$ of weight $r$, put
\[
 \theta_{<r}(v)=\theta^0+Lv+
 \sum_{\operatorname{wt}\nu<r}q_\nu\nu(v).
\]
Its equation is
\[
 \nabla E_i(\theta^0)\cdot q_\mu
 =-[\mu]E_i(\theta_{<r}(v)),\qquad e_i>r.
\]
The right side includes all previously chosen nonlinear corrections.
Because the weights are positive, products involving a new weight-$r$
correction and another positive-weight term have weight greater than $r$;
this proves the displayed linear equation at that stage.
The correction vectors must lie in the subspace $V$ of
Lemma~\ref{lem:stage}, so independence of unrestricted gradients alone
does not guarantee solvability. The construction is conditional on a
solution of these equations, an invertible $L$, and that span condition.
For the explicit stages below, all conditions are checked exactly.
The Jacobian determinants then multiply to
\begin{equation}\label{eq:chain}
\det J(F)\;=\;
\underbrace{x^{\sum m_j}}_{\text{monomials}}\cdot
\underbrace{\det L}_{\text{stage}}\cdot
\underbrace{\gamma^{\sum d_j}}_{\text{scalings}}\cdot
\underbrace{c\,\gamma^{\,k+1}}_{\det J(F_0)}\cdot
\underbrace{(\gamma x)^{-\sum e_i}}_{\text{twist}}
\;=\;c\,\det L ,
\end{equation}
a nonzero constant. The discrete data of the six maps of this paper are:
\begin{center}
\begin{tabular}{lcccccc}
\hline
map & $n$ & $\det J(F_0)$ & $(d_j)$ & $(e_i)$ & $(m_j)$ & $\det J(F)=\pm c\det L$\\
\hline
$F$   & $3$ & $2\gamma^2$ & $(1)$     & $(1,2)$     & $(1,2)$     & $-2$\\
$G$   & $3$ & $2\gamma^2$ & $(1)$     & $(1,2)$     & $(1,2)$     & $2$\\
$F_4$ & $4$ & $2\gamma^2$ & $(1,1)$   & $(1,2,1)$   & $(1,1,2)$   & $-\tfrac{44}{9}$\\
$F_5$ & $4$ & $\gamma^3$  & $(1,2)$   & $(1,2,3)$   & $(1,2,3)$   & $\tfrac{160}{29}$\\
$F_6$ & $5$ & $\gamma^2$  & $(1,3,4)$ & $(1,2,3,4)$ & $(1,2,3,4)$ & $-290$\\
$F_7$ & $5$ & $\gamma^2$  & $(1,3,4)$ & $(1,2,3,4)$ & $(1,2,3,4)$ & $119377$\\
\hline
\end{tabular}
\end{center}
(the sign for $F$ comes from the traditional component reversal). All the
stage displays in the following subsections are explicit solutions of
\eqref{eq:side}.

\begin{lemma}[fibers on the invertible twist chart]\label{lem:twist-fibers}
Assume the stage is a polynomial automorphism and all polynomiality
conditions above hold. For a target $(C,v_2,\ldots,v_n)$ with $C\ne0$,
put $Y_i=C^{e_i}v_{i+1}$. Its fiber is in bijection with
\[
 \{(\gamma,w):S(\gamma,w)=Y,\ \gamma\ne0\}.
\]
\end{lemma}
\begin{proof}
A preimage has $C=x\gamma\ne0$. Conversely, given such a sweep point,
recover $x=C/\gamma$ and $u_j=w_j/\gamma^{d_j}$, invert the polynomial
stage to obtain $v_j^{\rm src}$, and set $y_j=v_j^{\rm src}/x^{m_j}$.
All divisions are by nonzero numbers, and these operations invert the
construction. Targets with $C=0$ are not covered by this argument.
\end{proof}

\subsection{The case \texorpdfstring{$n=4$}{n=4}}\label{subsec:n4}

\subsubsection{Specialization I: \texorpdfstring{$\Delta=(1,w_1,w_1^2)^{\mathsf T}$ and the map $F_4$}{the curve-type field and F4}}
\label{subsec:specI}

The column $\Delta=(1,w_1,w_1^2)^{\mathsf T}$ is unimodular, with the natural
syzygy matrix of \S\ref{subsec:forms} (rows
$-\delta_{i+1}e_1^{\mathsf T}+e_{i+1}^{\mathsf T}$):
\[
N=\begin{pmatrix}-w_1&1&0\\[1pt] -w_1^2&0&1\end{pmatrix},
\qquad N\Delta=0,\qquad
\det(\Delta,V,W)=\det\bigl(N\,(V,W)\bigr)\ \ \forall V,W .
\]
Since $\Delta$ depends only on $w_1$, its independent partial derivatives are
$\Delta'=\partial_1\Delta=(0,1,2w_1)^{\mathsf T}$ and
$\Delta''=\partial_1^2\Delta=(0,0,2)^{\mathsf T}$ (all higher ones vanish), and
$(\Delta,\Delta',\Delta'')$ is a frame of constant volume
$\det(\Delta,\Delta',\Delta'')=2$; against the syzygy matrix,
\[
N\Delta=0,\qquad
N\Delta'=\begin{pmatrix}1\\ 2w_1\end{pmatrix},\qquad
N\Delta''=\begin{pmatrix}0\\ 2\end{pmatrix}.
\]
For a ruled surface $X=A(w_1)+w_2B(w_1)$ --- so that
$\partial_1X=A'(w_1)+w_2B'(w_1)$ and $\partial_2X=B(w_1)$ --- expand the data
in the frame,
\[
B=b_1\Delta+b_2\Delta'+b_3\Delta'',\qquad
A'=a_1\Delta+a_2\Delta'+a_3\Delta''\qquad(a_i,b_i\in\C[w_1]),
\]
so that, using $\Delta'''=0$,
$B'=b_1'\Delta+(b_1+b_2')\Delta'+(b_2+b_3')\Delta''$. The product
$P=N\,J(X)=(N\,\partial_1X,\;N\,\partial_2X)$ is now explicit: abbreviating
$c_2=a_2+w_2(b_1+b_2')$ and $c_3=a_3+w_2(b_2+b_3')$,
\[
P=\begin{pmatrix}c_2 & b_2\\ 2(w_1c_2+c_3) & 2(w_1b_2+b_3)\end{pmatrix},
\qquad
\widetilde P=N\,J(\Delta)=\begin{pmatrix}1&0\\ 2w_1&0\end{pmatrix},
\]
and therefore
\[
\det\bigl(P+\gamma\widetilde P\bigr)\;=\;2\,(c_2b_3-c_3b_2)\;+\;2\,\gamma\,b_3 .
\]
The coefficient $M$ of $\gamma^2$ vanishes \emph{automatically} --- this
$\Delta$ lives on the $L$-branch --- the normalization $L\equiv2$ forces
$b_3\equiv1$, and the tangency, i.e.\ the identical vanishing in $w_2$ of the
$\gamma$-free term, reduces to
\[
b_3=1,\qquad b_1=b_2^2-b_2',\qquad a_2=a_3b_2 .
\]
Thus, on the ruled ansatz, the branch data are \emph{three free polynomials}
$b_2,a_3,a_1\in\C[w_1]$ (together with the constant of integration of $A$):
the relations displayed determine $b_3,b_1,a_2$, the curve $A$ is recovered
from $A'$ by integration --- which stays polynomial --- and \emph{every}
such choice satisfies $\det J(F_0)=2\gamma^2$ (verified by exact expansion on
the displayed determinant identities). The finitely many side conditions of the twist then select
special members of this family. (In the language of \S\ref{subsec:curvefam},
this is the case $n=4$ of the curve-type family with $G_2$ linear in $w_2$:
the three free univariate polynomials correspond to the coefficients of
$G_2=\alpha(w_1)w_2+\beta(w_1)$ and to $\varphi(w_1)$ in
$H_3=cw_2+\varphi$.)
Crucially there are solutions with $b_2\not\equiv0$: the surface is then ruled but
\emph{not} a cylinder, and the second parameter $w_2$ enters every swept
coordinate. (Cylindrical solutions $b_1=b_2=0$ produce only shears of
three-dimensional examples.)

Take $b_2=w_1$, $a_3=1$ (so $a_2=w_1$, $b_1=w_1^2-1$) and
$a_1=\tfrac13-\tfrac83w_1$. Integrating, the curve and ruling are
\[
A(w_1)=\Bigl(\tfrac13w_1-\tfrac43w_1^2,\;\;
\tfrac23w_1^2-\tfrac89w_1^3,\;\;
2w_1+\tfrac79w_1^3-\tfrac23w_1^4\Bigr)^{\mathsf T},\qquad
B(w_1)=\bigl(w_1^2-1,\;w_1^3,\;w_1^4+w_1^2+2\bigr)^{\mathsf T},
\]
so the surface $X=A+w_2B=(p,q,r)^{\mathsf T}$ and the tangent field $\Delta$ are, explicitly,
\begin{align*}
p(w_1,w_2)&=\tfrac13w_1-\tfrac43w_1^2+w_2\,(w_1^2-1),\\
q(w_1,w_2)&=\tfrac23w_1^2-\tfrac89w_1^3+w_2\,w_1^3,\\
r(w_1,w_2)&=2w_1+\tfrac79w_1^3-\tfrac23w_1^4+w_2\,(w_1^4+w_1^2+2),\\
\Delta(w_1,w_2)&=(1,\;w_1,\;w_1^2)^{\mathsf T}
\qquad(\text{independent of }w_2).
\end{align*}
The tangency identity $\det(\Delta,J(X))\equiv0$ and the
normalization $L\equiv2$, $M\equiv0$ hold for this pair by construction (and were
verified by expansion).
Define on $\C^4\ni(x,y,z,t)$:
\[
\gamma=1+xy+x^2t,\quad u_1=1+xz,\quad
u_2=-\tfrac79+\tfrac{22}9xy+\tfrac83xz,\quad
w_1=\gamma u_1,\quad w_2=\gamma u_2,\quad C=\gamma x,
\]
\[
S_i=A_i(w_1)+w_2B_i(w_1)+\gamma\,\Delta_i(w_1)\quad(i=1,2,3),
\]
\[
\boxed{\;F_4\;=\;\Bigl(C,\;\;\frac{S_1}{C},\;\;\frac{S_2}{C^2},\;\;\frac{S_3}{C}\Bigr)^{\mathsf T}.\;}
\]
In the format of the composition diagram of \S\ref{subsec:twist}, $F_4$ is the
composite
\[
\begin{pmatrix}x\\ y\\ z\\ t\end{pmatrix}
\xrightarrow{\ \text{monomial}\ }
\begin{pmatrix}x\\ xy\\ xz\\ x^2t\end{pmatrix}
\xrightarrow{\ \text{affine}\ }
\begin{pmatrix}x\\ \gamma\\ u_1\\ u_2\end{pmatrix}
\xrightarrow{\ \substack{w_1=\gamma u_1\\[1pt] w_2=\gamma u_2}\ }
\begin{pmatrix}x\\ \gamma\\ w_1\\ w_2\end{pmatrix}
\xrightarrow{\ \text{padded sweep}\ }
\begin{pmatrix}\gamma x\\ S_1\\ S_2\\ S_3\end{pmatrix}
\xrightarrow{\ \text{twist}\ }
\begin{pmatrix}C\\ S_1/C\\ S_2/C^2\\ S_3/C\end{pmatrix},
\]
whose Jacobian determinants multiply to the constant computed after
Theorem~\ref{thm:F4} below.

\begin{theorem}\label{thm:F4}
$F_4$ is a polynomial map $\C^4\to\C^4$ with components of degrees $4,11,12,21$;
its Jacobian determinant is identically $-\tfrac{44}{9}$; and its generic fiber
consists of exactly $5$ points.
\end{theorem}

The stronger conclusion of Theorem~\ref{thm:product-summary} applies:
$F_4$ is not equivalent to any three-dimensional map with an identity
coordinate adjoined, even after further stabilization. The proof uses
its omitted locus, a disjoint union of five copies of $\C^\times$;
see Theorem~\ref{prod:thm:main}.

\begin{proof}[Proof sketch]
The degrees and divisibilities are read off the construction, and
$\det J(F_4)\equiv-\tfrac{44}9$ is the chain factorization through the stage
(Lemma~\ref{lem:stage}); the required polynomial identities are verified exactly. For the fibers we use the elimination scheme that recurs for $F_5$,
$F_6$ and $F_7$: eliminate $\gamma$ through the first sweep equation, then
eliminate the remaining transverse parameter by a resultant, leaving a
univariate polynomial in $w_1$ certified by exact gcd computations. Here, let
$v_1\neq0$, set $(Y_1,Y_2,Y_3)=(v_1v_2,\,v_1^2v_3,\,v_1v_4)$, and substitute
$\gamma=Y_1-A_1-w_2B_1$ (the first sweep equation; $\Delta_1=1$) into the
other two. The resulting equations are \emph{linear} in $w_2$, with leading
coefficients $w_1$ and $2w_1^2+2$ having no common zero, and their resultant
with respect to $w_2$ is, up to sign, the compatibility determinant
\[
\begin{aligned}
R(w_1;Y)&=\det\bigl(A(w_1)-Y,\;B(w_1),\;\Delta(w_1)\bigr)\\
 &=\tfrac29w_1^5+\tfrac29w_1^4+\bigl(\tfrac89+Y_1\bigr)w_1^3
 -\bigl(\tfrac43+2Y_2\bigr)w_1^2\\
 &\qquad +(2Y_1+Y_3)w_1-2Y_2.
\end{aligned}
\]
a quintic with \emph{constant} leading coefficient $\tfrac29$ --- the tangency
quintic of the swept surface, playing the role of the tangency polynomial
\eqref{eq:tangency}. Each root carries a unique $(w_2,\gamma)$ (the two
leading coefficients are coprime), and each root with $\gamma\neq0$ lifts to
exactly one fiber point via $x=v_1/\gamma$ and the triangular stage, which
recovers $(y,z,t)$ rationally. At the witness target
$v=(\tfrac23,-\tfrac15,\tfrac12,3)$, exact univariate gcd certificates show
that $R$ is squarefree and coprime to the $w_2$-resultant of the second sweep
equation with $\gamma=0$, so that fiber has exactly $5$ points, and by
openness of these conditions the generic count is $5$.
\end{proof}

Written out in full, $F_4=(F_{4,1},F_{4,2},F_{4,3},F_{4,4})^{\mathsf T}$ with
{\small
\begin{align*}
F_{4,1} ={}& x^3t+x^2y+x\\[2pt]
F_{4,2} ={}&\tfrac{22}{9}x^6yz^2t^2+\tfrac{8}{3}x^6z^3t^2+\tfrac{44}{9}x^5y^2z^2t+\tfrac{16}{3}x^5yz^3t+\tfrac{44}{9}x^5yzt^2+\tfrac{41}{9}x^5z^2t^2+\tfrac{22}{9}x^4y^3z^2\\
&+\tfrac{8}{3}x^4y^2z^3+\tfrac{88}{9}x^4y^2zt+14x^4yz^2t+\tfrac{16}{3}x^4z^3t+\tfrac{22}{9}x^4yt^2+\tfrac{10}{9}x^4zt^2+\tfrac{44}{9}x^3y^3z+\tfrac{85}{9}x^3y^2z^2\\
&+\tfrac{16}{3}x^3yz^3+\tfrac{44}{9}x^3y^2t+12x^3yzt+\tfrac{70}{9}x^3z^2t-\tfrac{7}{9}x^3t^2+\tfrac{22}{9}x^2y^3+\tfrac{98}{9}x^2y^2z+\tfrac{92}{9}x^2yz^2+\tfrac{8}{3}x^2z^3\\
&+\tfrac{10}{3}x^2yt-\tfrac{4}{9}x^2zt+\tfrac{37}{9}xy^2+\tfrac{40}{9}xyz+\tfrac{29}{9}xz^2-\tfrac{26}{9}xt-\tfrac{26}{9}y-\tfrac{35}{9}z\\[2pt]
F_{4,3} ={}& \tfrac{22}{9}x^6yz^3t^2+\tfrac{8}{3}x^6z^4t^2+\tfrac{44}{9}x^5y^2z^3t+\tfrac{16}{3}x^5yz^4t+\tfrac{22}{3}x^5yz^2t^2+\tfrac{65}{9}x^5z^3t^2+\tfrac{22}{9}x^4y^3z^3\\
&+\tfrac{8}{3}x^4y^2z^4+\tfrac{44}{3}x^4y^2z^2t+\tfrac{58}{3}x^4yz^3t+\tfrac{16}{3}x^4z^4t+\tfrac{22}{3}x^4yzt^2+\tfrac{17}{3}x^4z^2t^2+\tfrac{22}{3}x^3y^3z^2\\
&+\tfrac{109}{9}x^3y^2z^3+\tfrac{16}{3}x^3yz^4+\tfrac{44}{3}x^3y^2zt+26x^3yz^2t+\tfrac{122}{9}x^3z^3t+\tfrac{22}{9}x^3yt^2+\tfrac{1}{3}x^3zt^2+\tfrac{22}{3}x^2y^3z\\
&+\tfrac{61}{3}x^2y^2z^2+16x^2yz^3+\tfrac{8}{3}x^2z^4+\tfrac{44}{9}x^2y^2t+\tfrac{46}{3}x^2yzt+\tfrac{26}{3}x^2z^2t-\tfrac{7}{9}x^2t^2+\tfrac{22}{9}xy^3+15xy^2z\\
&+16xyz^2+\tfrac{19}{3}xz^3+\tfrac{10}{3}xyt-2xzt+\tfrac{37}{9}y^2+\tfrac{16}{3}yz+\tfrac{11}{3}z^2-\tfrac{22}{9}t\\[2pt]
F_{4,4} ={}& \tfrac{22}{9}x^{12}yz^4t^4+\tfrac{8}{3}x^{12}z^5t^4+\tfrac{88}{9}x^{11}y^2z^4t^3+\tfrac{32}{3}x^{11}yz^5t^3+\tfrac{88}{9}x^{11}yz^3t^4+\tfrac{89}{9}x^{11}z^4t^4\\
&+\tfrac{44}{3}x^{10}y^3z^4t^2+16x^{10}y^2z^5t^2+\tfrac{352}{9}x^{10}y^2z^3t^3+\tfrac{148}{3}x^{10}yz^4t^3+\tfrac{32}{3}x^{10}z^5t^3+\tfrac{88}{9}x^9y^4z^4t\\
&+\tfrac{32}{3}x^9y^3z^5t+\tfrac{44}{3}x^{10}yz^2t^4+\tfrac{116}{9}x^{10}z^3t^4+\tfrac{176}{3}x^9y^3z^3t^2+\tfrac{266}{3}x^9y^2z^4t^2+32x^9yz^5t^2\\
&+\tfrac{22}{9}x^8y^5z^4+\tfrac{8}{3}x^8y^4z^5+\tfrac{176}{3}x^9y^2z^2t^3+\tfrac{272}{3}x^9yz^3t^3+\tfrac{350}{9}x^9z^4t^3+\tfrac{352}{9}x^8y^4z^3t+\tfrac{620}{9}x^8y^3z^4t\\
&+32x^8y^2z^5t+\tfrac{88}{9}x^9yzt^4+6x^9z^2t^4+88x^8y^3z^2t^2+\tfrac{584}{3}x^8y^2z^3t^2+\tfrac{394}{3}x^8yz^4t^2+16x^8z^5t^2\\
&+\tfrac{88}{9}x^7y^5z^3+\tfrac{59}{3}x^7y^4z^4+\tfrac{32}{3}x^7y^3z^5+\tfrac{352}{9}x^8y^2zt^3+\tfrac{248}{3}x^8yz^2t^3+\tfrac{440}{9}x^8z^3t^3+\tfrac{176}{3}x^7y^4z^2t\\
&+\tfrac{1520}{9}x^7y^3z^3t+146x^7y^2z^4t+32x^7yz^5t+\tfrac{22}{9}x^8yt^4-\tfrac{4}{9}x^8zt^4+\tfrac{176}{3}x^7y^3zt^2+212x^7y^2z^2t^2\\
&+\tfrac{616}{3}x^7yz^3t^2+\tfrac{172}{3}x^7z^4t^2+\tfrac{44}{3}x^6y^5z^2+52x^6y^4z^3+\tfrac{482}{9}x^6y^3z^4+16x^6y^2z^5+\tfrac{88}{9}x^7y^2t^3\\
&+\tfrac{112}{3}x^7yzt^3+20x^7z^2t^3+\tfrac{352}{9}x^6y^4zt+200x^6y^3z^2t+264x^6y^2z^3t+\tfrac{1120}{9}x^6yz^4t+\tfrac{32}{3}x^6z^5t\\
&-\tfrac{7}{9}x^7t^4+\tfrac{44}{3}x^6y^3t^2+\tfrac{344}{3}x^6y^2zt^2+\tfrac{1354}{9}x^6yz^2t^2+\tfrac{655}{9}x^6z^3t^2+\tfrac{88}{9}x^5y^5z+\tfrac{194}{3}x^5y^4z^2\\
&+\tfrac{968}{9}x^5y^3z^3+\tfrac{604}{9}x^5y^2z^4+\tfrac{32}{3}x^5yz^5+\tfrac{20}{3}x^6yt^3-\tfrac{40}{9}x^6zt^3+\tfrac{88}{9}x^5y^4t+\tfrac{1040}{9}x^5y^3zt\\
&+\tfrac{2168}{9}x^5y^2z^2t+\tfrac{554}{3}x^5yz^3t+\tfrac{338}{9}x^5z^4t+\tfrac{74}{3}x^5y^2t^2+\tfrac{452}{9}x^5yzt^2+\tfrac{287}{9}x^5z^2t^2+\tfrac{22}{9}x^4y^5\\
&+\tfrac{116}{3}x^4y^4z+\tfrac{994}{9}x^4y^3z^2+\tfrac{1007}{9}x^4y^2z^3+40x^4yz^4+\tfrac{8}{3}x^4z^5-\tfrac{34}{9}x^5t^3+\tfrac{236}{9}x^4y^3t+\tfrac{1024}{9}x^4y^2zt\\
&+\tfrac{382}{3}x^4yz^2t+\tfrac{454}{9}x^4z^3t+\tfrac{52}{9}x^4yt^2-\tfrac{47}{9}x^4zt^2+9x^3y^4+\tfrac{532}{9}x^3y^3z+\tfrac{859}{9}x^3y^2z^2+\tfrac{542}{9}x^3yz^3\\
&+\tfrac{83}{9}x^3z^4+\tfrac{206}{9}x^3y^2t+\tfrac{346}{9}x^3yzt+\tfrac{250}{9}x^3z^2t-\tfrac{17}{3}x^3t^2+\tfrac{40}{3}x^2y^3+\tfrac{131}{3}x^2y^2z+\tfrac{404}{9}x^2yz^2\\
&+\tfrac{41}{3}x^2z^3+\tfrac{10}{3}x^2yt+\tfrac{10}{9}x^2zt+9xy^2+\tfrac{142}{9}xyz+\tfrac{89}{9}xz^2-\tfrac{28}{9}xt+\tfrac{20}{3}y+\tfrac{29}{3}z
\end{align*}
}

The proof of the Jacobian statement is the verified chain factorization
$x^4\cdot(-\tfrac{22}9)\cdot\gamma^2\cdot2\gamma^2\cdot(\gamma x)^{-4}
=-\tfrac{44}9$, in which the factor $2\gamma^2$ is Lemma~\ref{lem:JL} for the
surface above ($L=2$, $M=0$; a polynomial identity), the factor $\gamma^2$ comes
from the scaling $(w_1,w_2)=(\gamma u_1,\gamma u_2)$, $x^4$ from the monomial
substitution $(v_1,v_2,v_3)=(xy,xz,x^2t)$ with $x$-degrees $(1,1,2)$, and
$-\tfrac{22}9$ is the determinant of the affine stage. The five side conditions
(divisibility of $S_1,S_3$ by $C$ and of $S_2$ by $C^2$) determine the constants
$(\gamma_0,u_1^0,u_2^0,\alpha_0,\alpha_1)=(1,1,-\tfrac79,\tfrac13,-\tfrac83)$;
the grading $(1,1,2)$ leaves the affine stage invertible. The fiber count was
established by the quintic and the exact witness in the proof above.
The reconstruction and all exceptional fibers on the nonzero-coordinate
chart are analyzed in Part~\ref{part:equivalence}.

\subsubsection{Specialization II: \texorpdfstring{$\Delta=(1,w_1,w_2)^{\mathsf T}$}{the tautological field} and a map of geometric degree ten}
\label{subsec:specII}

The column $\Delta=(1,w_1,w_2)^{\mathsf T}$ is unimodular with the natural
syzygy matrix
\[
N=\begin{pmatrix}-w_1&1&0\\ -w_2&0&1\end{pmatrix},
\qquad N\Delta=0 ,
\]
and now the partial derivatives of $\Delta$ are the \emph{constant} columns
$\partial_1\Delta=e_2$ and $\partial_2\Delta=e_3$: the triple
$(\Delta,\partial_1\Delta,\partial_2\Delta)$ is a frame of constant volume
$\det(\Delta,e_2,e_3)=1$, with $N\Delta=0$, $Ne_2=(1,0)^{\mathsf T}$,
$Ne_3=(0,1)^{\mathsf T}$, so that $\widetilde P=N\,J(\Delta)=I_2$, the identity
matrix. Exactly as in \S\ref{subsec:specI}, expand the columns of $J(X)$ in the
frame,
\[
\partial_1X=a_1\Delta+a_2e_2+a_3e_3,\qquad
\partial_2X=b_1\Delta+b_2e_2+b_3e_3
\qquad(a_i,b_i\in\C[w_1,w_2]);
\]
comparing first entries forces $a_1=\partial_1X_1$ and $b_1=\partial_2X_1$, so
the first coefficients are automatically integrable, with $X_1$ itself as their
potential. The products with $N$ are again explicit:
\[
P=N\,J(X)=\begin{pmatrix}a_2 & b_2\\ a_3 & b_3\end{pmatrix},
\qquad
\det\bigl(P+\gamma\widetilde P\bigr)
\;=\;(a_2b_3-a_3b_2)\;+\;\gamma\,(a_2+b_3)\;+\;\gamma^2 .
\]
Thus $M\equiv1$ identically and $L=a_2+b_3$, the trace of $P$: this $\Delta$
lives on the $M$-branch, with normalization $L\equiv0$ and
$\det J(F_0)=\gamma^3$.

A useful family on this branch is obtained by combining the algebraic
conditions with the integrability of the expansion. The tangency and the normalization
$L\equiv0$ give two equations,
\[
a_2b_3-a_3b_2\;=\;0
\qquad\text{and}\qquad
a_2+b_3\;=\;0 ,
\]
which say that $P$ is nilpotent. On the chart $b_2\ne0$, set
$\lambda=b_2$ and $\rho=-a_2/b_2$; these are generally rational functions,
and one obtains
\[
(a_2,\;b_2,\;a_3,\;b_3)\;=\;(-\rho\lambda,\;\lambda,\;-\rho^2\lambda,\;\rho\lambda).
\]
Choosing $\rho,\lambda$ polynomial gives a polynomial family, rather than
a global classification of all nilpotent matrices (the chart omits,
for instance, $b_2=a_2=b_3=0$, $a_3\ne0$).
The equality of mixed partials $\partial_2(\partial_1X)=\partial_1(\partial_2X)$,
expanded in the frame (using $\partial_2\Delta=e_3$ and $\partial_1\Delta=e_2$,
the $e_i$ being constant), gives three scalar equations:
\[
\partial_2a_1=\partial_1b_1\ \ (\text{automatic}),\qquad
\partial_2a_2=b_1+\partial_1b_2,\qquad
a_1+\partial_2a_3=\partial_1b_3 .
\]
The last two express the gradient of $X_1$ through $(\rho,\lambda)$:
\[
\partial_2X_1=\partial_2a_2-\partial_1b_2=-\,\partial_2(\rho\lambda)-\partial_1\lambda,
\qquad
\partial_1X_1=\partial_1b_3-\partial_2a_3=\partial_1(\rho\lambda)+\partial_2(\rho^2\lambda),
\]
and the compatibility $\partial_1\partial_2X_1=\partial_2\partial_1X_1$ of this
gradient is the wave-type equation, linear in $\lambda$ for fixed $\rho$,
\begin{equation}\label{eq:wave}
\lambda_{11}+2(\rho\lambda)_{12}+(\rho^2\lambda)_{22}\;=\;0 ,
\end{equation}
after which $X_1$ is recovered by integration,
and then the potentials
$G_2=X_2-w_1X_1$ and $G_3=X_3-w_2X_1$ --- whose gradients are, by
\S\ref{subsec:forms}, the rows of $P-X_1\widetilde P$ --- by two further
integrations.

Solutions of \eqref{eq:wave} are produced in two steps. First fix the
multiplier $\rho\in\C[w_1,w_2]$; since the rank-one parametrization above is
linear in $\lambda$ for fixed $\rho$, equation \eqref{eq:wave} is then
\emph{linear} in $\lambda$, with polynomial coefficients (a degenerate
second-order equation in the plane, cf.\ \cite{CourantHilbert1962}), and its
polynomial solutions of any bounded degree are found by comparing
coefficients: they form the kernel of an explicit linear map between
finite-dimensional spaces of polynomials. Each solution $(\rho,\lambda)$ then
determines a surface: $X_1$ is recovered from the displayed gradient ---
whose compatibility is exactly \eqref{eq:wave}, which is \emph{why} the
equation governs the branch --- and $X_2,X_3$ follow by the further
integrations. For the choice $\rho=w_1$ used below no ad hoc computation is
needed: \eqref{eq:wave} reads
$\lambda_{11}+2w_1\lambda_{12}+w_1^2\lambda_{22}+2\lambda_2=0$, and the
substitution $\eta=w_2-\tfrac12w_1^2$ transforms it into the \emph{heat
equation},
\[
\lambda(w_1,w_2)\;=\;\mu\bigl(w_1,\;w_2-\tfrac12w_1^2\bigr),
\qquad
\partial_{w_1}^{2}\mu+\partial_\eta\mu\;=\;0
\]
(verified by exact expansion). Its polynomial solutions are completely
described in the literature: by Rosenbloom and Widder
\cite{RosenbloomWidder1959}, they are spanned by the \emph{heat polynomials}
\[
v_k(x,t)\;=\;\sum_{j\le k/2}\frac{k!}{j!\,(k-2j)!}\,x^{\,k-2j}\,t^{\,j},
\]
one in each degree $k$. Hence \eqref{eq:wave} with $\rho=w_1$ has exactly one
new polynomial solution in each degree, namely
$v_k\bigl(w_1,\tfrac12w_1^2-w_2\bigr)$; normalized to be monic in $w_1$, the
first six are
\[
1,\quad w_1,\quad w_1^2-w_2,\quad w_1^3-\tfrac32w_1w_2,\quad
w_1^4-\tfrac{12}5w_1^2w_2+\tfrac65w_2^2,\quad
w_1^5-\tfrac{40}{13}w_1^3w_2+\tfrac{30}{13}w_1w_2^2 .
\]
The affine $\lambda$ chosen below is a combination of the first three; every
further heat polynomial is a candidate surface for the construction, whose
fiber geometry is unexplored.

Taking $\rho=w_1$ and
$\lambda=\nu_1+\nu_2w_1+\nu_3(w_2-w_1^2)$, which solves \eqref{eq:wave}, and
integrating, one obtains a three-parameter family of quintic surfaces, linear
in $(\nu_1,\nu_2,\nu_3)$.
The five divisibility constants of the twist (with scalings
$(w_1,w_2)=(\gamma u_1,\gamma^2u_2)$, monomials $(v_1,v_2,v_3)=(xy,x^2z,x^3t)$
of $x$-degrees $(1,2,3)$, and twist exponents $(1,2,3)$, so that
$\det=\;C^{-6}\cdot\gamma^3\cdot\gamma^3\cdot x^6\cdot\det(\text{stage})$)
now form a \emph{solvable linear system}: at the base point
$(\gamma_0,u_1^0,u_2^0)=(1,1,1)$ one finds
$(\nu_1,\nu_2,\nu_3)=\bigl(\tfrac{51}{29},-\tfrac{60}{29},-\tfrac{240}{29}\bigr)$,
with surface
\begin{align*}
X_1&=\tfrac{160}{29}w_1^3-\tfrac{60}{29}w_1^2-\tfrac{240}{29}w_1w_2+\tfrac{51}{29}w_1+\tfrac{60}{29}w_2,\\
X_2&=\tfrac{60}{29}w_1^4-\tfrac{20}{29}w_1^3-\tfrac{120}{29}w_2^2+\tfrac{51}{29}w_2,\\
X_3&=-\tfrac{48}{29}w_1^5+\tfrac{15}{29}w_1^4+\tfrac{240}{29}w_1^3w_2-\tfrac{17}{29}w_1^3-\tfrac{60}{29}w_1^2w_2-\tfrac{240}{29}w_1w_2^2+\tfrac{51}{29}w_1w_2+\tfrac{30}{29}w_2^2 .
\end{align*}
The residual quadratic side condition is absorbed by a \emph{nonlinear
stage}: with $v=(v_1,v_2,v_3)=(xy,x^2z,x^3t)$, set
\[
\gamma=1+\tfrac{20}{29}v_1,\qquad
u_1=1-\tfrac{49}{29}v_1+8v_2,\qquad
u_2=1-\tfrac{69}{29}v_1+9v_2+v_3+\tfrac{4197}{1682}\,v_1^2 ,
\]
where the $v_1$-column lies in $\ker(\nabla E_2)\cap\ker(\nabla E_3)$, the
$v_2$-column in $\ker(\nabla E_3)$, and the coefficient $\tau=\tfrac{4197}{1682}$
kills the $v_1^2$-obstruction while leaving the stage Jacobian constant
($=\tfrac{160}{29}$), because the corresponding cofactor vanishes by design.
With $w_1=\gamma u_1$, $w_2=\gamma^2u_2$, $C=\gamma x$ and
$S_i=X_i(w_1,w_2)+\gamma\Delta_i(w_1,w_2)$, define
\[
\boxed{\;F_5\;=\;\Bigl(C,\;\;\frac{S_1}{C},\;\;\frac{S_2}{C^2},\;\;\frac{S_3}{C^3}\Bigr)^{\mathsf T}.\;}
\]
As a composition diagram,
\[
\begin{pmatrix}x\\ y\\ z\\ t\end{pmatrix}
\xrightarrow{\ \text{monomial}\ }
\begin{pmatrix}x\\ xy\\ x^2z\\ x^3t\end{pmatrix}
\xrightarrow{\ \text{stage}\ }
\begin{pmatrix}x\\ \gamma\\ u_1\\ u_2\end{pmatrix}
\xrightarrow{\ \substack{w_1=\gamma u_1\\[1pt] w_2=\gamma^2u_2}\ }
\begin{pmatrix}x\\ \gamma\\ w_1\\ w_2\end{pmatrix}
\xrightarrow{\ \text{padded sweep}\ }
\begin{pmatrix}\gamma x\\ S_1\\ S_2\\ S_3\end{pmatrix}
\xrightarrow{\ \text{twist}\ }
\begin{pmatrix}C\\ S_1/C\\ S_2/C^2\\ S_3/C^3\end{pmatrix},
\]
in which the second arrow is now the affine stage followed by a polynomial shear displayed above
(a polynomial automorphism of the coordinates $(\gamma,u_1,u_2)$ over $x$)
rather than a plain affine map.

\begin{theorem}\label{thm:F5}
$F_5$ is a polynomial map $\C^4\to\C^4$ with components of degrees $3,12,14,16$;
its Jacobian determinant is identically $\tfrac{160}{29}$; and its generic fiber
consists of exactly $10$ points.
\end{theorem}

Consequently $F_5$ is a four-dimensional
non-injective Keller map arising from the $M$-branch, of geometric degree ten,
which cannot split off an identity coordinate even after stabilization,
by Theorem~\ref{prod:thm:F5-complete}. This excludes three-dimensional
factors of degree ten as well as factors of different degrees.

\begin{proof}[Proof sketch]
The divisibilities and the constancy of the determinant are the chain
factorization above, with polynomiality checked by exact division.
For the fibers we run the same elimination scheme as for $F_4$: let
$v_1\neq0$ and set $(Y_1,Y_2,Y_3)=(v_1v_2,\,v_1^2v_3,\,v_1^3v_4)$. Eliminating
$\gamma=Y_1-X_1(w_1,w_2)$ leaves the two equations
\[
X_2+(Y_1-X_1)\,w_1\;=\;Y_2,\qquad X_3+(Y_1-X_1)\,w_2\;=\;Y_3
\]
in $(w_1,w_2)$, whose resultant with respect to $w_2$ is, after the harmless
rescaling by $-29^4/30$, the degree-ten polynomial
{\small
\begin{align*}
R(w_1;Y)={}&20480w_1^{10}-51200w_1^{9}+57600w_1^{8}+\bigl(222720Y_1-126720\bigr)w_1^{7}\\
&+\bigl(-204160Y_1-742400Y_2+234480\bigr)w_1^{6}
+\bigl(-119712Y_1+890880Y_2+890880Y_3-155448\bigr)w_1^{5}\\
&+\bigl(201840Y_1^2+59160Y_1-417600Y_2-1113600Y_3+65025\bigr)w_1^{4}\\
&+\bigl(-67280Y_1^2-538240Y_1Y_2+41412Y_1+306240Y_2+556800Y_3-44217\bigr)w_1^{3}\\
&+\bigl(-146334Y_1^2+201840Y_1Y_2+807360Y_1Y_3+75429Y_1-88740Y_2-459360Y_3\bigr)w_1^{2}\\
&+\bigl(97556Y_1^3-42891Y_1^2-403680Y_1Y_3+177480Y_3\bigr)w_1\\
&-97556Y_1^2Y_2+42891Y_1Y_2+171564Y_1Y_3-25230Y_2^2+201840Y_2Y_3-403680Y_3^2-75429Y_3 ,
\end{align*}
}
with constant leading coefficient. This controls the degree of the
projection polynomial; lifting each root to a sweep point still requires
checking the common-root equations. At
the witness target
$v=(\tfrac12,\tfrac23,-\tfrac13,1)$, exact univariate gcd certificates show
that $R$ is squarefree; that $R$ is coprime to the $w_2$-resultant of the
first equation with $\gamma=0$, so every root has $\gamma\neq0$; and that $R$
is coprime to the coefficient of $w_2$ in the linear first subresultant
of the pair, so each root $w_1$ carries a unique $w_2$. Since $x=v_1/\gamma$ and the triangular stage recovers $(y,z,t)$
rationally wherever $\gamma\neq0$, that fiber has exactly $10$ points, and by
openness of these conditions the generic count is $10$. 
\end{proof}

Written out in full,
$F_5=(F_{5,1},F_{5,2},F_{5,3},F_{5,4})^{\mathsf T}$ with
{\small
\begin{align*}
F_{5,1} ={}& \tfrac{20}{29}x^2y+x\\[2pt]
F_{5,2} ={}& \tfrac{32768000}{24389}x^7y^2z^3-\tfrac{602112000}{707281}x^6y^3z^2-\tfrac{768000}{24389}x^6y^2zt+\tfrac{3276800}{841}x^6yz^3+\tfrac{2076288000}{20511149}x^5y^4z\\
&+\tfrac{4704000}{707281}x^5y^3t-\tfrac{54835200}{24389}x^5y^2z^2+\tfrac{80752000}{20511149}x^4y^5-\tfrac{76800}{841}x^5yzt+\tfrac{81920}{29}x^5z^3+\tfrac{152428800}{707281}x^4y^3z\\
&+\tfrac{374400}{24389}x^4y^2t-\tfrac{1044480}{841}x^4yz^2+\tfrac{169140800}{20511149}x^3y^4-\tfrac{1920}{29}x^4zt+\tfrac{515520}{24389}x^3y^2z+\tfrac{3360}{841}x^3yt+\tfrac{9600}{29}x^3z^2\\
&+\tfrac{907520}{707281}x^2y^3-\tfrac{108960}{841}x^2yz-\tfrac{180}{29}x^2t+\tfrac{14050}{24389}xy^2-\tfrac{252}{29}xz+y\\[2pt]
F_{5,3} ={}& \tfrac{98304000}{24389}x^8y^2z^4-\tfrac{2408448000}{707281}x^7y^3z^3+\tfrac{9830400}{841}x^7yz^4+\tfrac{22127616000}{20511149}x^6y^4z^2-\tfrac{191692800}{24389}x^6y^2z^3\\
&-\tfrac{90354432000}{594823321}x^5y^5z-\tfrac{48000}{24389}x^6y^2t^2+\tfrac{245760}{29}x^6z^4-\tfrac{201456000}{20511149}x^5y^4t+\tfrac{45158400}{24389}x^5y^3z^2\\
&-\tfrac{73022484000}{17249876309}x^4y^6-\tfrac{864000}{24389}x^5y^2zt-\tfrac{1310720}{841}x^5yz^3-\tfrac{5316643200}{20511149}x^4y^4z-\tfrac{4800}{841}x^5yt^2-\tfrac{13521600}{707281}x^4y^3t\\
&-\tfrac{25906560}{24389}x^4y^2z^2-\tfrac{4696088400}{594823321}x^3y^5-\tfrac{86400}{841}x^4yzt+\tfrac{112640}{29}x^4z^3+\tfrac{69664320}{707281}x^3y^3z-\tfrac{120}{29}x^4t^2\\
&+\tfrac{62760}{24389}x^3y^2t-\tfrac{1613760}{841}x^3yz^2+\tfrac{40879390}{20511149}x^2y^4-\tfrac{2160}{29}x^3zt+\tfrac{3661320}{24389}x^2y^2z+\tfrac{240}{29}x^2yt+\tfrac{9480}{29}x^2z^2\\
&-\tfrac{516820}{707281}xy^3-\tfrac{105360}{841}xyz-\tfrac{189}{29}xt-\tfrac{1437}{1682}y^2-z\\[2pt]
F_{5,4} ={}& -\tfrac{629145600}{24389}x^9y^2z^5+\tfrac{19267584000}{707281}x^8y^3z^4+\tfrac{49152000}{24389}x^8y^2z^3t-\tfrac{62914560}{841}x^8yz^5-\tfrac{132882432000}{20511149}x^7y^4z^3\\
&-\tfrac{903168000}{707281}x^7y^3z^2t+\tfrac{1975910400}{24389}x^7y^2z^4-\tfrac{15504384000}{20511149}x^6y^5z^2-\tfrac{768000}{24389}x^7y^2zt^2+\tfrac{4915200}{841}x^7yz^3t\\
&-\tfrac{1572864}{29}x^7z^5+\tfrac{2308608000}{20511149}x^6y^4zt-\tfrac{15174451200}{707281}x^6y^3z^3+\tfrac{3799290048000}{17249876309}x^5y^6z+\tfrac{4704000}{707281}x^6y^3t^2\\
&-\tfrac{85708800}{24389}x^6y^2z^2t+\tfrac{54312960}{841}x^6yz^4+\tfrac{8448384000}{594823321}x^5y^5t-\tfrac{11697561600}{20511149}x^5y^4z^2+\tfrac{2437443220800}{500246412961}x^4y^7\\
&-\tfrac{76800}{841}x^6yzt^2+\tfrac{122880}{29}x^6z^3t+\tfrac{195724800}{707281}x^5y^3zt-\tfrac{496404480}{24389}x^5y^2z^3+\tfrac{184263724800}{594823321}x^4y^5z\\
&+\tfrac{374400}{24389}x^5y^2t^2-\tfrac{1873920}{841}x^5yz^2t+\tfrac{184320}{29}x^5z^4+\tfrac{484262400}{20511149}x^4y^4t+\tfrac{1251989760}{707281}x^4y^3z^2\\
&+\tfrac{135717082080}{17249876309}x^3y^6-\tfrac{1920}{29}x^5zt^2+\tfrac{2238720}{24389}x^4y^2zt-\tfrac{844800}{841}x^4yz^3-\tfrac{4465704480}{20511149}x^3y^4z+\tfrac{2760}{841}x^4yt^2\\
&+\tfrac{7680}{29}x^4z^2t-\tfrac{6755640}{707281}x^3y^3t-\tfrac{9897600}{24389}x^3y^2z^2-\tfrac{2476692498}{594823321}x^2y^5-\tfrac{145200}{841}x^3yzt+\tfrac{123776}{29}x^3z^3\\
&-\tfrac{61829880}{707281}x^2y^3z-\tfrac{210}{29}x^3t^2-\tfrac{126390}{24389}x^2y^2t-\tfrac{1938744}{841}x^2yz^2+\tfrac{128225745}{41022298}xy^4-\tfrac{2412}{29}x^2zt\\
&+\tfrac{5740350}{24389}xy^2z+\tfrac{11001}{841}xyt+\tfrac{9318}{29}xz^2+\tfrac{1629343}{1414562}y^3-\tfrac{99879}{841}yz-\tfrac{160}{29}t
\end{align*}
}

\subsection{The case \texorpdfstring{$n=5$}{n=5}: mixed direction fields}\label{subsec:n5}

In dimension five the parameters are $w=(w_1,w_2,w_3)$, the hypersurface
$X\in\C[w]^4$ is a parametrized threefold in $\C^4$, the matrices
$P,\widetilde P$ are $3\times3$, and there are three potential branches
$L_1,L_2,L_3$. Three natural unimodular direction fields with first entry $1$
present themselves, and the general principle of \S\ref{subsec:forms} --- the
rank of $\widetilde P$ bounds the highest branch --- sorts them immediately.
Throughout we use the syzygy matrix $N$ with rows
$-\delta_{i+1}e_1^{\mathsf T}+e_{i+1}^{\mathsf T}$ and the potentials
$\Gamma=(G_2,G_3,G_4)$ of \S\ref{subsec:forms}, so that
$\widetilde P=J(\widehat\Delta)$ and the branch analysis lives in the pencil
$J(\Gamma)+\mu\widetilde P$, $\mu=X_1+\gamma$.

\begin{itemize}
\item $\Delta=(1,w_1,w_1^2,w_1^3)^{\mathsf T}$: the field depends on $w_1$ alone, so
$J(\widehat\Delta)$ has only its first column nonzero and
$\operatorname{rank}\widetilde P=1$: $L_2\equiv L_3\equiv0$ automatically, only
the bottom branch is available, and the analysis is parallel to
Specialization~I --- one sweeps a one-parameter direction field along a
threefold, using the constant-volume frame $(\Delta,\Delta',\Delta'',\Delta''')$
of its partial derivatives ($\det=12$). In \S\ref{subsec:curvefam} we solve this case
completely, in every dimension at once, and extract a second five-dimensional
counterexample from it.
\item $\Delta=(1,w_1,w_2,w_3)^{\mathsf T}$: the tautological field, with
$\widetilde P=J(w_1,w_2,w_3)^{\mathsf T}=I_3$, so $L_3\equiv1$: the top branch
is open, $\det J(F_0)=\gamma^4$, and the normalization $L_1\equiv L_2\equiv0$
is equivalent to nilpotence of the $3\times3$ matrix $P$.
Unlike the $2\times2$ case, such a matrix can have rank two; a separate
analysis is required, and we do not pursue this branch here.
\item $\Delta=(1,w_1,w_1^2,w_2)^{\mathsf T}$: the first genuinely \emph{mixed} field ---
neither curve-like nor tautological. (Any field
$(1,w_1,w_1^2,\ell(w_2,w_3))^{\mathsf T}$ with $\ell$ a nonconstant linear
form, e.g.\ $w_2+w_3$, reduces to this one by an affine change of the
parameters $(w_2,w_3)$, which only rescales $\det J(F_0)$.) This is the case
we now work out; it looks very different, and indeed behaves differently: its
middle branch turns out to be rigid on a natural locus.
\end{itemize}

\medskip\noindent\textbf{The pencil for $\Delta=(1,w_1,w_1^2,w_2)^{\mathsf T}$.}
Here
\[
\partial_3\Delta=0,\qquad
\widehat\Delta=(w_1,w_1^2,w_2)^{\mathsf T},\qquad
\widetilde P=J(\widehat\Delta)=
\begin{pmatrix}1&0&0\\ 2w_1&0&0\\ 0&1&0\end{pmatrix},
\]
of rank $2$: the top branch dies, $L_3\equiv0$, and both the bottom and the
middle branch are a priori available.

The frame viewpoint of \S\ref{subsec:n4} applies here as well, provided one
takes \emph{all} independent partial derivatives of $\Delta$, including the
second-order one: $(\Delta,\partial_1\Delta,\partial_1^2\Delta,\partial_2\Delta)$
is a frame of constant volume $2$, with
\[
N\Delta=0,\qquad
N\partial_1\Delta=(1,\,2w_1,\,0)^{\mathsf T},\qquad
N\partial_1^2\Delta=(0,\,2,\,0)^{\mathsf T},\qquad
N\partial_2\Delta=(0,\,0,\,1)^{\mathsf T},
\]
and expanding
$\partial_jX=c_{0j}\Delta+c_{1j}\partial_1\Delta+c_{2j}\partial_1^2\Delta
+c_{3j}\partial_2\Delta$ forces $c_{0j}=\partial_jX_1$, as always, and gives
the columns of $P=N\,J(X)$ as $(c_{1j},\,2w_1c_{1j}+2c_{2j},\,c_{3j})^{\mathsf T}$
(all verified by expansion). The difference with \S\ref{subsec:n4} lies in
what must be done with these first-order data: the branch equations below have
to be \emph{integrated} --- twice on the middle branch --- so it is more
efficient to work from the start with the integrated objects, the potentials
of \S\ref{subsec:forms}, of whose derivatives the frame coefficients are
triangular combinations: $P=J(\Gamma)+X_1\widetilde P$. With the potentials
\[
G_2=X_2-w_1X_1,\qquad G_3=X_3-w_1^2X_1,\qquad G_4=X_4-w_2X_1 ,
\]
the pencil determinant is quadratic in $\mu$,
\[
\det\bigl(J(\Gamma)+\mu\widetilde P\bigr)\;=\;D_0+\mu D_1+\mu^2D_2 ,
\]
with coefficients (writing $H:=G_3-2w_1G_2$)
\begin{align*}
D_2&=-\,\partial_3H,\\
D_1&=(\partial_2H)\,\partial_3G_4-(\partial_3H)\,\partial_2G_4
-\bigl(\partial_1G_2\,\partial_3G_3-\partial_1G_3\,\partial_3G_2\bigr),\\
D_0&=\det J(G_2,G_3,G_4),
\end{align*}
and the padded sweep of Lemma~\ref{lem:JL} satisfies the exact identity
\begin{equation}\label{eq:pencil5}
\det J(F_0)\;=\;\gamma\,\Bigl[D_0+(X_1+\gamma)D_1+(X_1+\gamma)^2D_2\Bigr],
\end{equation}
verified by expansion for generic data. Comparing coefficients of $\gamma$,
the two branches become:
\begin{itemize}
\item \emph{bottom ($L_1$-)branch}: $D_2\equiv0$, $D_1\equiv c\neq0$, and
$X_1=-D_0/c$; then $\det J(F_0)=c\,\gamma^2$;
\item \emph{middle ($L_2$-)branch}: $D_2\equiv c\neq0$, $X_1=-D_1/(2c)$, and
the \emph{master equation}
\begin{equation}\label{eq:master5}
4c\,D_0\;=\;D_1^{\,2} ;
\end{equation}
then $\det J(F_0)=c\,\gamma^3$.
\end{itemize}
The two branches behave very differently --- the bottom one is solved in
closed form and yields our counterexample, while the middle one turns out to
be rigid on a natural locus --- and we treat them in turn.

\subsubsection{The bottom branch: closed-form solution and the map \texorpdfstring{$F_6$}{F6}}
\label{subsubsec:bottom5}

Here $D_2=-\partial_3H\equiv0$ says that $H=G_3-2w_1G_2$ does not involve
$w_3$:
\[
G_3\;=\;2w_1G_2+\varphi(w_1,w_2),
\qquad\text{$\varphi$ a free bivariate polynomial.}
\]
The remaining condition $D_1\equiv c$ is then solved in closed form.
Substituting this expression for $G_3$ into the displayed formula for $D_1$
--- so that $\partial_3H=0$, $\partial_2H=\partial_2\varphi$,
$\partial_3G_3=2w_1\,\partial_3G_2$ and
$\partial_1G_3=2G_2+2w_1\partial_1G_2+\partial_1\varphi$ --- two of the four
terms cancel and $D_1$ collapses to
\[
D_1\;=\;(\partial_2\varphi)\,\partial_3G_4
\;+\;\bigl(2G_2+\partial_1\varphi\bigr)\,\partial_3G_2 :
\]
the total derivative
\[
 D_1=\partial_3\bigl((\partial_2\varphi)G_4+G_2^2+
 (\partial_1\varphi)G_2\bigr).
\]
This yields a complete polynomial description of the bottom branch.
\begin{proposition}\label{prop:bottom-classification}
For $c\ne0$, all polynomial bottom-branch data are obtained by choosing
$\tau\in\C^\times$, $\eta\in\C[w_1]$,
$G_2\in\C[w_1,w_2,w_3]$, and $\psi\in\C[w_1,w_2]$, and putting
\[
 G_3=2w_1G_2+\tau w_2+\eta(w_1),\qquad
 G_4=\frac{cw_3-G_2^2-\eta'(w_1)G_2}{\tau}+\psi.
\]
Then $X_1=-D_0/c$, with the remaining $X_i$ reconstructed from the
potentials, gives $\det J(F_0)=c\gamma^2$.
\end{proposition}
\begin{proof}
Integration of $D_1=c$ gives
\[
 (\partial_2\varphi)G_4+G_2^2+(\partial_1\varphi)G_2
 =cw_3+\chi(w_1,w_2).
\]
If $\partial_2\varphi=0$, completing the square makes a polynomial square
in $w_3$ equal to a polynomial of degree one in $w_3$, which is impossible.
If $\partial_2\varphi$ has an irreducible factor $q\in\C[w_1,w_2]$,
reduce modulo $q$ and complete the square in the domain
$(\C[w_1,w_2]/(q))[w_3]$. The same contradiction follows, since $c\ne0$
remains nonzero. Thus $\partial_2\varphi=\tau\in\C^\times$ and
$\varphi=\tau w_2+\eta(w_1)$. Solving the integrated identity for $G_4$
and writing $\chi=\tau\psi$ gives the formula. Conversely the formulas
give $D_2=0$ and $D_1=c$, and $X_1=-D_0/c$ imposes tangency.
\end{proof}

We now choose data with genuine, nonlinear dependence on all three
parameters:
\[
\tau=c=1,\qquad
G_2=w_2+w_1w_3+w_1^2-w_1^3,\qquad
\varphi=w_2,\qquad
\psi=2w_1^4-2w_1^5 ,
\]
so that $G_3=2w_1G_2+w_2$, $G_4=w_3-G_2^{\,2}+2w_1^4-2w_1^5$, and
$X_1=-D_0$. Explicitly,
\begin{align*}
X_1&=-w_3+2w_2-2w_1+2w_1w_3+5w_1^2-2w_1^3+8w_1^4-10w_1^5,\\
X_2&=w_2+2w_1w_2-w_1^2+2w_1^2w_3+4w_1^3-2w_1^4+8w_1^5-10w_1^6,\\
X_3&=w_2+2w_1w_2+w_1^2w_3+2w_1^2w_2+2w_1^3w_3+3w_1^4-2w_1^5+8w_1^6-10w_1^7,\\
X_4&=w_3-w_2w_3+w_2^2-2w_1w_2+3w_1^2w_2-w_1^2w_3^2-2w_1^3w_3\\
&\qquad +w_1^4+2w_1^4w_3+8w_1^4w_2-10w_1^5w_2-w_1^6 ,
\end{align*}
of degrees $5,6,7,6$. The sweep $S=X+\gamma\Delta$ satisfies the polynomial
identity $\det J(S)=\gamma$ (variables ordered $(\gamma,w_1,w_2,w_3)$), the
threefold $X$ has $\operatorname{rank}J(X)=3$ generically, and $w_3$ enters $X$
quadratically. These features concern the sweep parametrization and do
not establish a product obstruction for the twisted map. In fact the
resulting $F_6$ splits as $\Psi\times\id_{\A^1}$, as proved in
Theorem~\ref{prod:thm:F6-split}. The pure-$w_1$ tails
$t_2=w_1^2-w_1^3\subset G_2$ and $t_4=\psi=2w_1^4-2w_1^5$ serve a second
purpose: they arrange the base point that the twist below needs. Indeed, since
$S_1=X_1+\gamma$ and $S_{i+1}=G_{i+1}+\delta_{i+1}(X_1+\gamma)$, the base
condition $S(1,1,0,0)=0$ is equivalent to the four scalar conditions
\[
G_2(1,0,0)=G_3(1,0,0)=G_4(1,0,0)=0,\qquad
X_1(1,0,0)=-1\ \ \text{(i.e.\ }D_0(1,0,0)=c=1\text{)}.
\]
Because $G_3(1,0,0)=2\,G_2(1,0,0)$ and $G_4(1,0,0)=-G_2(1,0,0)^2+\psi(1,0)$,
the first three reduce to $t_2(1)=0$ and $t_4(1)=0$, which the displayed tails
satisfy; the last, $D_0(1,0,0)=1$, is a linear condition on the leading tail
coefficients and holds by direct evaluation (all verified exactly).

For the twist, assign the weights $(1,3,4)$ to $(w_1,w_2,w_3)$. Every monomial of
$X_i$ has weight $\ge i$ against the twist exponents $(1,2,3,4)$, so the
scalings
\[
w_1=\gamma u_1,\qquad w_2=\gamma^3u_2,\qquad w_3=\gamma^4u_3
\]
make $E_i:=S_i/\gamma^{\,i}$ a polynomial in $(\gamma,u_1,u_2,u_3)$ for
$i=1,2,3,4$, with $E_i(1,1,0,0)=0$ at the base point. The gradients there are
\[
\nabla E_2=(-16,-17,3,2),\qquad
\nabla E_3=(-17,-18,5,3),\qquad
\nabla E_4=(-2,-2,0,1).
\]
On $\C^5\ni(x,y,z_1,z_2,z_3)$ --- we name the last three source coordinates
$z_1,z_2,z_3$ to keep them clear of the sweep parameters --- put $v=(v_1,v_2,v_3,v_4)=(xy,\,x^2z_1,\,x^3z_2,\,x^4z_3)$
(of $x$-degrees $1,2,3,4$) and define the \emph{stage}, solving the
hierarchy \eqref{eq:side},
\[
\begin{aligned}
\gamma&=1-29v_1+999v_1^2+355v_1v_2-41553v_1^3+v_4,\\
u_1&=1+27v_1-5v_2,\\
u_2&=v_1-12v_2+\tfrac{2128}{5}v_1^2+v_3,\\
u_3&=-4v_1-10v_2 .
\end{aligned}
\]
Here the $v_1$-column $(-29,27,1,-4)$ spans the common kernel
$\ker\nabla E_2\cap\ker\nabla E_3\cap\ker\nabla E_4$, the $v_2$-column
$(0,-5,-12,-10)$ lies in $\ker\nabla E_3\cap\ker\nabla E_4$, the $v_3$- and
$v_4$-columns are $e_{u_2}$ and $e_\gamma$ with
$\nabla E_4\cdot e_{u_2}=0$, and the three correction vectors (the
coefficients of $v_1^2$, $v_1v_2$, $v_1^3$) are chosen inside
$\mathrm{span}(e_{u_2},e_\gamma)$ --- the span of the last two columns --- to
kill the obstructing Taylor coefficients of $E_3$ and $E_4$ along the
$v_1$-direction. Two consequences: first, $E_i$ composed with the stage is
divisible by $x^{\,i}$; second, because the corrections lie in the span of two
columns of the linear part, the stage is an \emph{elementary polynomial
automorphism} of $\C^4$ with constant Jacobian determinant $-290$ (each
correction's contribution to the Jacobian is a determinant with a repeated
column). With $w_1=\gamma u_1$, $w_2=\gamma^3u_2$, $w_3=\gamma^4u_3$ substituted
throughout, the sweep components are
\[
S_i\;=\;X_i(w_1,w_2,w_3)+\gamma\,\Delta_i(w_1,w_2)\quad(i=1,2,3,4),
\quad\text{i.e.}\quad
\begin{aligned}
S_1&=X_1+\gamma, & S_2&=X_2+\gamma\,w_1,\\
S_3&=X_3+\gamma\,w_1^2, & S_4&=X_4+\gamma\,w_2 ,
\end{aligned}
\]
with $X_1,\dots,X_4$ the polynomials displayed above. With $C=\gamma x$, define
\[
\boxed{\;F_6\;=\;\Bigl(C,\;\;\frac{S_1}{C},\;\;\frac{S_2}{C^2},\;\;
\frac{S_3}{C^3},\;\;\frac{S_4}{C^4}\Bigr)^{\mathsf T}\;:\;\C^5\to\C^5 .\;}
\]
As a composition diagram, in the format of \S\ref{subsec:twist},
\begin{multline*}
\begin{pmatrix}x\\ y\\ z_1\\ z_2\\ z_3\end{pmatrix}
\;\xrightarrow{\ \text{monomial}\ }\;
\begin{pmatrix}x\\ xy\\ x^2z_1\\ x^3z_2\\ x^4z_3\end{pmatrix}
\;\xrightarrow{\ \text{stage}\ }\;
\begin{pmatrix}x\\ \gamma\\ u_1\\ u_2\\ u_3\end{pmatrix}
\;\xrightarrow{\ \substack{w_1=\gamma u_1\\[1pt] w_2=\gamma^3u_2\\[1pt] w_3=\gamma^4u_3}\ }\;
\begin{pmatrix}x\\ \gamma\\ w_1\\ w_2\\ w_3\end{pmatrix}\\
\xrightarrow{\ \text{padded sweep}\ }\;
\begin{pmatrix}\gamma x\\ S_1\\ S_2\\ S_3\\ S_4\end{pmatrix}
\;\xrightarrow{\ \text{twist}\ }\;
\begin{pmatrix}C\\ S_1/C\\ S_2/C^2\\ S_3/C^3\\ S_4/C^4\end{pmatrix},
\end{multline*}
where the stage is the elementary polynomial automorphism displayed above
(over $x$), of constant Jacobian determinant $-290$; the Jacobian determinants
of the five arrows multiply to the constant in the chain factorization of the
proof of Theorem~\ref{thm:F6}.

\begin{theorem}\label{thm:F6}
$F_6$ is a polynomial map $\C^5\to\C^5$ with components of degrees
$7,38,40,42,44$; its Jacobian determinant is identically $-290$; and its
generic fiber consists of exactly $6$ points.
\end{theorem}

The proof below establishes two further properties: $F_6$ fixes the axis
$\{y=z_1=z_2=z_3=0\}$ pointwise, and for every $C_0\neq0$ the fiber over
$(C_0,0,0,0,0)$ consists of exactly $4$ points --- the visible one and three
escape companions. Theorem~\ref{prod:thm:F6-split} further identifies $F_6$ with
$\Psi\times\id_{\A^1}$ by explicit polynomial changes of coordinates.
The four-dimensional factor $\Psi$ has geometric degree six and cannot
split off another identity coordinate, even after stabilization.

\begin{proof}[Proof sketch]
The Jacobian statement is the chain factorization
\[
\det J(F_6)\;=\;
\underbrace{x^{10}}_{\text{monomials}}\cdot
\underbrace{(-290)}_{\text{stage}}\cdot
\underbrace{\gamma^{8}}_{\text{scalings }\gamma\cdot\gamma^3\cdot\gamma^4}\cdot
\underbrace{\gamma^{2}}_{\det J(F_0)=1\cdot\gamma^2}\cdot
\underbrace{(\gamma x)^{-10}}_{\text{twist exponents }1+2+3+4}
\;=\;-290 ,
\]
together with the divisibility statements above, all verified exactly. For the fibers, let
$Y\in\C^4$ be a sweep target. The system $S(\gamma,w)=Y$ \emph{triangularizes}:
$S_1=Y_1$ gives $\gamma=Y_1-X_1$; substituting into
$S_{i+1}=G_{i+1}+\delta_{i+1}(X_1+\gamma)$ turns the remaining three equations
into the potential equations
\[
G_2(w)=Y_2-w_1Y_1,\qquad G_3(w)=Y_3-w_1^2Y_1,\qquad G_4(w)=Y_4-w_2Y_1 ,
\]
By our choice of the data, $G_3=2w_1G_2+w_2$ and
$G_4=w_3-G_2^{\,2}+\psi(w_1)$ with $\psi=2w_1^4-2w_1^5$; hence
\[
w_2\;=\;G_3-2w_1G_2\;=\;\bigl(Y_3-w_1^2Y_1\bigr)-2w_1\bigl(Y_2-w_1Y_1\bigr)
\;=\;Y_3-2w_1Y_2+w_1^2Y_1
\]
and, inserting $G_2=Y_2-w_1Y_1$ and this value of $w_2$ (the mixed terms
cancel identically),
\[
w_3\;=\;Y_4-w_2Y_1+G_2^{\,2}-\psi(w_1)
\;=\;Y_4+Y_2^{\,2}-Y_1Y_3+2w_1^5-2w_1^4 .
\]
Substituting both expressions into the one equation not yet used,
$G_2\bigl(w_1,w_2(w_1),w_3(w_1)\bigr)=Y_2-w_1Y_1$, leaves the univariate \emph{fiber polynomial}
\[
R_Y(w_1)\;=\;2w_1^6-2w_1^5-w_1^3+(Y_1+1)\,w_1^2
+\bigl(Y_4+Y_2^{\,2}-Y_1Y_3-2Y_2+Y_1\bigr)\,w_1+\bigl(Y_3-Y_2\bigr)\;=\;0 ,
\]
of degree exactly $6$ with \emph{constant} leading coefficient $2$. Since the
stage is a polynomial automorphism and the monomial conjugation is bijective
wherever $x\neq0$, the fiber of $F_6$ over a target with first coordinate
$C_0\neq0$ is in bijection with the roots of $R_Y$ having $\gamma=Y_1-X_1\neq0$,
via $x=C_0/\gamma$. At the rational target $Y=(1,2,-1,3)$, exact gcd computations give
$\gcd(R_Y,R_Y')=\gcd(R_Y,\gamma)=1$. The nonvanishing of these resultants
persists on a Zariski open subset, proving generic fiber size six.
Over
$Y=0$ the fiber polynomial factors,
\[
R_0(w_1)\;=\;w_1^2\,(w_1-1)\,(2w_1^3-1),
\]
and $\gamma=-X_1$ vanishes at the double root $w_1=0$ (which is therefore
deleted by the twist) while $\gamma=1$ at $w_1=1$ and
$\gamma=3w_1^2-3w_1=3w_1(w_1-1)$ at the three roots of $2w_1^3=1$; since
$3w_1(w_1-1)$ vanishes only at $w_1\in\{0,1\}$, which are not roots of
$2w_1^3-1$, all four remaining points survive, and the fiber over
$(C_0,0,0,0,0)$ has exactly $4$ points --- the visible one
$(C_0,0,0,0,0)$ (note $S(1,1,0,0)=0$ makes $F_6$ fix the axis) and three
escape companions.
\end{proof}

The first component of $F_6$ is
\[
F_{6,1}\;=\;x-29x^2y+999x^3y^2-41553x^4y^3+355x^4yz_1+x^5z_3 ,
\]
and the remaining components have $342$, $421$, $507$ and $904$ terms
respectively; the accompanying exact-arithmetic script reconstructs them from the
compact data above. Unlike $F_4$ and $F_5$, whose geometric degrees ($5$ and
$10$) were established by elimination of comparable difficulty, here the
triangular collapse of the fiber system to the sextic $R_Y$ makes the
five-dimensional example the \emph{easiest} of the three to analyze --- a
direct payoff of the potentials $(G_2,G_3,G_4)$ adapted to the mixed direction
field.

\subsubsection{The middle branch: rigidity}\label{subsubsec:middle5}

Here $D_2\equiv c$ integrates to
\[
G_3\;=\;2w_1G_2-c\,w_3+\varphi(w_1,w_2),
\qquad\text{$\varphi$ a free bivariate polynomial,}
\]
$X_1=-D_1/(2c)$ is determined, and the data $(G_2,G_4,\varphi)$ remain
constrained by the nonlinear master equation \eqref{eq:master5}: unlike its
neighbour, this branch does not integrate in closed form. Introduce the
\emph{characteristic field} and its invariant coordinate
\[
\mathcal W\;=\;c\,\partial_2+(\partial_2\varphi)\,\partial_3 ,
\qquad
\theta\;=\;w_3-\varphi(w_1,w_2)/c
\qquad(\mathcal W\theta=0,\ \ \mathcal W w_1=0).
\]
The quadratic part of \eqref{eq:master5} in the derivatives of $G_4$ is the
rank-one form $(\mathcal W G_4)^2$, so \eqref{eq:master5} is a
Hamilton--Jacobi-type equation degenerate along $\mathcal W$.

\begin{theorem}[collapse on the characteristic-invariance locus]\label{thm:collapse5}
Let $c\neq0$, let $\varphi(w_1,w_2)$ be arbitrary, and suppose the datum $G_2$ is
constant along the characteristic field, $\mathcal WG_2=0$, i.e.\
$G_2=G_2(w_1,\theta)$. Set
\[
K(w_1,\theta)\;=\;c\,\frac{\partial G_2}{\partial w_1}\Big|_{\theta}
\;+\;\frac{\partial\,(G_2^{\,2})}{\partial\theta}.
\]
Then the master equation \eqref{eq:master5} factors as a perfect square,
\[
4c\,D_0-D_1^{\,2}\;=\;-\bigl(\mathcal WG_4-K\bigr)^2 ,
\]
so it linearizes to the transport equation $\mathcal WG_4=K$, whose general
polynomial solution is $G_4=(w_2/c)\,K(w_1,\theta)+h(w_1,\theta)$ with $h$
arbitrary. For every such solution the components $X_1,X_2,X_3$ of the swept
threefold depend on $(w_1,\theta)$ alone, $X_4=h(w_1,\theta)$, and in the
parameters $(\gamma,w_1,\theta,w_2)$ the sweep is the triangular suspension
\[
\begin{aligned}
S&=\bigl(T(\gamma,w_1,\theta),\ \gamma w_2+h(w_1,\theta)\bigr),\\
T&=\widehat X(w_1,\theta)+\gamma\,(1,w_1,w_1^2)^{\mathsf T},\qquad
\det J_{(\gamma,w_1,\theta)}(T)=-c\gamma:
\end{aligned}
\]
a four-dimensional bottom-branch sweep with direction field $(1,w_1,w_1^2)^{\mathsf T}$,
extended by one shear-type coordinate. This is a statement about the sweep. Product structure after a
polynomial twist requires a separate invariant or explicit coordinate change.
\end{theorem}

\begin{proof}
All statements are direct computations from the displayed formulas for
$D_0,D_1,D_2$. With $H=-cw_3+\varphi$ one finds
$\partial_2H=\partial_2\varphi$, $\partial_3H=-c$, and, using
$\partial_2G_2=-((\partial_2\varphi)/c)\partial_\theta G_2$,
$\partial_3G_2=\partial_\theta G_2$:
the coefficient of $\partial_1G_4$ in $D_0$ is
$-(c\,\partial_2G_2+(\partial_2\varphi)\partial_3G_2)=-\mathcal WG_2=0$; the remaining
two coefficients of $D_0$ collapse to $(K/c)\,\mathcal W$, giving
$D_0=(K/c)\,\mathcal WG_4$; and $D_1=\mathcal WG_4+K$. Hence
$4cD_0-D_1^2=4K\,\mathcal WG_4-(\mathcal WG_4+K)^2=-(\mathcal WG_4-K)^2$.
Since $\mathcal W$ annihilates functions of $(w_1,\theta)$ and
$\mathcal W\bigl((w_2/c)K\bigr)=K$, the transport equation has the stated general
solution. Finally $X_1=-D_1/(2c)=-K/c$ depends on $(w_1,\theta)$ only, hence so
do $X_2=G_2+w_1X_1$ and
$X_3=G_3+w_1^2X_1=2w_1G_2+\varphi-cw_3+w_1^2X_1=2w_1G_2-c\theta+w_1^2X_1$, while
$X_4=G_4+w_2X_1=h$; and $\partial S_4/\partial w_2=\gamma$ with
$S_1,S_2,S_3$ independent of $w_2$ gives
$\det J_{(\gamma,w_1,\theta,w_2)}S=\gamma\det J(T)$.
The change from $(w_1,w_2,w_3)$ to $(w_1,\theta,w_2)$ has determinant
$-1$, so the left side is $-c\gamma^2$ and $\det J(T)=-c\gamma$.
\end{proof}

\begin{remark}\label{rem:middle-scope}
Theorem~\ref{thm:collapse5} is restricted to $\mathcal WG_2=0$.
It does not assert rigidity of the full middle branch, nor does a
triangular suspension at sweep level imply that every subsequent
polynomial twist splits an identity coordinate.
\end{remark}

\begin{problem}[twists on the middle branch]\label{prob:rigid5}
Which polynomial solutions of \eqref{eq:master5} admit stage data
satisfying all conditions of \S\ref{subsec:stage}? Among the resulting
Keller maps, which cannot split off an identity coordinate under
polynomial left-right equivalence, even after stabilization?
\end{problem}

\subsection{The curve-type family: a uniform solution in every dimension}
\label{subsec:curvefam}

We now return to the first direction field of the classification above and
treat it uniformly for all $n\ge3$:
\[
\Delta\;=\;\bigl(1,\,w_1,\,w_1^2,\,\dots,\,w_1^{\,n-2}\bigr)^{\mathsf T},
\]
the rational normal curve, depending on the single parameter $w_1$. Here the
normalization system admits a \emph{complete} solution, and the solutions for
$n=3,4,5$ fall into a single pattern.

\begin{theorem}[reduction for the curve-type field]\label{thm:curvered}
Let $n\ge3$ and $\Delta=(1,w_1,\dots,w_1^{\,n-2})^{\mathsf T}$. With the
potentials $\Gamma=(G_2,\dots,G_{n-1})$ of \S\ref{subsec:forms}, set
\[
H_{i+1}\;:=\;G_{i+1}-i\,w_1^{\,i-1}G_2\qquad(2\le i\le n-2).
\]
Then $\widetilde P=\widehat\Delta{}'\,e_1^{\mathsf T}$ is of rank one, with
$\widehat\Delta{}'=(1,2w_1,\dots,(n-2)w_1^{\,n-3})^{\mathsf T}$, the pencil
is linear in $\mu$,
\[
\det\bigl(J(\Gamma)+\mu\widetilde P\bigr)\;=\;D_0+\mu D_1,
\qquad
D_0=\det J(\Gamma),
\qquad
D_1\;=\;\det J_{(w_2,\dots,w_{n-2})}\bigl(H_3,\dots,H_{n-1}\bigr),
\]
and the (unique) branch normalization $D_1\equiv c\neq0$ says precisely:
\begin{itemize}
\item[(i)] the family
$H(w_1;\,\cdot\,)=(H_3,\dots,H_{n-1})^{\mathsf T}\colon\C^{n-3}\to\C^{n-3}$ has constant
Jacobian determinant $c$ in the transverse parameters $(w_2,\dots,w_{n-2})$,
for every value of the spectator parameter $w_1$;
\item[(ii)] $G_2\in\C[w]$ is completely free.
\end{itemize}
Given such data, $X_1=-D_0/c$ and $X_{i+1}=G_{i+1}+w_1^{\,i}X_1$ yield
$\det J(F_0)=c\,\gamma^2$.
\end{theorem}

\begin{proof}
Since $\Delta$ depends on $w_1$ alone, $J(\Delta)$ has only its first column
nonzero and $\widetilde P=N\,J(\Delta)=\widehat\Delta{}'\,e_1^{\mathsf T}$;
a rank-one summand makes the pencil determinant linear
in $\mu$, with $\mu$-coefficient $D_1=\det(\widehat\Delta{}',\partial_2\Gamma,\dots,
\partial_{n-2}\Gamma)$. Subtracting $i\,w_1^{\,i-1}$ times the first row from
the $i$-th row ($2\le i\le n-2$) reduces the column $\widehat\Delta{}'$ to $e_1$; these row
operations commute with $\partial_j$ for $j\ge2$ because the multipliers
depend on $w_1$ alone, and they replace $\partial_jG_{i+1}$ by
$\partial_jH_{i+1}$. Expanding along the first column leaves the stated
$(n-3)\times(n-3)$ Jacobian. Finally $X_1=-D_0/c$ gives
$\det J(S)=D_0+(X_1+\gamma)c=c\gamma$, and padding contributes the extra
$\gamma$. The determinant and row-operation identities above are polynomial identities.
\end{proof}

\begin{remark}
The frame viewpoint of \S\ref{subsec:n4} is available here too: the collection
$(\Delta,\Delta',\dots,\Delta^{(n-2)})$ of \emph{all} independent partial
derivatives of the rational normal field is a frame of constant volume
$\prod_{k=0}^{n-2}k!$ ($=2$ for $n=4$ and $12$ for $n=5$, the values already
met), and expanding the columns of $J(X)$ in it --- the coefficient of
$\Delta$ being $\partial_jX_1$, as always --- reproduces the pencil above; the
row reduction in the proof is precisely the triangular change between the
frame coefficients and the derivatives of the potentials. We phrase the
theorem through the potentials because the construction ultimately
\emph{integrates} the data: $G_2$ enters the solution undifferentiated, and
the fiber theory of Proposition~\ref{prop:curvefib} below is written directly
in terms of $\Gamma$.
\end{remark}

The pattern for small $n$ is now transparent. For $n=3$ there are no
transverse parameters: the condition (i) is empty ($D_1=1$, $c=1$), $G_2(w_1)$
is the only datum, and $X=(-G_2',\,G_2-w_1G_2')^{\mathsf T}$ sweeps the envelope of the
line family $\{Y_2=w_1Y_1+G_2(w_1)\}$ --- the classical Clairaut picture, with
$G_2$ the support function; the whole of \S\ref{sec:sweep} is the case $n=3$.
For $n=4$ the condition reads $\partial_2H_3=c$, i.e.\ $H_3=cw_2+\varphi(w_1)$;
taking $G_2$ \emph{linear} in $w_2$ recovers exactly the ruled surfaces of
Specialization~I, while a nonlinear $G_2$ gives non-ruled solutions that the
ruled ansatz of \S\ref{subsec:specI} does not see. For $n=5$ the condition
asks for a $w_1$-family of \emph{unit-Jacobian pairs}
$(H_3,H_4)$ in $(w_2,w_3)$ --- the first case where (i) is a genuine
Keller-type condition on the data.

\begin{proposition}[univariate fiber theory]\label{prop:curvefib}
Assume in addition that each $H(w_1;\,\cdot\,)$ is a polynomial automorphism
of $\C^{n-3}$ whose inverse depends polynomially on $w_1$ (e.g.\ any
triangular family). Then for every target $Y\in\C^{n-1}$ the fiber system
$S(\gamma,w)=Y$ triangularizes: $\gamma=Y_1-X_1$, the potential equations
become $G_{i+1}(w)=Y_{i+1}-w_1^{\,i}Y_1$, hence
\[
H\bigl(w_1;\,w_2,\dots,w_{n-2}\bigr)\;=\;h(w_1),
\qquad
h_{i+1}\;=\;Y_{i+1}-i\,w_1^{\,i-1}Y_2+(i-1)\,w_1^{\,i}Y_1 ,
\]
so $(w_2,\dots,w_{n-2})=H^{-1}(w_1;h(w_1))$ are polynomial functions of $w_1$,
and the last remaining equation is the univariate \emph{fiber polynomial}
\[
R_Y(w_1)\;:=\;G_2\bigl(w_1,\,H^{-1}(w_1;h(w_1))\bigr)-\bigl(Y_2-w_1Y_1\bigr)
\;=\;0 .
\]
The fibers of the twisted map over targets with nonzero first coordinate are
in bijection with the roots of $R_Y$ at which $\gamma\neq0$. For $n=3$, $R_Y$
is the tangency polynomial of \S\ref{sec:sweep}.
\end{proposition}

\begin{proof}
The displayed triangular substitutions solve all but the second sweep
equation, which is exactly $R_Y=0$. Each root determines a unique
$(w_2,\ldots,w_{n-2},\gamma)$; Lemma~\ref{lem:twist-fibers} gives the
last assertion on the invertible twist chart.
\end{proof}

\begin{lemma}[simple-root criterion for the curve-type field]\label{lem:curve-derivative}
In Proposition~\ref{prop:curvefib}, write $t=w_1$,
$z=(w_2,\ldots,w_{n-2})$, $z(t)=H^{-1}(t;h(t))$, and
$\gamma_Y(t)=Y_1-X_1(t,z(t))$. Then
\[
 \gamma_Y(t)\equiv R_Y'(t)\pmod{R_Y(t)}.
\]
Thus a root survives the twist exactly when it is simple.
\end{lemma}
\begin{proof}
Put $Q_i=G_{i+1}+t^iY_1-Y_{i+1}$ for $1\le i\le n-2$.
Eliminating the first sweep coordinate in its Jacobian gives
$\det J_{(t,z)}Q=c\gamma_Y$ after substitution of $z(t)$.
For $i\ge2$,
$Q_i-i t^{i-1}Q_1=H_{i+1}-h_{i+1}(t)$.
Subtracting the corresponding multiples of the first differential row
replaces the lower rows by $d(H-h)$, up to terms proportional to $Q_1$
in the first column. Along $z(t)$, $Q_1=R_Y$ and $H(t,z(t))=h(t)$.
The chain rule and $\det H_z=c$ therefore give
$c\gamma_Y=cR_Y'+R_Y B_Y$ for a polynomial $B_Y$.
Dividing by $c\ne0$ proves the congruence, and evaluating at a root
proves the last assertion.
\end{proof}

\begin{remark}[higher-dimensional sweep data]
Condition (i) of Theorem~\ref{thm:curvered} asks for a $w_1$-family of
constant-Jacobian --- Keller-type --- maps of $\C^{n-3}$. Any family of
(tame) polynomial automorphisms qualifies, and then
Proposition~\ref{prop:curvefib} applies. But for $n\ge6$ one may instead take
$H$ to be a \emph{constant} family equal to a Keller counterexample in $n-3$
variables: Alp\"oge's $F$ for $n=6$, the maps $F_4$ or $F_5$ for $n=7$, $F_6$
or $F_7$ for $n=8$, and so on --- the construction feeds on its own output,
three dimensions at a time. (The sweep-level statement is unconditional;
carrying out the monomial twist requires the same finitely many side
conditions as in the examples above.) For such non-invertible $H$ the fiber
system is no longer triangular, and its geometry is unexplored.
\end{remark}

\medskip\noindent\textbf{An explicit five-dimensional example.} Take $n=5$,
$c=1$, and the data
\[
H_3=w_2+w_1^3+w_1^2w_3,\qquad
H_4=w_3+w_1^4,\qquad
G_2=w_2^{\,2}+w_2+w_1w_3+w_1^3 ,
\]
so that $\det J_{(w_2,w_3)}(H_3,H_4)\equiv1$, the pair
$H(w_1;\,\cdot\,)$ is a triangular automorphism for every $w_1$, and $G_2$ is
genuinely nonlinear in the transverse parameters. Then $G_3=2w_1G_2+H_3$,
$G_4=3w_1^2G_2+H_4$, $X_1=-\det J(G_2,G_3,G_4)$, and the components
$X_1,\dots,X_4$ (degrees $7,8,9,10$) are reconstructed as in
Theorem~\ref{thm:curvered}; the sweep $S=X+\gamma\Delta$ satisfies
$\det J(S)=\gamma$ and passes through the origin at the base point,
$S(1,1,0,-1)=0$, with transversality constant $D_0(1,0,-1)=1$ (the pure-$w_1$
terms of $H_3$, $H_4$, $G_2$ are tuned for exactly this). Assign the weights
$(1,3,4)$ to $(w_1,w_2,w_3)$; every monomial of $X_i$ has weight $\ge i$, so
the scalings $w_1=\gamma u_1$, $w_2=\gamma^3u_2$, $w_3=\gamma^4u_3$ make
$E_i=S_i/\gamma^{\,i}$ polynomial, with base point
$(\gamma,u_1,u_2,u_3)=(1,1,0,-1)$ and gradients
\[
\nabla E_2=(-15,0,-3,4),\qquad
\nabla E_3=(-20,3,-1,6),\qquad
\nabla E_4=(-19,8,-1,7).
\]
On $\C^5\ni(x,y,z_1,z_2,z_3)$ with $v=(xy,\,x^2z_1,\,x^3z_2,\,x^4z_3)$ the
stage is
\[
\begin{aligned}
\gamma&=1-61v_1+9012v_1^2+\tfrac{238}{19}v_1v_2-\tfrac{35823670}{19}v_1^3+8v_3+v_4,\\
u_1&=1+59v_1+v_2-5178v_1^2+19v_3,\\
u_2&=-7v_1-27v_2,\\
u_3&=-1-234v_1-5v_2 ,
\end{aligned}
\]
built by the same rules as for $F_6$: the $v_1$-column $(-61,59,-7,-234)$
spans the common kernel of the three gradients, the $v_2$-column
$(0,1,-27,-5)$ lies in $\ker\nabla E_3\cap\ker\nabla E_4$, the $v_3$- and
$v_4$-columns $(8,19,0,0)$ and $e_\gamma$ satisfy $\nabla E_4\cdot(8,19,0,0)=0$,
and all correction vectors lie in the span of the last two columns, so the
stage is an elementary polynomial automorphism of constant Jacobian
determinant $119377$. With $C=\gamma x$ and
$S_i=X_i(w_1,w_2,w_3)+\gamma\,w_1^{\,i-1}$ substituted throughout, define
\[
\boxed{\;F_7\;=\;\Bigl(C,\;\;\frac{S_1}{C},\;\;\frac{S_2}{C^2},\;\;
\frac{S_3}{C^3},\;\;\frac{S_4}{C^4}\Bigr)^{\mathsf T}\;:\;\C^5\to\C^5 .\;}
\]

\begin{theorem}\label{thm:F7}
$F_7$ is a polynomial map $\C^5\to\C^5$ with components of degrees
$7,86,89,92,95$; its Jacobian determinant is identically $119377$; and its
generic fiber consists of exactly $12$ points.
\end{theorem}

The determinant is the chain factorization
$x^{10}\cdot119377\cdot\gamma^{8}\cdot\gamma^{2}\cdot(\gamma x)^{-10}$, exactly
as for $F_6$. As before, the proof yields more: $F_7$ fixes the axis
$\{y=z_1=z_2=z_3=0\}$ pointwise, and for $C_0\neq0$ the fiber over
$(C_0,0,0,0,0)$ consists of exactly $7$ points. In particular $F_7$ is a
five-dimensional non-injective Keller map of geometric degree twelve.
Theorem~\ref{prod:thm:F7-complete} excludes every four-dimensional factor,
including degree twelve, even after stabilization.

\begin{proof}[Proof sketch]
The Jacobian and divisibility statements are verified exactly as for $F_6$. For the
fibers, Proposition~\ref{prop:curvefib} gives the values of $w_2$ and $w_3$
along the fiber in closed form: writing them as $p_1(w_1)$ and $p_2(w_1)$,
\[
p_2(w_1)=Y_4-3Y_2w_1^2+2Y_1w_1^3-w_1^4,\qquad
p_1(w_1)=Y_3-2Y_2w_1+Y_1w_1^2-w_1^3-w_1^2\,p_2(w_1),
\]
\[
R_Y(w_1)\;=\;p_1(w_1)^2+p_1(w_1)+w_1\,p_2(w_1)+w_1^3+Y_1w_1-Y_2 ,
\]
of degree exactly $12$ with constant leading coefficient $1$, obtained by
substitution in the second sweep equation. At the rational witness target $Y=(1,2,-1,3)$ the polynomial
$R_Y$ is squarefree and coprime to $\gamma(w_1)=Y_1-X_1(w_1,p_1(w_1),p_2(w_1))$
--- both certified by exact univariate gcd computations --- so that fiber has
exactly $12$ points, and by openness of both conditions the generic count is
$12$. Over $Y=0$,
\[
R_0(w_1)\;=\;w_1^{\,5}\,(w_1-1)\,
\bigl(w_1^6+w_1^5+w_1^4-w_1^3-w_1^2-w_1+1\bigr),
\]
the quintuple root $w_1=0$ has $\gamma=0$ and is deleted by the twist, while
$\gamma(1)=1$ and $\gcd(\gamma,R_0/w_1^5)=1$ exactly, so all seven remaining
roots survive: the fiber over $(C_0,0,0,0,0)$ has exactly $7$ points,
including the visible fixed point $(C_0,0,0,0,0)$.
\end{proof}

The first component is
\[
F_{7,1}\;=\;x-61x^2y+9012x^3y^2-\tfrac{35823670}{19}x^4y^3
+\tfrac{238}{19}x^4yz_1+8x^4z_2+x^5z_3 ,
\]
and the remaining components have between $14{,}000$ and $26{,}000$ terms;
the accompanying verification script reconstructs them from the compact data.
Together, $F_6$ and $F_7$ illustrate the two mechanisms available in dimension
five: a mixed direction field on its flexible bottom branch, and the
curve-type field with transverse Keller-type data.

\clearpage
\part{Equivalence and obstructions}\label{part:equivalence}
This part proves that $F_4,F_5,F_7$ cannot split off an identity
coordinate, even after further stabilization, and gives an explicit
splitting $F_6\sim\Psi\times\id_{\A^1}$ with a four-dimensional factor
$\Psi$ that cannot split further, even stably. The proofs use omitted
loci, relative normalization, and intrinsic fiber strata. The required
sweep and coordinate data are repeated where they enter the calculations.
The maps are defined in Sections~\ref{sec:deg4}, \ref{subsec:specI},
\ref{subsec:specII}, \ref{subsubsec:bottom5}, and \ref{subsec:curvefam};
Theorem~\ref{thm:product-summary} states the principal conclusions.

\section{General methods for equivalence}\label{prod:sec:theory}
This section develops the general methods used in the examples. The central
idea is to attach algebraic sets to a map in a way that is preserved by
coordinate changes and by adjoining identity coordinates. A component of
one of these sets can then obstruct a product decomposition. We first use
omitted targets, and then introduce a finite cover that retains inverse
branches missing from the original affine source.

\subsection{Equivalence, geometric degree, and stabilization}
Work over $\C$. An affine variety is taken with its reduced structure;
scheme-theoretic fibers, introduced below, retain their nilpotents.
Throughout, ``open'', ``closed'', and ``closure'' refer to the Zariski
topology on the indicated ambient variety, unless stated otherwise.
Closures are endowed with their reduced induced structure.
Write $\A^n=\C^n$ and $\C^\times=\C\setminus\{0\}$.
A polynomial automorphism is a polynomial map with a polynomial inverse.
For polynomial maps $F,H:\A^n\to\A^n$, write
\[
 F\sim H\quad\Longleftrightarrow\quad
 F=\alpha\circ H\circ\beta
 \quad\text{for polynomial automorphisms }\alpha,\beta.
\]
This is \emph{left-right equivalence}: source and target coordinates may
be changed independently. We write $F\nsim H$ for its negation.

For $\Phi:\A^d\to\A^d$, its stabilization by $m$ identity coordinates is
\[
 (\Phi\times\id_{\A^m})(u,z)=(\Phi(u),z).
\]
Saying that $F$ \emph{splits off an identity coordinate} means
$F\sim\Phi\times\id_{\A^1}$ for some $\Phi$ in one lower dimension.
The stronger exclusion used here is
\begin{equation}\label{prod:eq:stable-exclusion}
 F\times\id_{\A^r}\nsim\Phi\times\id_{\A^{r+1}}
 \quad\text{for every }r\geq0\text{ and every such }\Phi.
\end{equation}
This is what ``cannot split, even after stabilization'' means in this part;
it does not refer to indecomposability under composition of maps.

A polynomial map $F:\A^n\to\A^n$ is a tuple
$F=(F_1,\ldots,F_n)$, where
$F_i\in\C[x_1,\ldots,x_n]$ for each $i=1,\ldots,n$; it sends
$x=(x_1,\ldots,x_n)$ to $(F_1(x),\ldots,F_n(x))$.
The map is \emph{dominant} if its image is Zariski dense in the target.
For such a dominant polynomial map, pullback of target coordinates gives
an injective homomorphism
\[
 F^*:\C[v_1,\ldots,v_n]\hookrightarrow\C[x_1,\ldots,x_n],
 \qquad v_i\longmapsto F_i.
\]
Indeed, a polynomial in the kernel vanishes on the dense image of $F$,
and hence on the entire target, so is the zero polynomial. In particular,
$F_1,\ldots,F_n$ are algebraically independent over $\C$.
Passing to fields of fractions gives the inclusion
\[
 \C(F_1,\ldots,F_n)\subset\C(x_1,\ldots,x_n).
\]
This is a finite extension: both fields have transcendence degree $n$, and
the larger is finitely generated over the smaller. The map is therefore
\emph{generically finite}. Its geometric degree is the degree of this
extension; in characteristic zero it is also the number of points in a
general fiber. A nonzero constant Jacobian implies dominance, since the
complex inverse function theorem gives an analytically open image near
each point, and such an image is Zariski dense. All the maps studied here
are thus dominant and generically finite.

Left-right equivalence gives isomorphic field extensions, so preserves
geometric degree. Adjoining identity coordinates merely adjoins independent
transcendentals to both fields, preserving that degree as well. Degree
therefore separates examples of different degrees, but gives no obstruction
when the degrees agree. The following constructions address that case.

\subsection{Omitted targets and the obstruction from invertible functions}
Define the omitted locus as a set of target points
\[
 \Miss(F)=\A^n\setminus F(\A^n).
\]
This is the locus of empty fibers. It is different from the nonproper-value
locus: a value may have some preimages while other inverse branches escape
to infinity. By Chevalley's theorem, $F(\A^n)$ is constructible,
and so is its complement $\Miss(F)$. Here constructible means a finite
union of Zariski locally closed subsets, each an intersection of a
Zariski open subset and a Zariski closed subset. We do not assume that
$\Miss(F)$ is Zariski closed for an arbitrary polynomial map.
We write $\overline{\Miss(F)}$ for its Zariski closure in the target
$\A^n$, the smallest Zariski closed subset containing it. The notation
$\Miss(F)$ always denotes the actual omitted locus, not its closure;
these two sets are identified only after closedness has been proved.

If $F=\alpha\circ H\circ\beta$, surjectivity of $\beta$ gives
$\Miss(F)=\alpha(\Miss(H))$. Also, direct inspection of the fibers gives
\begin{equation}\label{prod:eq:missing-product}
 \Miss(\Phi\times\id_{\A^m})=\Miss(\Phi)\times\A^m.
\end{equation}
Since a polynomial automorphism is a homeomorphism for the Zariski
topology, the corresponding closure formulas are
\[
 \overline{\Miss(F)}=\alpha\bigl(\overline{\Miss(H)}\bigr),\qquad
 \overline{\Miss(\Phi\times\id_{\A^m})}
 =\overline{\Miss(\Phi)}\times\A^m.
\]
The closures are taken in the respective target affine spaces.
One way to check the product assertion is to expand a polynomial in the added variables:
if it vanishes on $S\times\A^m$, every coefficient vanishes on $S$ and
hence on $\overline S$.

For an affine variety $X$, the notation $\C[X]$ denotes its
\emph{coordinate ring}, equivalently the ring of globally regular
algebraic functions $X\to\C$, with pointwise addition and multiplication.
It does not mean the ring of all set-theoretic, continuous, or holomorphic
functions on $X$. For $X\subset\A^N$ Zariski closed, this ring is
$\C[t_1,\ldots,t_N]/I(X)$, where $I(X)$ is the ideal of polynomials
vanishing on $X$.

A \emph{unit} of the coordinate ring of an affine variety is a regular
function with a regular multiplicative inverse. An isomorphism of affine
varieties induces an isomorphism of their coordinate rings and preserves
units and constant functions. The only units of
$\C[z_1,\ldots,z_m]$ are the nonzero constants: if $fg=1$, the additivity
of total degree forces both polynomials to have degree zero. By contrast, for $m\geq1$,
\[
 \C[\C^\times\times\A^{m-1}]
 =\C[c,c^{-1},z_1,\ldots,z_{m-1}].
\]
Here $c$ is the nonzero coordinate on $\C^\times$, and $z_1,\ldots,z_{m-1}$
are the affine coordinates on the second factor. To see the ring explicitly,
realize this product as the Zariski closed affine variety $ct=1$ in
$\A^{m+1}$, with coordinates $(c,t,z_1,\ldots,z_{m-1})$. Its coordinate
ring is
\[
 \C[c,t,z_1,\ldots,z_{m-1}]/(ct-1)
 \simeq\C[c,c^{-1},z_1,\ldots,z_{m-1}],
\]
where $t$ corresponds to $c^{-1}$. Thus $c$ is a nonconstant regular
function with the regular inverse $c^{-1}$. When $m=1$, there are no
$z$-variables and $\A^0$ is a point.
If this variety were isomorphic to $\A^m$, the induced $\C$-algebra
isomorphism would send $c$ to a unit in $\C[z_1,\ldots,z_m]$, hence
to a nonzero constant. Since the isomorphism fixes constants and is
injective, that would force $c$ itself to be constant, a contradiction.
Therefore $\C^\times\times\A^{m-1}$ is not isomorphic to $\A^m$.
This elementary distinction will be used repeatedly.

\begin{lemma}\label{prod:lem:obstruction}
Let $F:\A^n\to\A^n$ be a polynomial map and let $1\leq m\leq n$.
Suppose $\Miss(F)$ is nonempty and Zariski closed in $\A^n$, of dimension $m$.
If $F\sim\Phi\times\id_{\A^m}$ for a polynomial map
$\Phi:\A^{n-m}\to\A^{n-m}$, then $\Miss(F)$ is a finite disjoint
union of copies of $\A^m$. In particular, such an equivalence is impossible
if one irreducible component is isomorphic to
$\C^\times\times\A^{m-1}$.
\end{lemma}
\begin{proof}
Put $d=n-m$ and $S=\Miss(\Phi)\subset\A^d$. By the definition of
left-right equivalence, there are polynomial automorphisms $\alpha,\beta$
of $\A^n$ such that
$F=\alpha\circ(\Phi\times\id_{\A^m})\circ\beta$.
For a target point $(v,z)\in\A^d\times\A^m$, the product map has a
preimage exactly when there is a $u\in\A^d$ with $\Phi(u)=v$:
the identity coordinates force the other source coordinates to equal $z$.
Thus its omitted locus is $S\times\A^m$. Since $\beta$ is surjective
and $\alpha$ is bijective, taking complements of images gives
\begin{equation}\label{prod:eq:lemma-omitted-isomorphism}
 \alpha^{-1}(\Miss(F))=S\times\A^m.
\end{equation}

We first show that $S$ is itself an algebraic set. The left side of
\eqref{prod:eq:lemma-omitted-isomorphism} is closed because $\Miss(F)$ is
closed and $\alpha$ is an automorphism. The inclusion of the zero slice,
$i:\A^d\to\A^d\times\A^m$, $i(v)=(v,0)$, is a morphism, and
\[
 S=i^{-1}(S\times\A^m).
\]
Consequently $S$ is closed. It is nonempty because $\Miss(F)$ is nonempty.

Next we determine its dimension. If $S_j$ is an irreducible component
of $S$ with coordinate ring $R_j$, then
$S_j\times\A^m$ has coordinate ring $R_j[z_1,\ldots,z_m]$.
This is a domain, and its fraction field has transcendence degree
$\dim S_j+m$ over $\C$. Thus
$\dim(S_j\times\A^m)=\dim S_j+m$.
These products are precisely the irreducible components of
$S\times\A^m$: they are irreducible closed sets covering it, and none
is contained in another, since projection to $S$ would imply the same
containment between components of $S$.
Taking the maximum of their dimensions and using the isomorphism
\eqref{prod:eq:lemma-omitted-isomorphism}, we obtain
\[
 m=\dim\Miss(F)=\dim(S\times\A^m)=\dim S+m.
\]
Hence $\dim S=0$.

For completeness, a zero-dimensional affine algebraic set over $\C$
is a finite set of points. Indeed, it has finitely many irreducible
components because its coordinate ring is Noetherian. Each component
has a finitely generated domain $R_j$ of dimension zero as its coordinate
ring. The zero prime is then maximal, so $R_j$ is a field. By the weak
Nullstellensatz this field is $\C$, and hence the component is a single
point. We can therefore write $S=\{p_1,\ldots,p_q\}$ with $q\geq1$.
The sets $\{p_j\}\times\A^m$ are pairwise disjoint, irreducible, and
both closed and open in $S\times\A^m$; each is isomorphic to $\A^m$.
Transporting them by $\alpha$ proves that $\Miss(F)$ is a finite
disjoint union of copies of $\A^m$, and that these copies are exactly
its irreducible components.

Finally, an irreducible component isomorphic to
$\C^\times\times\A^{m-1}$ would therefore also be isomorphic to
$\A^m$. The coordinate-ring unit argument above rules this out,
proving the last assertion.
\end{proof}

The dimension hypothesis matters. A component $\C^\times\times\A^1$
can obstruct splitting off two identity coordinates by this argument,
but is itself a product with one affine line. This explains why the
omitted locus of $F_6$ alone could not exclude a four-dimensional factor.

\subsection{Normalization of the target in the source function field}\label{prod:sec:finite}
Let $F:\A^n_x\to\A^n_v$ be dominant and generically finite. Put
\[
 R=\C[v_1,\ldots,v_n],\qquad A=\C[x_1,\ldots,x_n],\qquad
 K=\operatorname{Frac}(R),\qquad L=\operatorname{Frac}(A),
\]
where $R\hookrightarrow A$ means $v_i\mapsto F_i(x)$. Here
$\operatorname{Frac}$ is the field of fractions. Let $B$ be the integral
closure of $R$ in $L$: an element $b\in L$ belongs to $B$ exactly when
it satisfies a monic equation
\[
 b^d+r_1b^{d-1}+\cdots+r_d=0,\qquad r_i\in R.
\]
Thus $B$ contains the inverse-field functions that are algebraic without
poles over the target in the sense of integrality. Set
\[
 N_F=\operatorname{Spec}(B),\qquad
 \pi_F:N_F\longrightarrow\operatorname{Spec}(R)=\A^n_v.
\]
For the reader using classical affine varieties, $\operatorname{Spec}(B)$
here may be read as the affine variety with coordinate ring $B$; the ring
notation additionally records scheme-theoretic fibers.

The morphism $\pi_F:N_F\to\A^n_v$ is the \emph{normalization of the
target in the source function field $L$}. This is an instance of
\emph{relative normalization} \cite{StacksRelativeNormalization}: the
inclusion $R\hookrightarrow L$ is the one induced by $F$, so the
construction depends on the map. The qualification ``in $L$'' is essential.
The target $\A^n_v$ is already normal, and its normalization in its own
function field $K$ is itself, since $R$ is integrally closed in $K$.
Here we instead take the integral closure in the finite extension $L/K$,
which can produce a nontrivial finite cover of the target.

In this part we use \emph{canonical finite cover} as shorthand for this
normalization morphism. Its finiteness and surjectivity are established
in the following proposition. ``Cover'' means a finite surjective morphism,
not an everywhere unramified topological covering.

\begin{proposition}\label{prod:prop:finite-cover}
The ring $B$ is a finite $R$-module, is integrally closed, and has fraction
field $L$. The original polynomial map factors as
\[
 \A^n_x\xrightarrow{\ j_F\ }N_F\xrightarrow{\ \pi_F\ }\A^n_v,
 \qquad F=\pi_F\circ j_F,
\]
where $j_F$ is birational. In general $j_F$ need not be surjective, and
fibers of $\pi_F$ must not be identified with fibers of $F$.
\end{proposition}
\begin{proof}
We give a finiteness argument in this setting. Since $L/K$ is finite and
has characteristic zero, it is separable. Choose a $K$-basis
$b_1,\ldots,b_d$ of $L$ consisting of elements integral over $R$.
Such a basis is obtained from any basis by multiplying each element by
a suitable nonzero element of $R$ to clear denominators in its monic
equation over $K$.

The trace matrix $T=(\operatorname{Tr}_{L/K}(b_i b_j))_{i,j}$ has entries
in $R$: the trace of an integral element is integral over $R$ and belongs
to $K$, and the polynomial ring $R$ is integrally closed in $K$.
Separability makes the trace pairing nondegenerate, so
$\delta=\det T\in R\setminus\{0\}$. For $b\in B$, write
$b=\sum_j a_jb_j$ with $a_j\in K$. Taking traces against all $b_i$
gives a linear system with matrix $T$ and right side in $R^d$.
Cramer's rule gives $a_j\in\delta^{-1}R$. Therefore
\[
 \sum_j Rb_j\ \subseteq\ B\ \subseteq\ \sum_j\delta^{-1}Rb_j.
\]
The last module is finite over the Noetherian ring $R$, so its submodule
$B$ is finite. The included basis shows $\operatorname{Frac}(B)=L$.
If an element of $L$ is integral over $B$, transitivity of integrality
makes it integral over $R$, hence it belongs to $B$. This proves normality.

Finally, an element of $B$ is integral over $R\subset A$, so its same
monic equation makes it integral over $A$. Since $A$ is integrally closed
in $L$, it belongs to $A$. Thus $R\subset B\subset A$, giving the two
morphisms and their factorization. The fraction fields of $B,A$ agree,
which is precisely the asserted birationality of $j_F$. Integrality and
the lying-over theorem make $\pi_F$ surjective. They do not make $j_F$
surjective.
\end{proof}

For a simple illustration, the map $(x,y)\mapsto(x,xy)$ is birational,
so its canonical finite cover is the identity of the target plane.
The original map omits $(0,1)$, whereas the finite cover has a point above
it. This illustrative map is not Keller; its purpose is to distinguish the
two kinds of fibers. In our Keller examples, the additional points of the
finite cover similarly retain inverse branches absent from the affine source.

\subsection{Why the finite cover respects equivalence and products}
The construction uses the inclusion $R\hookrightarrow L$, not a particular
choice of an inverse equation or projection coordinate. An equivalence
$F=\alpha\circ H\circ\beta$ induces compatible isomorphisms of target
coordinate rings and source function fields. Transporting a monic equation
through these isomorphisms preserves integrality, so transports $B$ as well.
Consequently it gives an isomorphism of the corresponding finite covers,
lying over the target automorphism. This is the sense in which the cover
and its strata are \emph{intrinsic}.

\begin{proposition}\label{prod:prop:cover-products}
The construction commutes with restriction to a principal target open set
and with adjoining identity coordinates. In particular,
\[
 N_{\Phi\times\id_{\A^m}}\simeq N_\Phi\times\A^m,
 \qquad \pi_{\Phi\times\id}=\pi_\Phi\times\id_{\A^m}.
\]
\end{proposition}
\begin{proof}
For a target polynomial $f\ne0$, the integral closure of $R_f$ in $L$
is $B_f$. One inclusion follows by localizing monic equations. Conversely,
if $b$ is integral over $R_f$, multiply $b$ by a sufficiently large power
of $f$. Its monic equation then has all coefficients in $R$, so
$f^Nb\in B$ and $b\in B_f$.

For one new variable $z$, the stabilized extension is
$R[z]\subset L(z)$. The ring $B[z]$ is integral over $R[z]$, is normal
because a polynomial ring over a normal domain is normal
\cite{StacksNormal}, and has fraction field $L(z)$. Any element of $L(z)$
integral over $R[z]$ is also integral over $B[z]$, hence belongs to $B[z]$.
Thus $B[z]$ is exactly the required integral closure. Iterating proves the
claim for $m$ variables, and the ring inclusions identify the morphisms.
\end{proof}

This proposition justifies computing a cover on $c\ne0$ using a convenient
inverse polynomial. That polynomial still has to describe the correct
function field, and its hypersurface may require normalization. A multiple
root caused merely by projecting two distinct inverse points to the same
coordinate is not automatically intrinsic ramification.

\subsection{The one-point-fiber stratum in the target}
For the finite cover, define
\[
 T_1(F)=\{v\in\A^n_v:\pi_F^{-1}(v)\text{ has exactly one geometric point}\},
 \qquad Z_1(F)=\overline{T_1(F)}.
\]
Here the bar denotes Zariski closure in the target $\A^n_v$.
Thus $Z_1(F)$ is a Zariski closed algebraic subset, endowed with its
reduced structure. The locus $T_1(F)$ is constructible, as proved below;
we do not assume it is Zariski closed or identify it with $Z_1(F)$
without a separate closedness argument.
Over $\C$, a geometric point is an ordinary complex point. This definition
counts the support of the fiber, ignoring multiplicity. A single point of
length ten still counts as one point in $T_1(F)$. It is a statement about
$\pi_F$, not about the original map $F$ and not about an omitted target.

\begin{proposition}\label{prod:prop:one-point}
The set $T_1(F)$ is constructible, and its Zariski closure
$Z_1(F)\subset\A^n_v$ is intrinsic. Both commute with identity products:
\[
 T_1(\Phi\times\id_{\A^m})=T_1(\Phi)\times\A^m,\qquad
 Z_1(\Phi\times\id_{\A^m})=Z_1(\Phi)\times\A^m.
\]
\end{proposition}
\begin{proof}
Form the fiber product $N_F\times_{\A^n}N_F$: its points are pairs of
points in the same finite fiber. The diagonal consists of equal pairs;
it is Zariski closed because a finite morphism is separated. The image in the
target of the complement of this diagonal is precisely the set of values
with at least two distinct preimages. By Chevalley's theorem this image
is constructible, meaning a finite union of Zariski locally closed sets. Every
fiber of $\pi_F$ is nonempty by Proposition~\ref{prod:prop:finite-cover}.
The complement of that image is therefore exactly $T_1(F)$, proving
constructibility. Equivalence transports the cover and hence this set
and its Zariski closure, since the target automorphism is a Zariski
homeomorphism.

The finite fiber of $\pi_\Phi\times\id$ over $(v,z)$ is
$\pi_\Phi^{-1}(v)\times\{z\}$, with the same number of geometric points.
This proves the first product formula. The second follows by the closure
argument following \eqref{prod:eq:missing-product}.
\end{proof}

For example, the finite map $t\mapsto t^d$, $d\geq2$, has $d$ distinct
points over a nonzero value and one point over zero. Thus $T_1=\{0\}$.
Its fiber over zero is nevertheless nonreduced; the next invariant records
that additional information.

\subsection{Local fiber length and its closed strata}
Let $p\in N_F$ be a complex point, with $v=\pi_F(p)$. The local ring
$\mathcal O_{N_F,p}\subset L$ consists of rational functions in the source
function field $L$ that are regular on a Zariski open neighborhood of $p$
in $N_F$.
Let $\mathfrak m_v$ denote the maximal ideal of functions vanishing at $v$.
The \emph{local fiber algebra} and its length are
\begin{equation}\label{prod:eq:fiber-length}
 C_p=\mathcal O_{N_F,p}/\mathfrak m_v\mathcal O_{N_F,p},\qquad
 e_F(p)=\dim_\C C_p.
\end{equation}
Finiteness of $\pi_F$ makes $C_p$ a finite-dimensional local Artinian
$\C$-algebra. Its nilpotent elements record the multiplicity of the point
inside the fiber. This definition does not require $N_F$ to be smooth or
the finite cover to be flat. For $t\mapsto t^d$, the algebra at zero is
$\C[t]/(t^d)$, of length $d$, whereas every point over a nonzero target
has local fiber algebra $\C$ and length one.

The integer $e_F(p)$ is the \emph{local fiber length}, also called the
\emph{fiber multiplicity} at $p$. For a finite dominant map of smooth
complex curves, it coincides with the usual ramification index. Indeed,
if $t$ is a local parameter at $v$ and $u$ is a local parameter at $p$,
then $\pi_F^*t=a u^e$ for a unit $a$, and
\[
 e_F(p)=\dim_\C\bigl(\mathcal O_{N_F,p}/(u^e)\bigr)=e.
\]
The last equality follows from the filtration by powers of $u$, whose
$e$ successive quotients are all isomorphic to $\C$.
The usual ramification index is defined for an extension of discrete
valuation rings \cite{StacksRamification}. In higher dimensions, it is
attached to prime divisors via their generic points. At an arbitrary
closed point $p$, we retain the term ``local fiber length'' to distinguish
$e_F(p)$ from such divisorial ramification indices. Throughout, this
invariant refers to the finite morphism $\pi_F$, not to the original map $F$.

For $k\geq1$, put
\[
 E_k(F)=\{p\in N_F:e_F(p)\geq k\}.
\]
This stratum lies in the \emph{source of the finite cover}; $T_1(F)$ and
$\Miss(F)$ lie in the target. Keeping their ambient spaces distinct is
essential when using an equivalence to transport them.

\begin{proposition}\label{prod:prop:length-strata}
The set $E_k(F)$ is closed and intrinsic. It satisfies
\[
 E_k(\Phi\times\id_{\A^m})=E_k(\Phi)\times\A^m.
\]
\end{proposition}
\begin{proof}
We construct a finite module that detects the inequality in the definition.
Using $N_F=\operatorname{Spec}(B)$ and the target ring $R$, let
\[
 D=B\otimes_R B,\qquad
 \mathcal I=\ker(D\xrightarrow{\ b\otimes b'\mapsto bb'\ }B),\qquad
 P_{k-1}=D/\mathcal I^k.
\]
The ideal $\mathcal I$ cuts out the diagonal. The module $P_{k-1}$ is
viewed as a $B$-module through the first factor. It is finite, since
$B$ is finite over $R$, so $D$ is finite over that first copy of $B$.
Geometrically this is the module of functions on the $(k-1)$st infinitesimal
neighborhood of the relative diagonal, pushed to its first factor.

Specialize the first factor to $p$. The second factor becomes the finite
fiber over $v$. In its Artinian decomposition into local factors, the
diagonal ideal becomes the maximal ideal $\mathfrak n_p$ of the factor
$C_p$, and becomes the unit ideal at every other point of the fiber.
Indeed, if $p'\ne p$, some regular function separates the two points, so
its difference between the second and first factors is invertible in the
local Artinian factor at $p'$. It follows that
\begin{equation}\label{prod:eq:principal-parts-fiber}
 P_{k-1}\otimes_B\C(p)\simeq C_p/\mathfrak n_p^k.
\end{equation}
This argument also explains why other points in the same fiber do not
contribute to this vector-space dimension.

For any finite-dimensional local algebra $C$ with maximal ideal
$\mathfrak n$, one has
\[
 \dim_\C(C/\mathfrak n^k)\geq k
 \quad\Longleftrightarrow\quad \dim_\C C\geq k.
\]
The forward implication follows because the quotient cannot increase
dimension. For the reverse implication, if $\mathfrak n^k=0$, the quotient
is the whole algebra. Otherwise the successive quotients
$C/\mathfrak n,\mathfrak n/\mathfrak n^2,\ldots,
\mathfrak n^{k-1}/\mathfrak n^k$ are all nonzero: an equality between
successive powers would imply that power is zero by Nakayama's lemma.
Their dimensions sum to at least $k$. Thus
\eqref{prod:eq:principal-parts-fiber} detects precisely $e_F(p)\geq k$.

To see directly that this is a closed condition, take a finite presentation
$B^a\xrightarrow{M}B^b\to P_{k-1}\to0$. At a point $p$, the dimension
of the last fiber is $b-\operatorname{rank}M(p)$. Requiring it to be at
least $k$ is the vanishing of the $(b-k+1)$-minors of $M$ (with the usual
conventions for vacuous or impossible rank bounds). These minors generate
the Fitting ideal $\operatorname{Fitt}_{k-1}(P_{k-1})$; their zero set is
independent of the presentation \cite{StacksFitting}. This proves closedness
without assuming a globally chosen inverse coordinate.

The construction is preserved by isomorphisms of finite covers, proving
intrinsicness. For a stabilization the relative fiber product is
\[
 (N_\Phi\times_{\A^d}N_\Phi)\times\A^m,
\]
and the diagonal ideal is extended from the original one. Alternatively,
the local fiber algebra over $(v,z)$ at $(p,z)$ is exactly $C_p$, since
the added coordinates are set equal to $z$. Its length is unchanged,
which proves the product formula.
\end{proof}

Here is the link to the inverse-polynomial calculations. Suppose that on
a target open set a finite hypersurface is given by a monic polynomial
$R(t;Y)=0$, has the correct inverse function field, and is normal at
$(t_*,Y^*)$. Normalization is then an isomorphism near that point. If
$t_*$ has multiplicity $\ell$ in $R(t;Y^*)$, its local fiber algebra is
\[
 \C[t]_{(t-t_*)}/(R(t;Y^*))
 \simeq\C[\epsilon]/(\epsilon^\ell),
\]
so its intrinsic local fiber length is $\ell$. On such a normal neighborhood
$E_k$ is therefore cut out by
\[
 R=\partial_tR=\cdots=\partial_t^{k-1}R=0.
\]
Smoothness of the hypersurface is a sufficient normality condition and is
what we verify for $F_7$. Without normality, distinct branches of a singular
projection can separate under normalization; the raw root multiplicity
alone does not justify this conclusion.

\subsection{The general component obstruction and its use}
The preceding results provide three intrinsic closed sets:
\[
 \overline{\Miss(F)}\subset\A^n_v,\qquad
 Z_1(F)\subset\A^n_v,\qquad E_k(F)\subset N_F.
\]
Each is transported by equivalence and commutes with identity products.
For a polynomial map the image and its complement are constructible by
Chevalley's theorem; the formulas for $\overline{\Miss(F)}$ also follow
directly from \eqref{prod:eq:missing-product}. In the next proposition
$\mathcal D(F)$ may be any one of these assignments, with a fixed $k$
when the last is used.

\begin{proposition}\label{prod:prop:general-component}
Suppose $F:\A^n\to\A^n$ is dominant and generically finite, and
$\mathcal D(F)$ has an irreducible component isomorphic to
$\C^\times\times\A^s$, where $0\leq s\leq n-1$.
Then for every $r\geq0$ and every polynomial map
$\Phi:\A^{n-s-1}\to\A^{n-s-1}$,
\[
 F\times\id_{\A^r}\nsim\Phi\times\id_{\A^{r+s+1}}.
\]
\end{proposition}
\begin{proof}
If an equivalence existed, it would force $\Phi$ to be dominant and
generically finite. Product compatibility would identify the relevant
closed sets on the two sides. The chosen component on the left becomes
$\C^\times\times\A^{s+r}$, of dimension $s+r+1$.

To identify the components on the right, write the finitely many irreducible
components of $\mathcal D(\Phi)$ as $Z_i$. Each
$Z_i\times\A^{r+s+1}$ is irreducible: its affine coordinate ring is a
polynomial ring over a domain. These products are exactly the irreducible
components, since none is contained in the union of the others. Thus our
component would be isomorphic to some $Z_i\times\A^{r+s+1}$. Dimensions
add under products, forcing $\dim Z_i=0$. An irreducible reduced affine
variety of dimension zero over $\C$ is a point. We would obtain an
isomorphism
$\C^\times\times\A^{s+r}\simeq\A^{s+r+1}$, contradicted by the
nonconstant unit discussed above.
\end{proof}

\begin{lemma}\label{prod:lem:finite-obstruction}
If either $Z_1(F)=\overline{T_1(F)}$ or $E_k(F)$ has an irreducible
component isomorphic to $\C^\times$, then $F$ cannot split off an identity
coordinate, even after further stabilization, in the precise sense of
\eqref{prod:eq:stable-exclusion}.
\end{lemma}
\begin{proof}
Apply Proposition~\ref{prod:prop:general-component} with $s=0$.
\end{proof}

It is essential to find an \emph{irreducible component} of an intrinsic
closed set, not merely a copy of $\C^\times$ embedded somewhere in it.
Affine spaces themselves can contain such curves. The applications certify
components as follows. On $c\ne0$, the weighted change of target variables
$(c,v)\mapsto(c,Y)$ is an isomorphism. An isolated stratum point $Y^*$
there gives a one-dimensional orbit
\[
 \{(c,Y_1^*/c^{a_1},\ldots,Y_{n-1}^*/c^{a_{n-1}}):c\ne0\},
 \qquad a_i>0.
\]
If some $Y_i^*\ne0$, the equations $c^{a_i}v_{i+1}=Y_i^*$ force
$c\ne0$, so this orbit is already closed in the full target. For a stratum
in the finite cover, a fixed-root lift is likewise checked to be closed
over this orbit. If the stratum has local dimension one along the curve,
the closed irreducible curve is a component: any component containing it
has dimension one there, and a proper closed subset of an irreducible curve
has dimension zero. These checks avoid a hidden assumption about the boundary
$c=0$. They are carried out explicitly for $F_5$ and $F_7$ below.

In practice, the proof has four steps: identify an intrinsic stratum;
compute it on a usable chart; prove that the proposed curve is closed
and is a component in the full ambient space; then apply the component
obstruction. The following sections separate these general methods from
the arithmetic specific to each map.

\section{Identification of \texorpdfstring{$F$ and $G$}{F and G} with Gallagher's examples}\label{prod:sec:G}
Gallagher's general recipe \cite[Section 3]{Gallagher} with $c=1$ permits a seed
$p_*$ satisfying $p_*(0)=0$, $p_*(1)=-1$, $\int_0^1p_*=0$.
Write $q_*'=sp_*'$, $q_*(0)=0$, and set
$a_*=-(1+\kappa)/(2+\kappa)$ for $\kappa=p_*'(1)\ne-2$.
For nonzero $b_*$, the map is
\[
 \Phi_{p_*}=\left(\frac{u+q_*(w')/\gamma'^2}{x^2},
 \frac{1+p_*(w')/\gamma'}x,x\gamma'\right),
 \quad u=1+xy,\quad \gamma'=1+a_*xy+b_*x^2z,\quad w'=\gamma'u.
\]
For $F$, take $p_*(s)=2s-3s^2$, $q_*(s)=s^2-2s^3$,
and $b_*=-1/2$. The seed conditions hold since $p_*(0)=0$,
$p_*(1)=-1$, and $\int_0^1p_*(s)\,ds=0$; also $q_*'=sp_*'$.
Here $\kappa=-4$ and $a_*=-3/2$, hence
$2\gamma'=2-3xy-x^2z=\gamma$ and $w=2w'$.
The polynomials $p_0(s)=-3s^2+4s$ and $q_0(s)=-s^3+s^2$
defining $F$ satisfy $p_0(2s)=4p_*(s)$ and $q_0(2s)=4q_*(s)$.
Substitution into the formulas of Section~\ref{sec:first} gives
$F=(\Phi_{p_*,1},2\Phi_{p_*,2},2\Phi_{p_*,3})$.
Thus $F$, the Alp\"oge example, belongs to the general seed family up to
an invertible linear target change. Its geometric degree is three.

For $G$, take specifically
\[
 p_*(s)=2s^3-6s^2+3s,\qquad
 q_*(s)=\frac32s^4-4s^3+\frac32s^2,
 \qquad b_*=-\frac12.
\]
The endpoint and integral conditions hold by direct evaluation.
Here $\kappa=-3$ and $a_*=-2$, so
$2\gamma'=2-4xy-x^2z=\gamma$, the stage defining $G$.
Its seed polynomials are
\[
 p_0(s)=s^3-6s^2+6s,\qquad
 q_0(s)=\frac38s^4-2s^3+\frac32s^2.
\]
Direct expansion gives $p_0(2s)=4p_*(s)$ and $q_0(2s)=4q_*(s)$.
Substituting $w=2w'$ and $\gamma=2\gamma'$ into $G$ yields
\begin{equation}\label{prod:eq:G-equivalence}
 \boxed{\ G=(2\Phi_{p_*,3},\ 2\Phi_{p_*,2},\ \Phi_{p_*,1}).\ }
\end{equation}
The output transformation $(z_1,z_2,z_3)\mapsto(2z_3,2z_2,z_1)$ is an
invertible linear map. Thus $G$ cannot be claimed inequivalent to the whole
general seed family. Equation~\eqref{prod:eq:G-equivalence} does not decide
whether $G$ is equivalent to the particular earlier degree-four seed
$p(s)=s-2s^3$; membership in the general family and equivalence to a
particular displayed member are different questions.

\section{Omitted-locus applications: \texorpdfstring{$F_4$}{F4}}\label{prod:sec:F4-omitted}
\begin{theorem}\label{prod:thm:main}
For the maps $F_4,F_6$ defined in the main text,
\[
 \Miss(F_4)\simeq\coprod_{j=1}^{5}\C^\times,
 \qquad
 \Miss(F_6)\simeq\C^\times\times\A^1.
\]
Both omitted loci are Zariski closed in their respective target affine spaces. Therefore, for every polynomial map
$\Phi:\A^3\to\A^3$ and every integer $r\geq0$,
\begin{align*}
 F_4\times\id_{\A^r}&\nsim\Phi\times\id_{\A^{r+1}},\\
 F_6\times\id_{\A^r}&\nsim\Phi\times\id_{\A^{r+2}}.
\end{align*}
\end{theorem}
The omitted-locus calculations are proved below. Their dimensions after
stabilization are $r+1$ and $r+2$, and their components still have
nonconstant units. Lemma~\ref{prod:lem:obstruction} therefore proves both
inequivalence statements, including the stabilized versions. The theorem
applies to \emph{every} three-dimensional map, so it includes all
three-dimensional members of Gallagher's family, even of the same
geometric degree. It also excludes extensions from smaller dimensions,
by first extending such a map to dimension three.

\subsection{An inverse equation without exceptional parameter values}
Write $s=w_1$, $a=w_2$. The defining data are
\begin{align*}
 A(s)&=\left(\frac{s}{3}-\frac{4s^2}{3},\quad
 \frac{2s^2}{3}-\frac{8s^3}{9},\quad
 2s+\frac{7s^3}{9}-\frac{2s^4}{3}\right),\\
 B(s)&=(s^2-1,s^3,s^4+s^2+2),\qquad
 X=A+aB,\qquad \Delta=(1,s,s^2).
\end{align*}
The sweep is $S=X+\gamma\Delta$. In source coordinates $(x,y,z,t)$,
\[
 \gamma=1+xy+x^2t,\quad u_1=1+xz,\quad
 u_2=-\frac79+\frac{22}{9}xy+\frac83xz,
 \quad s=\gamma u_1,\quad a=\gamma u_2,
\]
and $F_4=(C,S_1/C,S_2/C^2,S_3/C)$ with $C=\gamma x$.
These expressions extend polynomially by the exact divisibility
identities of the construction. The discussion below concerns those
polynomial extensions, including their values at $C=0$.

For a target $(c,v_2,v_3,v_4)$ with $c\ne0$, put
\[
 (Y_1,Y_2,Y_3)=(cv_2,c^2v_3,cv_4).
\]
The combination of sweep equations $S_3-2sS_2+s^2S_1$ gives
\[
 2a+\varphi(s)=Y_3-2sY_2+s^2Y_1,
 \qquad \varphi(s)=2s-\frac29s^3-\frac29s^4.
\]
Thus there is a unique candidate
\begin{equation}\label{prod:eq:a4}
 a(s)=\frac{Y_3-2sY_2+s^2Y_1-\varphi(s)}2
\end{equation}
for every $s$, with no exceptional values of $s$. The first sweep equation
then gives $\gamma(s)=Y_1-X_1(s,a(s))$; the second becomes $R_4(s;Y)=0$, where
\begin{align}
 R_4(s;Y)={}&\frac29s^5+\frac29s^4+
 \left(\frac89+Y_1\right)s^3
 -\left(\frac43+2Y_2\right)s^2\notag\\
 &+(2Y_1+Y_3)s-2Y_2.\label{prod:eq:R4}
\end{align}
Indeed $R_4=2(X_2-sX_1+sY_1-Y_2)$ after \eqref{prod:eq:a4}; satisfying
this equation and the first sweep equation also satisfies the third.
Direct differentiation and substitution give the polynomial identity
\begin{equation}\label{prod:eq:deriv4}
 R_4'(s;Y)-2\gamma(s)=sR_4(s;Y).
\end{equation}
Consequently a root is simple exactly when its candidate $\gamma$ is nonzero.

Every such root reconstructs one source point: set $x=c/\gamma$,
$u_1=s/\gamma$, $u_2=a/\gamma$, and then solve
\[
 xz=u_1-1,\qquad
 xy=\frac9{22}\left(u_2+\frac79-\frac83(u_1-1)\right),
 \qquad x^2t=\gamma-1-xy.
\]
All divisions here are by nonzero numbers. Conversely, a source point above
$c\ne0$ has $x\ne0$ and $\gamma\ne0$, so it arises in this way.
Thus its fiber has exactly as many points as $R_4$ has simple roots.

\subsection{All quintics without a simple root}
Normalize $P=(9/2)R_4$:
\begin{equation}\label{prod:eq:P4}
 P(s)=s^5+s^4+A_0s^3+B_0s^2+D_0s+(B_0+6),
\end{equation}
where
\[
 A_0=4+\frac92Y_1,\qquad B_0=-6-9Y_2,\qquad
 D_0=9Y_1+\frac92Y_3.
\]
The coefficient map $Y\mapsto(A_0,B_0,D_0)$ is invertible affine.
A quintic with no simple root has multiplicity partition $5$ or $3+2$.
In the first case the $s^4$ coefficient forces
$P=(s+1/5)^5$. But its constant coefficient minus its $s^2$ coefficient
minus $6$ is $-18999/3125$, contrary to \eqref{prod:eq:P4}.

In the second case write $P=(s-\rho)^3(s-\sigma)^2$. Its $s^4$ coefficient
forces $\sigma=(-1-3\rho)/2$. The constant--quadratic coefficient relation is
equivalent to
\begin{equation}\label{prod:eq:h}
 h(\rho)=3\rho^5+2\rho^4+2\rho^3-2\rho^2-\rho+8=0.
\end{equation}
For example, direct expansion gives
$[s^0]P-[s^2]P-6=-3h(\rho)/4$.
The exact discriminant is
\[
 \operatorname{disc}(h)=1292134656\ne0.
\]
There are therefore exactly five distinct choices of $\rho$ in $\C$.
Moreover $h(-1/5)\ne0$, so $\rho\ne\sigma$ in every case.
Each polynomial has a uniquely distinguished triple root; hence distinct
$\rho$ give distinct polynomials and distinct targets $Y$.
These targets form the five-element set $D_4$, given explicitly by
\begin{align*}
 Y_1(\rho)&=-\frac{5\rho^2+2\rho+5}{6},\\
 Y_2(\rho)&=-\frac{5\rho^3-6\rho^2-3\rho+24}{36},\\
 Y_3(\rho)&=\frac{15\rho^4+14\rho^3+33\rho^2+12\rho+30}{18},
 \qquad h(\rho)=0.
\end{align*}
None equals $0$: at $Y=0$ the polynomial is
\[
 R_4(s;0)=\frac29s^2(s-1)(s^2+2s+6),
\]
which has three simple roots.

For each $Y\in D_4$ define
\[
 Z_Y=\left\{(c,Y_1/c,Y_2/c^2,Y_3/c):c\in\C^\times\right\}.
\]
It is the closed subvariety of $\A^4$ defined by
$cv_2=Y_1$, $c^2v_3=Y_2$, $cv_4=Y_3$: since $Y\ne0$, these equations
force $c\ne0$. Projection onto $c$ and the displayed parametrization
are inverse algebraic isomorphisms with $\C^\times$. Distinct $Y$ give
disjoint subvarieties. These five curves are exactly the omitted targets
with $c\ne0$.

\subsection{The boundary contains no omitted target}
The printed polynomial expansion, evaluated at $x=0$, gives
\begin{align*}
 F_4(0,y,z,t)=\biggl(0,&-\frac{26}{9}y-\frac{35}{9}z,\quad
 \frac{37}{9}y^2+\frac{16}{3}yz+\frac{11}{3}z^2-\frac{22}{9}t,\\
 &\frac{20}{3}y+\frac{29}{3}z\biggr).
\end{align*}
The linear map from $(y,z)$ to the second and fourth outputs has
determinant $-2$. The third output then determines $t$ uniquely because
its coefficient is $-22/9$. This restriction is an isomorphism onto the
target hyperplane $c=0$. All its points are attained, completing the proof
that $\Miss(F_4)=\coprod_{Y\in D_4}Z_Y$.

\section{The obstruction for \texorpdfstring{$F_5$}{F5}: an isolated one-point fiber}\label{prod:sec:F5}
The sweep data, with $(t,a)=(w_1,w_2)$, are
\begin{align*}
 X_1&=(160t^3-60t^2-240ta+51t+60a)/29,\\
 X_2&=(60t^4-20t^3-120a^2+51a)/29,\\
 X_3&=(-48t^5+15t^4+240t^3a-17t^3-60t^2a
       -240ta^2+51ta+30a^2)/29,
\end{align*}
and $S=X+\gamma(1,t,a)$. Write
\[
 f=X_2+(Y_1-X_1)t-Y_2,\qquad
 h=X_3+(Y_1-X_1)a-Y_3,
 \quad R_5=-\frac{29^4}{30}\operatorname{Res}_a(f,h).
\]
The leading coefficients of $f,h$ as quadratics in $a$ are the nonzero
constants $-120/29,-30/29$. The polynomial $R_5$ has degree ten in $t$,
leading coefficient $20480$, and $t^9$ coefficient $-51200$.
Its resultant identity makes $t$ integral over $\C[Y]$; then $f$ makes
$a$ integral, and $\gamma=Y_1-X_1$ is polynomial in them. Thus
$S:\A^3_{\gamma,t,a}\to\A^3_Y$ is finite. Its source is normal and its
function field is the generic inverse field, of degree ten.

For completeness, on the chart $c\ne0$ the original stage is invertible:
with $v=(xy,x^2z,x^3t_{\rm src})$ it has
\[
 \gamma=1+\frac{20}{29}v_1,\quad
 u_1=1-\frac{49}{29}v_1+8v_2,\quad
 u_2=1-\frac{69}{29}v_1+9v_2+v_3+\frac{4197}{1682}v_1^2.
\]
These equations solve successively for $v_1,v_2,v_3$, and
$w_1=\gamma u_1$, $w_2=\gamma^2u_2$, $x=c/\gamma$ give the birational
identification. Hence $N_{F_5}$ on this target chart is exactly
$\C^\times\times\A^3_{\gamma,t,a}$, with target coordinates
$Y=(cv_2,c^2v_3,c^3v_4)$ and finite map $\id\times S$.

Every root of $R_5(t;Y)$ has at least one common root of $f,h$ in $a$:
their leading coefficients never vanish, so the resultant has no roots
coming from infinity in $a$. Consequently a one-point sweep fiber forces
all ten roots of $R_5$ to coincide. The $t^9$ coefficient then forces
\[
 R_5=20480(t-1/4)^{10}.
\]
Comparing the coefficients of $t^7,t^6,t^5$ determines a unique possible
target,
\begin{equation}\label{prod:eq:F5star}
 Y^*=\left(\frac{23}{58},\frac{1709}{9280},\frac{1409}{37120}\right).
\end{equation}
The exact coefficient differences, in that order, are
\begin{align*}
 3840(58Y_1-23),\qquad
 -80(2552Y_1+9280Y_2-2721),\\
 -24(4988Y_1-37120Y_2-37120Y_3+6267).
\end{align*}
At \eqref{prod:eq:F5star} direct substitution gives
\[
 R_5=\frac5{256}(4t-1)^{10},\qquad
 f(1/4,a)=-\frac3{4640}(80a-17)^2,\quad
 h(1/4,a)=-\frac3{18560}(80a-17)^2.
\]
The finite fiber is therefore the single point
$(\gamma,t,a)=(0,1/4,17/80)$. This proves both existence and uniqueness
of the one-point sweep target, rather than just a root collision in a
projection.

It follows that $T_1(F_5)\cap\{c\ne0\}$ is precisely
\[
 Z_5=\{(c,Y_1^*/c,Y_2^*/c^2,Y_3^*/c^3):c\in\C^\times\}.
\]
Because $Y_1^*\ne0$, this curve is closed in the full target and isomorphic
to $\C^\times$. It is an irreducible component of
$\overline{T_1(F_5)}$: on the open chart it is the entire stratum, and
any remaining portion is contained in the closed hyperplane $c=0$.
Lemma~\ref{prod:lem:finite-obstruction} proves:
\begin{theorem}\label{prod:thm:F5-complete}
The four-dimensional map $F_5$ cannot split off an identity coordinate,
even after further stabilization. In particular it is not equivalent to
any stabilized three-dimensional degree-ten example.
\end{theorem}

\section{Product structure of \texorpdfstring{$F_6$}{F6}}\label{prod:sec:F6}
\subsection{The omitted locus of \texorpdfstring{$F_6$}{F6}}\label{prod:sec:F6-omitted}
\subsubsection{Elimination on the nonzero-coordinate chart}
Write $(s,a,b)=(w_1,w_2,w_3)$ and set
\begin{align*}
 G_2&=a+sb+s^2-s^3,& G_3&=2sG_2+a,\\
 G_4&=b-G_2^2+\psi(s),& \psi(s)&=2s^4-2s^5,\\
 X_1&=-\det\frac{\partial(G_2,G_3,G_4)}{\partial(s,a,b)},&
 X&=(X_1,G_2+sX_1,G_3+s^2X_1,G_4+aX_1).
\end{align*}
Here $S=X+\gamma(1,s,s^2,a)$. For a target $(c,v_2,v_3,v_4,v_5)$ with
$c\ne0$, put $Y=(cv_2,c^2v_3,c^3v_4,c^4v_5)$.
The equations $S=Y$ imply
\begin{align}
 a(s)&=Y_3-2sY_2+s^2Y_1,\notag\\
 b(s)&=Y_4+Y_2^2-Y_1Y_3-\psi(s),\label{prod:eq:ab6}\\
 R_6(s;Y)&=2s^6-2s^5-s^3+(Y_1+1)s^2\notag\\
 &\quad +(Y_4+Y_2^2-Y_1Y_3-2Y_2+Y_1)s+Y_3-Y_2=0.
 \label{prod:eq:R6}
\end{align}
For verification, $G_2=Y_2-sY_1$ follows from the second sweep equation.
The third yields $a(s)$, and substitution into the fourth yields $b(s)$.
Conversely these choices and $R_6=0$ satisfy all four sweep equations
with $\gamma(s)=Y_1-X_1(s,a(s),b(s))$.
The exact identity here is
\begin{equation}\label{prod:eq:deriv6}
 \gamma(s)=R_6'(s;Y)-2R_6(s;Y).
\end{equation}

To check that every candidate with $\gamma\ne0$ reconstructs a unique
source point, recall the source stage. In coordinates
$v=(xy,x^2z_1,x^3z_2,x^4z_3)$ it is
\begin{align*}
 \gamma&=1-29v_1+999v_1^2+355v_1v_2-41553v_1^3+v_4,\\
 u_1&=1+27v_1-5v_2,\\
 u_2&=v_1-12v_2+\frac{2128}{5}v_1^2+v_3,\\
 u_3&=-4v_1-10v_2,
\end{align*}
with $s=\gamma u_1$, $a=\gamma^3u_2$, $b=\gamma^4u_3$.
Given $(\gamma,u_1,u_2,u_3)$, recover successively
\begin{align*}
 v_1&=(u_1-1-u_3/2)/29,&
 v_2&=-u_3/10-2v_1/5,\\
 v_3&=u_2-v_1+12v_2-\frac{2128}{5}v_1^2,\\
 v_4&=\gamma-1+29v_1-999v_1^2-355v_1v_2+41553v_1^3.
\end{align*}
Then $x=c/\gamma\ne0$ recovers $(y,z_1,z_2,z_3)$.
Thus \eqref{prod:eq:deriv6} proves that the fiber cardinality is the number
of simple roots of $R_6$.

\subsubsection{All sextics without a simple root}
The monic polynomial $R_6/2$ has fixed leading terms
\[
 s^6-s^5+0s^4-\frac12s^3.
\]
A degree-six polynomial with no simple root has multiplicity partition
$6$, $4+2$, $3+3$, or $2+2+2$. Apart from $3+3$, it is the square of a
monic cubic. Comparing its coefficients of $s^5,s^4,s^3$ forces that cubic to be
\[
 q(s)=s^3-\frac12s^2-\frac18s-\frac5{16}.
\]
In the remaining case it is $(s^2+us+v)^3$; coefficient comparison gives
$u=-1/3$, $v=-1/9$, whose $s^3$ coefficient is $5/27$, not $-1/2$.
This case is impossible. Conversely a square has no simple roots.
It follows that a target is omitted precisely when $R_6=2q^2$.
Since
\[
 2q(s)^2=2s^6-2s^5-s^3+\frac{21}{32}s^2+\frac5{32}s+\frac{25}{128},
\]
the omitted sweep targets are the curve $D_6\simeq\A^1$ parametrized by
\begin{equation}\label{prod:eq:D6}
 Y_1=-\frac{11}{32},\quad Y_2=h,\quad
 Y_3=h+\frac{25}{128},\quad
 Y_4=-h^2+\frac{53}{32}h+\frac{1773}{4096}.
\end{equation}

Under the twist these give
\[
 (c,h)\longmapsto
 \left(c,-\frac{11}{32c},\frac{h}{c^2},
 \frac{h+25/128}{c^3},
 \frac{-h^2+(53/32)h+1773/4096}{c^4}\right).
\]
This is a closed embedding of $\C^\times\times\A^1$ into $\A^5$:
substituting $Y_i=c^iv_{i+1}$ into \eqref{prod:eq:D6} gives polynomial
defining equations, the first of which forces $c\ne0$.
Its inverse is $(c,v_2,v_3,v_4,v_5)\mapsto(c,c^2v_3)$.

\subsubsection{Surjectivity onto the boundary hyperplane}
The full map is $F_6=(C,S_1/C,S_2/C^2,S_3/C^3,S_4/C^4)$.
Its polynomiality follows from the exact division identities of its stage.
Here is a short structural check of its boundary restriction.
For $i=1,\ldots,4$, the polynomial
\[
 E_i(\gamma,u)=\frac{S_i(\gamma,\gamma u_1,\gamma^3u_2,\gamma^4u_3)}{\gamma^i}
\]
composed with the source stage is a polynomial $P_i(v)$, and the corresponding
output is
\[
 x^{-i}P_i(xy,x^2z_1,x^3z_2,x^4z_3).
\]
A monomial $v_1^{e_1}\cdots v_4^{e_4}$ contributes a power
$x^{e_1+2e_2+3e_3+4e_4-i}$. Polynomiality removes all terms of negative
weight, and setting $x=0$ selects precisely weight $i$.
With $(y_1,y_2,y_3,y_4)=(y,z_1,z_2,z_3)$, the boundary outputs therefore have form
\[
 H_i=d_i y_i+K_i(y_1,\ldots,y_{i-1}).
\]
The diagonal coefficients are found by differentiating $P_i$ at $v=0$.
Writing $\mathbf P=(P_1,P_2,P_3,P_4)$, the exact Jacobian matrix is
\[
 \left.\frac{\partial\mathbf P}{\partial(v_1,v_2,v_3,v_4)}\right|_{v=0}
 =\begin{pmatrix}
 1&46&2&-15\\0&29&3&-16\\0&0&5&-17\\0&0&0&-2
 \end{pmatrix}.
\]
Thus $(d_1,d_2,d_3,d_4)=(1,29,5,-2)$, all nonzero. Successively solving
the four triangular equations shows that $H$ is a polynomial automorphism.
Also $C=\gamma x=0$ at $x=0$. Hence the entire target hyperplane $c=0$ is
attained. Together with \eqref{prod:eq:D6}, this completes the omitted-locus
calculation and the proof of Theorem~\ref{prod:thm:main}.

\subsection{An explicit identity factor of \texorpdfstring{$F_6$}{F6}}
\begin{theorem}\label{prod:thm:F6-split}
There are explicit polynomial automorphisms $T,\Theta$ such that
\[
 T\circ F_6\circ\Theta^{-1}=\Psi\times\id_{\A^1},
\]
where $\Psi:\A^4\to\A^4$ is polynomial, has constant Jacobian $-54$,
and has geometric degree six. Its omitted locus is a single closed copy
of $\C^\times$, so $\Psi$ cannot split off an identity coordinate, even
after stabilization.
\end{theorem}
Thus an inequivalence proof against \emph{all four-dimensional} extensions
is impossible for $F_6$. The earlier exclusion of three-dimensional
extensions remains valid.

\subsubsection{Target coordinates}
Write the original target as $(c,B,V,U,Z)$. Set
\begin{equation}\label{prod:eq:F6target}
 q=Z+V^2-BU,\qquad k=cU-V,\qquad l=cq-2U,
 \qquad T(c,B,V,U,Z)=(c,B,k,l,q).
\end{equation}
The inverse is
\[
 U=(cq-l)/2,\qquad V=cU-k,\qquad Z=q-V^2+BU,
\]
so $T$ is a polynomial automorphism, with determinant $2$.
For the sweep expressions in Section~\ref{prod:sec:F6-omitted}, direct cancellation gives
\begin{equation}\label{prod:eq:F6shadowq}
 S_4+S_2^2-S_1S_3=b+\psi(t).
\end{equation}
Hence the four transformed target expressions are
\[
 c=x\gamma,\qquad B=\frac{S_1}{c},\qquad
 k=\frac{S_3-S_2}{c^2},\qquad
 l=\frac{b+\psi(t)-2S_3}{c^3},
\]
and $q=(b+\psi(t))/c^4$.

\subsubsection{Polynomial source coordinates}
Start from $(x,y,z_1,z_2,z_3)$ and put
\begin{align*}
 I&=y-\frac5{27}xz_1,\\
 L(I,z_1)&=-\frac59z_1
 (41553I^2+7695Ixz_1+475x^2z_1^2-180z_1),\\
 W&=z_3+L(I,z_1),\qquad
 \alpha=-\frac{145}{27}(1-135xI),\\
 K&=\alpha z_1+x^2W,\\
 h&=-195750I^2z_1-2(1+135xI)W.
\end{align*}
The first two replacements are polynomial shears. The matrix taking
$(z_1,W)$ to $(K,h)$ has constant determinant $290/27$, and its inverse is
\begin{equation}\label{prod:eq:F6zinverse}
 z_1(h)=-\frac{27}{290}\{2K(1+135xI)+x^2h\},\qquad
 W(h)=\frac{36450I^2K+(135xI-1)h}{2}.
\end{equation}
Thus $(x,I,K,z_2,h)$ are polynomial coordinates. Define, as polynomials
in these invariant candidates,
\begin{align*}
 g&=1-29xI+999x^2I^2-41553x^3I^3+x^2K,\qquad u=1+27xI,\\
 y(h)&=I+\frac5{27}xz_1(h),\\
 P(h)&=xg(1/2-gu)-\frac{11}{10}x+\frac{2128}{145}x^2y(h),\\
 H(h)&=\int_0^h P(\eta)\,d\eta,\qquad J=z_2-H(h).
\end{align*}
This integral is just polynomial integration in the auxiliary variable
$\eta$, with $x,I,K$ held fixed; it has rational coefficients. Replacing
$z_2$ by $J$ is another polynomial shear. We now have coordinates
$(x,I,K,J,h)$.

In these coordinates, substitution into the original stage proves
\begin{equation}\label{prod:eq:F6Dstage}
 \gamma=g,\qquad u_1=u,\qquad
 \partial_hu_3=x^4,\qquad
 \partial_hu_2=x^4g(1/2-gu),
\end{equation}
while $x,I,K,J$ are fixed. To verify the last equality directly, use
\[
 v_3=u_2-v_1+12v_2-\frac{2128}{5}v_1^2,
 \quad \partial_hv_1=-x^4/58,\quad
 \partial_hv_2=-27x^4/290,
\]
divide by $x^3$, and obtain exactly the displayed expression $P(h)$ for
$\partial_hz_2$. The identities extend across $x=0$ because every expression
in these source coordinates is polynomial.

\subsubsection{The invariant coordinates and the identity factor}
Write $t=\gamma u_1$, $a=\gamma^3u_2$, $b=\gamma^4u_3$ as before.
Equation~\eqref{prod:eq:F6Dstage} yields
\[
 \partial_ht=\partial_h\gamma=\partial_hc=0,\qquad
 \partial_ha=c^4(1/2-t),\qquad \partial_hb=c^4.
\]
The sweep polynomials satisfy the exact identities
\begin{equation}\label{prod:eq:F6Dsweep}
 \partial_hS_1=0,\qquad
 \partial_hS_2=\partial_hS_3=c^4/2,\qquad
 \partial_h(b+\psi(t))=c^4.
\end{equation}
They follow by applying $c^4((1/2-t)\partial_a+\partial_b)$ to the
explicit expressions; the companion checks verify them as polynomial
identities. Equations~\eqref{prod:eq:F6target}--\eqref{prod:eq:F6Dsweep} imply
\[
 \partial_h(c,B,k,l)=0,\qquad \partial_hq=1.
\]
This calculation on $c\ne0$ is an identity of polynomial maps, so it holds
globally. In particular the first four transformed outputs are polynomials
in $(x,I,K,J)$ alone.

The last coordinate can be made completely explicit. Let
\begin{equation}\label{prod:eq:F6Q0}
 Q_0(x,I,K)=
 \frac{-4xI+2x^2K(1+135xI)+2u^4-2gu^5}{x^4}.
\end{equation}
The numerator has zero coefficients of $x^0,x^1,x^2,x^3$, as direct expansion
shows; thus $Q_0$ is a polynomial. Equations~\eqref{prod:eq:F6zinverse} give
\[
 u_3=-4xI+2x^2K(1+135xI)+x^4h.
\]
Using \eqref{prod:eq:F6shadowq} and $\psi(t)=2t^4-2t^5$, we therefore obtain
\[
 q\circ F_6=\frac{u_3+2u_1^4-2\gamma u_1^5}{x^4}=h+Q_0(x,I,K).
\]
The final shear $h\mapsto q_{
m src}=h+Q_0$ gives the promised source
automorphism
\[
 \Theta(x,y,z_1,z_2,z_3)=(x,I,K,J,q_{
m src}).
\]
Its elementary factors are given above, with polynomial inverses; its
Jacobian is $290/27$. Defining the first four independent outputs to be
$\Psi(x,I,K,J)$ proves the claimed product identity globally. Equivalently,
the following formula specifies the four-dimensional map without any choices:
\[
 \Psi(x,I,K,J)=\operatorname{pr}_{1,2,3,4}
 \bigl(T\circ F_6\circ\Theta^{-1}(x,I,K,J,0)\bigr).
\]
This is an explicit polynomial map, since all three maps and the inverse
coordinate formulas have been specified. The determinant and degree are
\[
 \det J\Psi=\frac{2(-290)}{290/27}=-54,\qquad \deg_{\rm geom}\Psi=6.
\]

Finally, apply $T$ to the omitted surface already computed in
\eqref{prod:eq:D6}. The result is
\[
 T(\Miss(F_6))=
 \left\{\left(c,-\frac{11}{32c},\frac{25}{128c^2},
 \frac7{64c^3},q\right):c\in\C^\times,\ q\in\C\right\}.
\]
It follows from the product identity that $\Miss(\Psi)$ is precisely the
closed curve given by the first four coordinates. Lemma~\ref{prod:lem:obstruction}
and its stabilized unit argument prove the final assertion of
Theorem~\ref{prod:thm:F6-split}.

\section{The obstruction for \texorpdfstring{$F_7$}{F7}: an isolated quintuple root}\label{prod:sec:F7}
Use the triangular inverse expressions
\begin{align*}
 b(t;Y)&=Y_4-3Y_2t^2+2Y_1t^3-t^4,\\
 a(t;Y)&=Y_3-2Y_2t+Y_1t^2-t^3-t^2b(t;Y),\\
 R_7(t;Y)&=a(t;Y)^2+a(t;Y)+tb(t;Y)+t^3+Y_1t-Y_2.
\end{align*}
This is a monic degree-twelve polynomial in $t$. The generic inverse is
reconstructed uniquely from $t$, so $R_7$ is the minimal polynomial of
this primitive element. Therefore the normalization of the finite
hypersurface $R_7=0$ is the canonical finite cover on $c\ne0$, where
$Y=(cv_2,c^2v_3,c^3v_4,c^4v_5)$. At a smooth point of this hypersurface
normalization is an isomorphism locally.

We establish the existence of an isolated quintuple root over the algebraic
numbers, using a small exact certificate rather than approximate complex roots.
Set
\[
 e_j=\frac1{j!}\partial_t^jR_7\quad(0\leq j\leq4),\qquad
 \mathbf e=(e_0,e_1,e_2,e_3,e_4).
\]
All are integral polynomials (they are Taylor coefficients).
At the point
\[
 (t,Y_1,Y_2,Y_3,Y_4)=(1,3,4,5,5)\quad\text{in }\mathbb F_7^5
\]
they all vanish, and the exact Jacobian matrix modulo seven is
\begin{equation}\label{prod:eq:F7certificate}
 \frac{\partial\mathbf e}{\partial(t,Y_1,Y_2,Y_3,Y_4)}
 =\begin{pmatrix}
 0&0&6&3&5\\
 0&2&4&4&5\\
 0&5&3&2&1\\
 0&3&6&0&6\\
 5&0&3&1&4
 \end{pmatrix},\qquad\det=4\ne0\quad\text{in }\mathbb F_7.
\end{equation}
At the same point,
\[
 \frac1{5!}\partial_t^5R_7=1\pmod7,\qquad
 \partial_{Y_3}R_7=3\pmod7.
\]

\begin{lemma}\label{prod:lem:Henselpoint}
There exists a complex algebraic point $(t_*,Y^*)$ satisfying
$e_0=\cdots=e_4=0$, with invertible Jacobian of these five equations,
$\partial_t^5R_7\ne0$, $\partial_{Y_3}R_7\ne0$, and $Y_1^*\ne0$.
\end{lemma}
\begin{proof}
The multivariable Hensel lemma \cite{ConradHensel} lifts the displayed
residue-class point to a zero of $\mathbf e$ in $\mathbb Z_7^5$, since its
Jacobian determinant is a unit. The nonzero residues just listed remain
nonzero at the lift; also $Y_1\equiv3\pmod7$. The five equations have
invertible Jacobian at the lifted point, so their local solution scheme
over $\mathbb Q$ is zero-dimensional and reduced. More explicitly, the
open locus where this square Jacobian is invertible is a zero-dimensional
finite-type $\mathbb Q$-scheme. Its residue fields are finite extensions
of $\mathbb Q$. The lifted coordinates therefore lie in such a number
field inside $\mathbb Q_7$. Embedding that number field into $\C$ supplies
the required point, preserving the equalities and nonvanishing conditions.
\end{proof}

At this point the hypersurface $R_7=0$ is smooth, since
$\partial_{Y_3}R_7\ne0$. The local fiber algebra has length exactly five,
because the root has multiplicity exactly five. In a neighborhood where
the hypersurface is normal, the locus $E_5(F_7)$ is defined by
$e_0=\cdots=e_4=0$, with the free coordinate $c$. The invertible Jacobian
in \eqref{prod:eq:F7certificate}, retained at the lift, proves that this locus
has local dimension one and is exactly the orbit with $(t,Y)=(t_*,Y^*)$.

Its target orbit is closed in $\A^5$: the equations
$cv_2=Y_1^*$, $c^2v_3=Y_2^*$, $c^3v_4=Y_3^*$, $c^4v_5=Y_4^*$
force $c\ne0$. Its lift at the fixed root $t_*$ is a closed curve in the
finite cover and isomorphic to $\C^\times$. To see closedness without
making a boundary assumption, take the finite preimage of the closed
target orbit. It lies entirely in $c\ne0$, where the fixed-root section
is closed and normalization is an isomorphism along that section.
The local dimension calculation shows that this curve is an irreducible
component of the closed intrinsic locus $E_5(F_7)$.
Lemma~\ref{prod:lem:finite-obstruction} now proves:
\begin{theorem}\label{prod:thm:F7-complete}
The five-dimensional map $F_7$ cannot split off an identity coordinate,
even after further stabilization. In particular it is not equivalent to
an extension of any four-dimensional or three-dimensional polynomial map.
\end{theorem}
No numerical approximation, assumption about generic coefficients, or
unexamined $c=0$ stratum is needed for this argument.

\section{Consequences and further questions}\label{sec:conclusion}
Part~\ref{part:construction} supplies explicit Keller maps, and
Part~\ref{part:equivalence} determines their product structure under
polynomial left-right equivalence:
\begin{center}
\begin{tabular}{@{}lcl@{}}\toprule
Map & Dimension & Result under polynomial left-right equivalence\\\midrule
$F_4$ & 4 & Cannot split an identity coordinate, stably\\
$F_5$ & 4 & Cannot split an identity coordinate, stably\\
$F_6$ & 5 & $\Psi\times\id_{\A^1}$; $\Psi$ cannot split further, stably\\
$F_7$ & 5 & Cannot split an identity coordinate, stably\\\bottomrule
\end{tabular}
\end{center}
Here ``stably'' means that the exclusion persists after adjoining identity
coordinates, in the sense of Lemma~\ref{prod:lem:finite-obstruction}.
The four-dimensional maps $F_4,F_5,\Psi$ have distinct geometric degrees
$5,10,6$, so are pairwise inequivalent, even after common stabilization;
none is a stabilization of a three-dimensional map.

Both three-dimensional examples $F$ and $G$ belong, up to invertible
linear target changes, to Gallagher's general seed family
(Section~\ref{prod:sec:G}); $F$ is the Alp\"oge example.
Their geometric degrees are $3$ and $4$, respectively, so they are
inequivalent to each other, even after adjoining identity coordinates.
More generally, members of Gallagher's family with distinct geometric
degrees are pairwise inequivalent. This degree argument does not compare
two members of the same geometric degree.

The higher-dimensional conclusions are stronger than a difference of
degrees: $F_4,F_5,F_7$ and the factor $\Psi$ cannot be equivalent to
any three-dimensional member of Gallagher's family embedded by adjoining
identity coordinates, including members of the same geometric degree.
These exclusions persist after further stabilization. In particular,
although $F_6$ splits as $\Psi\times\id_{\A^1}$, it is not an
embedding of a three-dimensional member. No inequivalence to every earlier
four- or five-dimensional example is asserted; such a comparison requires
a specified comparison family.

\subsection{Further questions}
Several questions remain. A general solvability theorem for the stage
conditions would turn specified sweep families into systematic families
of Keller maps; high degree in the free sweep data alone is insufficient.
For $F_4,\ldots,F_7$, the omitted loci and finite-cover strata studied in
Part~\ref{part:equivalence} do not yet give complete fiber stratifications.
On the nonzero-coordinate chart, multiplicity partitions and lifting
conditions are needed, and the boundary hyperplane requires its own
analysis. The middle-branch twist problem is formulated precisely in
Problem~\ref{prob:rigid5}. Further comparisons with earlier maps of the
same dimension require a specified comparison family and additional
invariants. Finally, it is natural to ask for the smallest geometric
degree among Keller maps that cannot split off an identity coordinate,
with or without further stabilization.

\clearpage
\phantomsection

\clearpage
\appendix

\section{Graph relations and exact fiber sizes}\label{app:fibers}
We give complete fiber classifications for the three-dimensional maps $F$
and $G$. Both proofs follow the same two steps. Where the twist coordinate
is nonzero, we reconstruct the fiber from the simple roots of a tangency
polynomial. Where that coordinate is zero, we solve the polynomial map
directly. This separates the fiber count from the elimination identities,
which we record afterwards. The higher-dimensional results are summarized
in Appendix~\ref{app:F4}; their fiber classifications remain incomplete.

A \emph{graph relation} for a polynomial map
$\Phi=(\varphi_1,\ldots,\varphi_n)$ defined over $\Q$ is a polynomial in the ideal
\[
 J_\Phi=(\varphi_1-v_1,\ldots,\varphi_n-v_n)
 \subset\Q[x_1,\ldots,x_n,v_1,\ldots,v_n].
\]
Equivalently, it vanishes after substituting $v=\Phi(x)$. Such relations
are useful for elimination, but a relation in one source coordinate need
not distinguish all points of a fiber. Throughout this appendix, fiber
size means the number of distinct source points. Since the maps are
\'etale, their finite fibers over complex points are reduced.

\subsection{The map \texorpdfstring{$F$}{F}}\label{app:F}
Let $F$ be the map of Section~\ref{sec:first}. Write a target as
$v=(v_1,v_2,v_3)^{\mathsf T}$, and define
\begin{align*}
 c_3(v)&=27v_1^2v_3^2-18v_1v_2v_3+16v_1+v_2^3v_3-v_2^2,\\
 a_2(v)&=54v_1v_3^2-18v_2v_3+16=\frac{\partial c_3}{\partial v_1},\\
 \sigma(v)&=(v_2v_3,v_1v_3^2)=(X,Y),\\
 E(X,Y)&=X^3-X^2-18XY+16Y+27Y^2.
\end{align*}
The swept curve is $\mathcal K=\{E=0\}$, parametrized by
$p(w)=-3w^2+4w$ and $q(w)=-w^3+w^2$.

\begin{theorem}[Fibers and graph relations of $F$]\label{thm:strat3}
Let
\[
 \mathcal C=\left\{\left(\frac4{27}t^{-2},\frac43t^{-1},t\right):
 t\in\C^\times\right\}.
\]
Then the fibers of $F$ satisfy
\[
\begin{array}{c|c}
 \text{condition on }v&\#F^{-1}(v)\\ \hline
 c_3(v)\ne0&3\\
 c_3(v)=0,\quad v\notin\mathcal C&1\\
 v\in\mathcal C&0
\end{array}
\]
In particular, $F(\C^3)=\C^3\setminus\mathcal C$, and no fiber has
exactly two points. Its non-properness set is the irreducible surface
\[
 \{c_3=0\}=\overline{\sigma^{-1}(\mathcal K)\cap\{v_3\ne0\}},
\]
where the closure is Zariski closure.

The cubic $h_1(x;v)=c_3(v)x^3+(4-3v_2v_3)x-2v_3$ is a graph relation, and
\begin{align*}
 \operatorname{disc}_x(h_1)&=-c_3a_2^2,\\
 a_2^2-4(4-3v_2v_3)^3&=108v_3^2c_3,\qquad E\circ\sigma=v_3^2c_3.
\end{align*}
The function $a_2$ is twice the pullback of the cuspidal tangent equation
$27Y-9X+8=0$. Among the three points of a fiber with $c_3\ne0$, two
share their $x$-coordinate exactly when $a_2=0$. Finally,
\[
 \sigma^{-1}(\mathcal K)=\{v_3=0\}\cup\{c_3=0\}.
\]
\end{theorem}

\begin{proof}
\emph{1. Fibers on $v_3\ne0$.}
The equations for the sweep are
$X=p(w)+2\gamma$ and $Y=q(w)+\gamma w$. Eliminating $\gamma$ gives
\[
 P(w)=w^3-2w^2+Xw-2Y=0,\qquad
 \gamma=\frac{X-p(w)}2=\frac{P'(w)}2.
\]
For every simple root $w$, the unique source point is
\begin{equation}\label{eq:app-F-inverse}
 x=\frac{v_3}\gamma,\qquad y=\frac{w-\gamma}{v_3},\qquad
 z=\frac{\gamma(5\gamma-3w-\gamma^2)}{v_3^2}.
\end{equation}
Indeed, these formulas invert
$v_3=x\gamma$, $w=\gamma(1+xy)$, and $\gamma=2-3xy-x^2z$.
Conversely, every source point over $v_3\ne0$ has $\gamma\ne0$ and
therefore gives a simple root of $P$. Thus the fiber size is exactly
the number of simple roots, not the number of distinct roots of $P$.

The cubic discriminant is $\operatorname{disc}_w(P)=-4E(X,Y)$.
If $E\ne0$, all three roots are simple. If $E=0$, the roots consist of
one double root and one simple root, except when $P$ is a cube.
Comparing $P=(w-r)^3$ gives
\[
 r=\frac23,\qquad (X,Y)=\left(\frac43,\frac4{27}\right).
\]
The fiber sizes are therefore $3$, $1$, and $0$, respectively.
The identity $E(v_2v_3,v_1v_3^2)=v_3^2c_3(v)$ translates these cases into the
first table on $v_3\ne0$. The omitted targets are exactly $\mathcal C$.
This curve is closed in the full target, since its equations are
$v_2v_3=4/3$ and $v_1v_3^2=4/27$; the first forces $v_3\ne0$.

\emph{2. Fibers on $v_3=0$.}
Since $v_3=x\gamma$, there are two disjoint cases. If $x=0$, substitution
in the polynomial components of $F$ gives the unique point
\[
 (x,y,z)=(0,v_2,v_1-4v_2^2).
\]
If $x\ne0$, then $\gamma=0$. The equations for $F(x,y,z)=(v_1,v_2,0)$
reduce to
\[
 y=\frac{v_2}4-\frac3{2x},\qquad
 (v_2^2-16v_1)x^2=4,\qquad z=\frac{2-3xy}{x^2}.
\]
These conditions are also sufficient, by substitution. Hence this case
contributes two points if $v_2^2\ne16v_1$ and none if $v_2^2=16v_1$.
Because $c_3(v_1,v_2,0)=16v_1-v_2^2$, the total count is again the one stated
in the table. This completes the fiber classification on the whole target.

\emph{3. Graph identities and the cusp tangent.}
Substitution in $F$ gives $h_1(x;F(x,y,z))=0$. The cubic discriminant
formula gives
\[
 \operatorname{disc}_x(h_1)
 =-4c_3(4-3v_2v_3)^3-108c_3^2v_3^2=-c_3a_2^2,
\]
using $a_2^2=4(4-3v_2v_3)^3+108v_3^2c_3$. This last identity and
$E(v_2v_3,v_1v_3^2)=v_3^2c_3$ follow by expansion.
Moreover, $P=P'=0$ is equivalent to $(X,Y)=(p(w),q(w))$, so the
zero set of $E$ is precisely the swept curve. At its cusp $w=2/3$,
$(p',q')=(0,0)$ and $(p'',q'')=(-6,-2)$. Its tangent line is therefore
\[
 Y-\frac4{27}=\frac13\left(X-\frac43\right),
 \quad\text{or}\quad 27Y-9X+8=0.
\]
Since $a_2=2(27Y-9X+8)\circ\sigma$, this proves the tangent assertion.
The identity $E\circ\sigma=v_3^2c_3$ also gives the asserted full inverse
image of $\mathcal K$.

\emph{4. Coincident projected coordinates.}
Suppose $c_3\ne0$ and $v_3\ne0$, and let $w_1,w_2,w_3$ be the three
simple roots of $P$. Formula~\eqref{eq:app-F-inverse} shows that two
points have the same $x$-coordinate exactly when their $p$-values agree.
For distinct $w_1,w_2$,
\[
 p(w_1)-p(w_2)=(w_1-w_2)\bigl(4-3(w_1+w_2)\bigr).
\]
Since $w_1+w_2+w_3=2$, this happens exactly when the remaining root
is $2/3$. Finally, $P(2/3)=-a_2/27$, proving the criterion.
On $v_3=0$, the three source coordinates are $0,x,-x$ with $x\ne0$,
and $a_2=16$, so the criterion holds there as well. A collision of
projected coordinates does not merge the source points.

\emph{5. The non-properness set.}
We have proved that every fiber has at most three points. At a target
with three preimages, the inverse function theorem supplies three
disjoint local inverse branches over a common neighborhood. They exhaust
every nearby fiber, since there can be no fourth point. Choose a smaller
neighborhood whose closure is compact and lies in that common domain.
The preimage of any compact subset is the union of its images under
the three continuous inverse branches, hence is compact. Thus $F$ is
proper near this target.

Conversely, if $F$ is proper near a target, its local biholomorphism
property makes it a finite covering there. Its fiber size is locally
constant and equals the generic value three. The table consequently
identifies the non-properness set as $\{c_3=0\}$.
On $v_3\ne0$, the coordinates $(v_3,X,Y)=(v_3,v_2v_3,v_1v_3^2)$ identify this
surface with $\C^\times\times\mathcal K$, which is irreducible.
No irreducible component of $\{c_3=0\}$ lies in $v_3=0$, since $v_3$ does
not divide $c_3$. Thus this open part is dense in the whole surface,
proving both irreducibility and the asserted closure formula.
\end{proof}

\subsection{The map \texorpdfstring{$G$}{G}}\label{app:G}
For the map $G$ of Section~\ref{sec:deg4}, write a target as $v=(v_1,v_2,v_3)^{\mathsf T}$
and put $(X,Y)=(v_1v_2,v_1^2v_3)$. Recall
\[
 p(w)=w^3-6w^2+6w,\qquad
 q(w)=\tfrac38w^4-2w^3+\tfrac32w^2,
\]
and let $\mathcal K_G$ be the curve parametrized by $(p,q)$.
The relevant tangency polynomial is
\[
 P_G(w)=\tfrac14w^4-2w^3+3w^2-Xw+2Y.
\]

\begin{theorem}[Fiber sizes of $G$]\label{thm:fibG}
On $v_1\ne0$, the fiber sizes are determined by the shadow target
$(X,Y)=(v_1v_2,v_1^2v_3)$ as follows:
\[
\begin{array}{c|c|c}
 \text{position of }(X,Y)&\text{root multiplicities of }P_G&\#G^{-1}(v_1,v_2,v_3)\\ \hline
 (X,Y)\notin\mathcal K_G&1+1+1+1&4\\
 \text{smooth point of }\mathcal K_G&2+1+1&2\\
 \text{either cusp of }\mathcal K_G&3+1&1\\
 \text{the node of }\mathcal K_G&2+2&0
\end{array}
\]
The cusps are $(-4\mp4\sqrt2,-11/2\mp4\sqrt2)$, and the node is
$(-4,1/2)$. On $v_1=0$, the fiber has three points if $v_2^2\ne24v_3$ and
one point if $v_2^2=24v_3$. Thus the set of fiber sizes is exactly
$\{4,3,2,1,0\}$.
\end{theorem}

\begin{proof}
\emph{1. Reconstruction on $v_1\ne0$.}
The sweep equations reduce to $P_G(w)=0$ and
\[
 \gamma=\frac{X-p(w)}2=-\frac{P_G'(w)}2.
\]
For every simple root, the unique inverse is
\begin{equation}\label{eq:app-G-inverse}
 x=\frac{v_1}\gamma,\qquad y=\frac{w-\gamma}{v_1},\qquad
 z=\frac{\gamma(6\gamma-4w-\gamma^2)}{v_1^2}.
\end{equation}
These formulas invert $v_1=x\gamma$, $w=\gamma(1+xy)$, and
$\gamma=2-4xy-x^2z$. As for $F$, they give a bijection between the
simple roots of $P_G$ and the fiber. It remains to classify the root
multiplicities.

\emph{2. The possible multiplicities.}
The equations $P_G=P_G'=0$ are equivalent to
$(X,Y)=(p(w),q(w))$. Thus a repeated root occurs exactly on
$\mathcal K_G$. A root of multiplicity at least three must also satisfy
\[
 P_G''(w)=p'(w)=3(w^2-4w+2)=0,
\]
so $w=2\pm\sqrt2$. Substitution gives the two targets in the theorem.
At these parameters $P_G'''(w)=6w-12\ne0$, so the multiplicity is
exactly three. This also rules out every quadruple root.

If there are two distinct double roots, write
\[
 P_G(w)=\tfrac14(w^2-sw+r)^2.
\]
Comparison of the cubic and quadratic coefficients gives $s=4$,
$r=-2$; the remaining coefficients give $(X,Y)=(-4,1/2)$. Conversely,
at that target
\[
 P_G(w)=\tfrac14(w^2-4w-2)^2,
\]
whose two roots $2\pm\sqrt6$ are distinct. Every other repeated-root
case has one double root and two simple roots. These exhaust the
partitions of four, and counting the simple roots gives the table.

To justify the geometric names in the table, put $k=(p,q)$. The identity
$q'=wp'/2$ gives $k'=p'(1,w/2)$. At $w=2\pm\sqrt2$, one has
$k'=0$ and
\[
 \det(k'',k''')=(p'')^2\ne0,
\]
so each singularity is an ordinary cusp. At $w=2\pm\sqrt6$, one has
$p'=12\ne0$ and two distinct tangent slopes $w/2$, so their common
image is an ordinary node. There are no other affine singularities:
the parametrization is finite, since $p$ is monic, and generically
one-to-one by the double-root calculation, so it is the normalization
of its image. Away from the cases just
listed, a curve point has one preimage with nonzero derivative and
hence a single smooth local branch.

\emph{3. The boundary $v_1=0$.}
Here $x\gamma=0$. The case $x=0$ gives exactly
\[
 (x,y,z)=(0,v_2/2,2v_3-9v_2^2/4).
\]
In the disjoint case $x\ne0$, one has $\gamma=0$; the fiber equations
are equivalent to
\[
 y=\frac{v_2}6-\frac4{3x},\qquad
 (v_2^2-24v_3)x^2=4,\qquad z=\frac{2-4xy}{x^2}.
\]
Thus there are two additional points if $v_2^2\ne24v_3$ and none otherwise.
The total boundary counts are three and one. Every row of the table
occurs: for instance $(X,Y)=(0,1)$ lies off the curve, $(0,0)$ is a
smooth point, and the cusp and node targets were given explicitly.
Each shadow target is realized by taking $v_1=1$. The boundary contributes
the remaining size three, proving the last assertion.
\end{proof}

\paragraph{Elimination identities.}
The preceding proof does not require a graph basis or a discriminant
expansion. For comparison with elimination calculations, define
\begin{align*}
 E_G(X,Y)={}&27X^4+80X^3-672X^2Y+1536XY^2-512Y^3\\
 &-144X^2-2304XY+4608Y^2+3456Y.
\end{align*}
Exact calculation gives
\[
 \operatorname{Res}_w(P_G,X-p)=-E_G/64,\qquad
 \operatorname{disc}_w(P_G)=-E_G/16,
\]
so $\mathcal K_G=\{E_G=0\}$. Put $D(v_1,v_2,v_3)=E_G(v_1v_2,v_1^2v_3)/v_1^2$;
this quotient is the polynomial
\begin{align*}
 D={}&27v_1^2v_2^4+80v_1v_2^3-672v_1^2v_2^2v_3+1536v_1^3v_2v_3^2-512v_1^4v_3^3\\
 &-144v_2^2-2304v_1v_2v_3+4608v_1^2v_3^2+3456v_3.
\end{align*}
A primitive graph relation in the source coordinate $x$ is
\begin{align*}
 \widehat R_G(x;v_1,v_2,v_3)={}&D(v_1,v_2,v_3)x^4
 -24(3v_1^2v_2^2+8v_1v_2+16v_1^2v_3-24)x^2\\
 &-64v_1(v_1v_2+4)x-16v_1^2.
\end{align*}
It is verified by substituting $x=v_1/\gamma$ and reducing modulo $P_G$,
using $\gamma=(v_1v_2-p(w))/2$. This proves the identity on the dense
chart $v_1\ne0$, and hence everywhere. On the boundary it specializes to
\[
 \widehat R_G(x;0,v_2,v_3)=144x^2\bigl(4-(v_2^2-24v_3)x^2\bigr).
\]
The factor $x^2$ represents the projection of the unique boundary point
with $x=0$; it does not represent two source points. Thus even this
primitive relation must be combined with reconstruction to count fibers.

\subsection{The maps \texorpdfstring{$F_4$--$F_7$}{F4--F7}}\label{app:F4}
The fiber structures of $F_4,F_5,F_6,F_7$ are not completely classified
here. Theorems~\ref{thm:F4}--\ref{thm:F7} determine their geometric
degrees $5,10,6,12$ using elimination polynomials and exact lifting
conditions. Part~\ref{part:equivalence} studies the particular omitted
loci and finite-cover strata needed for the inequivalence results.

A complete classification would require two further steps for each map:
classifying all root-multiplicity patterns and the corresponding lifts on
the nonzero-coordinate chart, and solving all fibers on its boundary.
Lemma~\ref{lem:twist-fibers} supplies the lifting criterion on the chart.
The degree of an elimination polynomial counts roots with multiplicity;
it does not by itself give the number of distinct points in a fiber.
The analyses above for $F$ and $G$ illustrate both steps.

\clearpage
\section{Verification and reproducibility}\label{app:verification}
The compact definitions and exact-arithmetic verification scripts accompany
this revision. The calculations establish polynomial identities and exact
certificates; the geometric proofs explain how these certificates imply
fiber and product-structure conclusions. Expanded component lists and
large graph bases are not needed for those proofs.

Each script is self-contained: its formulas and coefficients are explicitly
transcribed from this paper, and it reads no manuscript, companion script,
or other local data file. Each may be run separately with Python~3
(without the assertion-disabling option \texttt{-O}). All four require SymPy;
\texttt{tangent\_sweeps\_independent\_checks.py} additionally requires
\texttt{python-flint}. The checks were run with Python~3.11.4,
SymPy~1.14.0, and \texttt{python-flint}~0.9.0.

The companion files have the following uses:
\begin{itemize}
\item \texttt{tangent\_sweeps\_inequivalence\_checks.py} checks the
elimination identities \eqref{prod:eq:deriv4} and \eqref{prod:eq:deriv6},
the quintic discriminant and coefficients, the sextic square and obstruction
to a cube, the boundary derivative matrix, and the seed identities for $F$ and $G$.
\item \texttt{tangent\_sweeps\_independent\_checks.py} checks polynomiality,
constant Jacobians, and the $F_4,F_5$ expansions, by comparing the compact
constructions with embedded copies of the printed coefficients.
\item \texttt{tangent\_sweeps\_revision\_checks.py} checks the plane-fiber
identities, reconstruction and boundary formulas, cusp and node calculations,
projection criterion, and corrected coordinate-order sign used in the main
article.
\item \texttt{tangent\_sweeps\_remaining\_checks.py} checks the $F_5$
resultant and unique one-point fiber, every entry of the $F_7$ certificate,
and the polynomial source and target transformations that split $F_6$.
\end{itemize}
The files use exact symbolic, polynomial, and finite-field arithmetic;
no numerical root-finding or random tests are used. These checks certify the
displayed arithmetic; the root-multiplicity classification, closedness,
reconstruction, and inequivalence arguments are given in the text.
The finite-field calculation is an exact premise of the Hensel proof,
not numerical evidence for a complex solution. The $F_6$ verification uses
the exact stage and sweep identities and polynomial inverse formulas,
rather than testing finitely many input points.

\end{document}